\documentclass[pdflatex,sn-mathphys-num]{sn-jnl}% Math and Physical Sciences Numbered Reference Style
\usepackage{physics}% additional package added by RSL
\usepackage{graphicx}%
\usepackage{multirow}%
\usepackage{amsmath,amssymb,amsfonts}%
\usepackage{amsthm}%
\usepackage{mathrsfs}%
\usepackage[title]{appendix}%
\usepackage{xcolor}%
\usepackage{textcomp}%
\usepackage{manyfoot}%
\usepackage{booktabs}%
\usepackage{algorithm}%
\usepackage{algorithmicx}%
\usepackage{algpseudocode}%
\usepackage{listings}%
\theoremstyle{thmstyleone}%
\newtheorem{theorem}{Theorem}%  meant for continuous numbers
\newtheorem{lemma}[theorem]{Lemma}%

\theoremstyle{thmstyletwo}%
\newtheorem{remark}{Remark}%

\theoremstyle{thmstylethree}%

\begin{document}

\title[Lipolysis on Lipid Droplets: Modelling and Numerics]{Lipolysis on Lipid Droplets: Mathematical Modelling and Numerical Discretisations}

%%=============================================================%%
%% GivenName	-> \fnm{Joergen W.}
%% Particle	-> \spfx{van der} -> surname prefix
%% FamilyName	-> \sur{Ploeg}
%% Suffix	-> \sfx{IV}
%% \author*[1,2]{\fnm{Joergen W.} \spfx{van der} \sur{Ploeg}
%%  \sfx{IV}}\email{iauthor@gmail.com}
%%=============================================================%%

\author[1]{\fnm{Thomas} \sur{Apel}}\email{thomas.apel@unibw.de}

\author[2]{\fnm{Klemens} \sur{Fellner}}\email{klemens.fellner@uni-graz.at}

\author[1]{\fnm{Volker} \sur{Kempf}}\email{volker.kempf@unibw.de}

\author*[2]{\fnm{Reymart} \sur{Salcedo-Lagunero}}\email{reymart-salcedo.lagunero@uni-graz.at}

\author[1]{\fnm{Philipp} \sur{Zilk}}\email{philipp.zilk@unibw.de}

\affil[1]{\orgdiv{Institute for Mathematics and Computer-Based Simulation}, \orgname{University of the Bundeswehr Munich}, \orgaddress{\street{Werner-Heisenberg-Weg 39}, \city{Munich}, \postcode{85577}, \state{Bavaria}, \country{Germany}}}

\affil*[2]{\orgdiv{Department of Mathematics and Scientific Computing}, \orgname{University of Graz}, \orgaddress{\street{Heinrichstrasse 36}, \city{Graz}, \postcode{8010}, \state{Styria}, \country{Austria}}}

%%==================================%%
%% Sample for unstructured abstract %%
%%==================================%%

\abstract{
	Lipolysis is a life-essential metabolic process, which supplies fatty acids stored in lipid droplets to the body in order to match the demands of building new cells and providing cellular energy.

	In this paper, we present a first mathematical modelling approach for lipolysis, which takes into account that the involved enzymes act on the surface of lipid droplets. We postulate an active region near the surface where the substrates are within reach of the surface-bound enzymes and formulate a system of reaction-diffusion PDEs, which connect the active region to the inner core of lipid droplets via interface conditions.

	We establish two numerical discretisations based on finite element method and isogeometric analysis, and validate them to perform reliably. Since numerical tests are best performed on non-zero explicit stationary state solutions, we introduce and analyse a model, which describes besides lipolysis also a reverse process (yet in a physiologically much oversimplified way). The system is not coercive such that establishing well-posedness is a non-standard task. We prove the unique existence of global and equilibrium solutions. We establish exponential convergence to the equilibrium solutions using the entropy method. We then study the stationary state model and compute explicitly for radially symmetric solutions. Concerning the finite element methods, we show numerically the linear and quadratic convergence of the errors with respect to the $H^{1}$- and $L^{2}$-norms, respectively.

	Finally, we present numerical simulations of a prototypical PDE model of lipolysis and illustrate that enzyme clustering on lipid droplets can significantly slow down lipolysis.
	}



\maketitle

\section{Introduction}\label{sec1}

Fatty acids (FAs) are crucial for the production of adenosine triphosphates (ATP), synthesis of biological membranes, thermogenesis, and signal transduction \cite{Glatz-2014,Zechner-2012}. Animals and humans store FAs in the form of water-insoluble triglycerides (TGs) within cytosolic lipid droplets (LDs) of specialised fat cells called adipocytes, but also in other cell types. From a biochemical perspective, a TG molecule is composed of three FAs esterified to a glycerol (GL) backbone.

Lipolysis, also known as lipid hydrolysis, is the splitting of ester bonds in TGs by the addition of water releasing the fundamental components of three FAs and the glycerol backbone \cite{schweiger2014}. Lipolysis is an essential metabolic process as the body requires FAs for the build-up of the phospholipid bilayer surrounding every cell and many other purposes. During times of starvation or high energy consuming activities, our body processes TGs stored in LDs and the released FAs are used for ATP production, which provides cellular energy in the body.

In white adipose tissues, lipolysis involves a coordinated process catalysed by three enzymes \cite{zimmermann1383}. Figure~\ref{fig:lipolysis} provides a simplified schematic diagram: lipolysis is regulated by the activation of the enzyme adipose triglyceride lipase (ATGL) via the regulatory protein comparative gene identification-58 (CGI-58). The enzyme ATGL hydrolyses TGs producing diglycerides (DGs, a glycerol backbone with two attached FAs) and a first free FA. Then, another enzyme called hormone-sensitive lipase (HSL) hydrolyses DGs producing monoglycerides (MGs, a glycerol backbone with one attached FA) and a second free FA. Finally, the enzyme monoglyceride lipase (MGL) breaks down MGs to release the glycerol backbone and a third free FA.
\begin{figure}[htbp]
	\centering
	\includegraphics[scale=1.1]{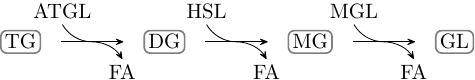}
	\caption{Lipolysis is a three-step process: ATGL hydrolyses TGs producing DGs and a first free FA; HSL hydrolyses DGs producing MGs and a second free FA; finally, MGL hydrolyses MGs releasing the backbone glycerol and a third free FA}
	\label{fig:lipolysis}
\end{figure}

The scheme depicted in Figure~\ref{fig:lipolysis} is a simplification of the \textit{in vivo} situation: Firstly, other enzymes are known to degrade TGs, but in white adipose tissues, those are believed to be negligible in a first approach. Secondly, ATGL is capable to also hydrolyse DGs (besides TGs), yet significantly less efficiently than HSL. Similarly, HSL is able to also hydrolyse TGs (besides DGs), but much less prominently than ATGL, which is believed to be the key regulator of lipolysis.

Thirdly, ATGL is a remarkable enzyme in the sense that besides its hydrolytic activities with substrate TGs (and DGs, which we neglect), it catalyses a transacylation reaction which transfers an FA from one DG to another DG as shown in Figure~\ref{fig:transacylation}.
\begin{figure}[htbp]
	\centering
	\includegraphics[scale=1.1]{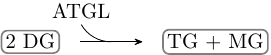}
	\caption{In parallel to TG hydrolysis, ATGL also catalyses the transacylation reaction of two DG molecules into one TG and one MG}
	\label{fig:transacylation}
\end{figure}

The mechanism of lipolysis represented by Figure~\ref{fig:lipolysis} has been first established in 2004 \cite{zimmermann1383}. Previous mathematical models for lipolysis \cite{Kim-2008,Loevfors-2021,elias2023}, were formulated in terms of (spatially homogeneous) rate equations as part of system biology approaches. Recently, another ODE rate equations showed that transacylation can play a significant role in the lipolytic cascade \cite{elias2023}: On one hand, transacylation is a partial feedback mechanism forming TGs from DGs and is therefore able to slow down the lipolytic machinery. On the other hand, transacylation forms also MGs leading (under some conditions) to an increased production of downstream products by up to 80\% \cite{elias2023}. Transacylation involves two DGs and its reaction rate is therefore a quadratic function in terms of the DG concentration. This is a significant difference to DG hydrolysis, which is modelled as Michaelis-Menten process with a linearly bounded reaction rate function. Therefore, the larger the DG concentration during lipolysis, the stronger  will be the effects of DG transacylation.

Lipolysis is localised at lipid droplets, consisting of a core made of lipids and a layer of phospholipids, which connects the lipids to the surrounding aquatic environment, as illustrated on Figure~\ref{fig:lipiddroplet}b. As an example, Figure~\ref{fig:lipiddroplet}a shows a cytoplasmic image of a lipid droplet.

\begin{figure}[htb]
	\centering
	(a)\includegraphics[width=0.45\linewidth]{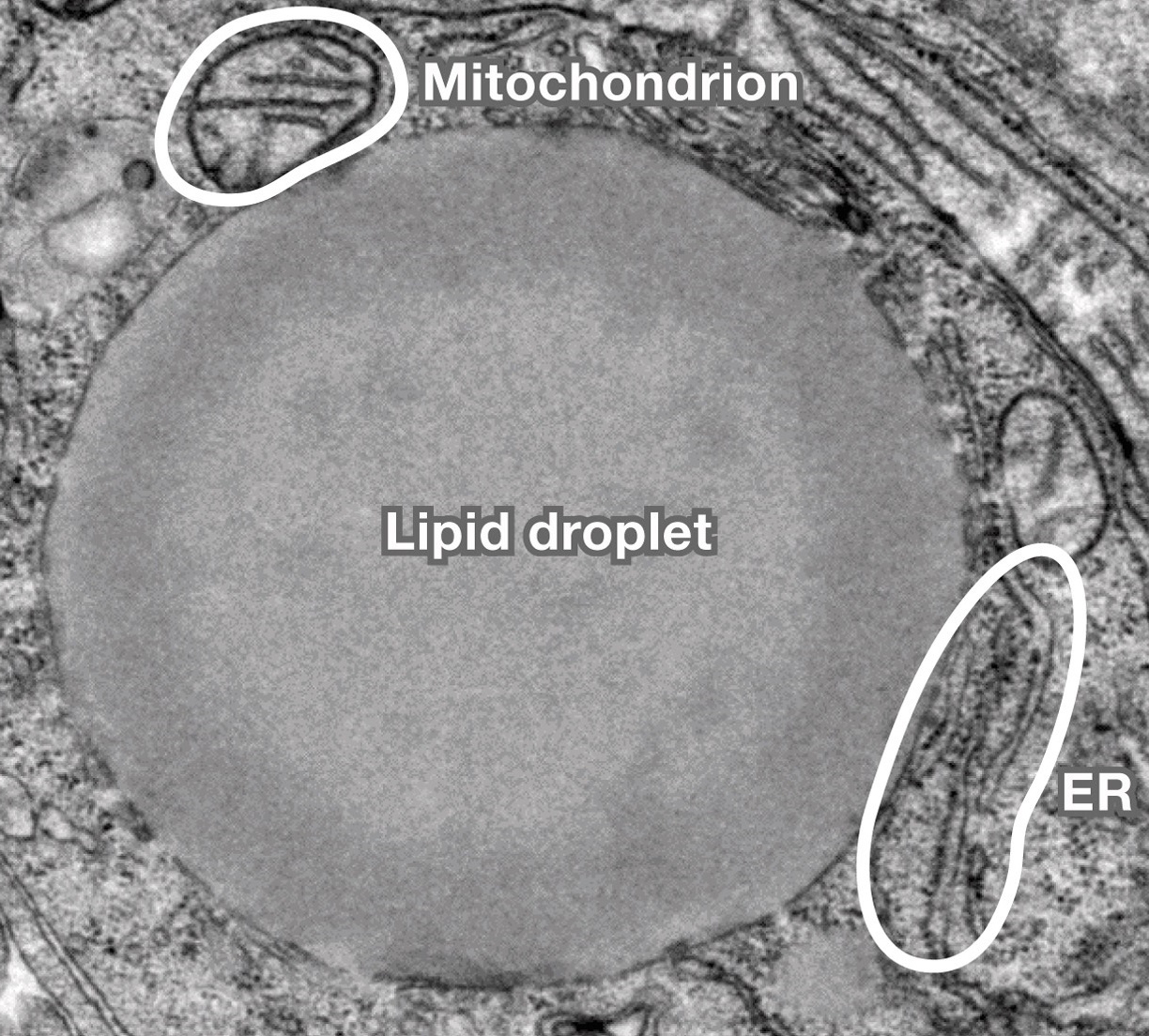}(b)\includegraphics[width=0.45\linewidth]{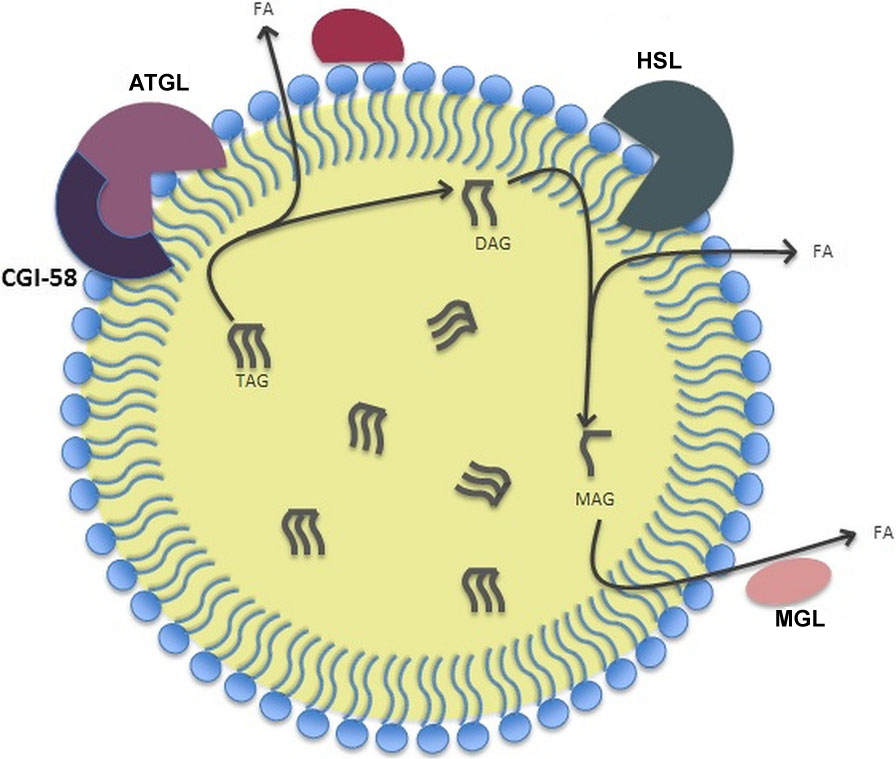}
	\caption{(a) A cytoplasmic lipid droplet in a cultured hepatoma cell adopted from \protect\cite{farese2009} (b) Structural components in a lipid droplet adopted from \protect\cite{onal2017} and modified to match notations. The lipid droplet is composed of several components; the triglycerides are found inside the droplet and are surrounded by a layer of phospholipids}
	\label{fig:lipiddroplet}
\end{figure}

All previous rate equation models \cite{elias2023,Kim-2008,Loevfors-2021} have the crucial drawback of not taking into account that TGs are stored in LDs and that ATGL and HSL act on or near the surface.

This article presents a first PDE modelling approach describing lipolysis on LDs based on reaction-diffusion equations. In good approximation of many \textit{in vivo} situations, we model an LD as spherical in shape and refer to it as a domain $\Omega \subset \mathbb{R}^{3}$ for the rest of this text. The enzymes which regulate lipolysis are localised on the surface of the LD. Accordingly, we postulate a thin active region near the surface where the substrates (i.e. TG, DG, MG) can come in contact with the enzymes (i.e. ATGL, HSL, MGL). TGs that are sufficiently far from the surface of a LD are no immediate substrate, therefore substantially reducing the substrate availability compared to the models using  ODE rate equations, which implicitly assume that enzymes and substrates are part of a three-dimensional solution.

\begin{figure}[htb]
	\centering
	(a)\includegraphics[scale=1]{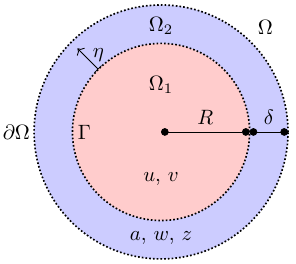}(b)\hspace{0.1in}\includegraphics[scale=1.3]{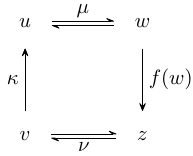}
	\caption{(a) The reservoir region $\Omega_{1}$ is the ball with radius $R > 0$ and outer unit normal $\eta$. The active region $\Omega_{2}$ is the annulus with thickness $\delta > 0$. The boundary of the reservoir region coincides with the \textit{interface} $\Gamma$ towards the active region. The outer boundary of the active region is the boundary $\partial\Omega$ (b) A chemical reaction network for TGs and DGs in the reservoir and active regions. The constants $\mu, \, \nu > 0$ denote the interface flux rate for TG and DG, $\kappa > 0$ models the feedback rate of DG into TG at the reservoir region, and $f$ is the Michaelis-Menten reaction term which governs the hydrolytic reaction between TG and ATGL}
	\label{fig:chemfig}
\end{figure}

Accordingly, we divide $\Omega$ into a \textit{reservoir region} denoted by $\Omega_{1} = B(0,R) \subset \Omega$, a ball of (non-dimensional) radius $R > 0$, and the \textit{active region} denoted by $\Omega_{2} = B(0,R+\delta) \setminus \overline{B(0,R)} \subset \Omega$, a spherical shell with thickness $\delta > 0$. The active region, where the substrate-enzyme reactions occur, is a modelling choice which describes that the enzymes are localised at the surface of the lipid droplet. One can think of the thickness $\delta$ as a molecular reach of the enzymes, which is independent of the overall size of the LD. The domain $\Omega$ in 2-D is illustrated on Figure~\ref{fig:chemfig}a.

A main part of this paper establishes a robust numerical scheme which is able to reliably simulate the dynamics of lipolysis on realistically sized LDs, for which the non-dimensional ratio of $\delta / R$ might vary between $10^{-1}$ and $10^{-3}$. In order to trustfully simulate LDs with such thin active regions, we present two different finite element implementations: a first based on \verb|FEniCS| \cite{logg2012} with the advantages of a more direct implementation but at the cost of discretisation issues based on the nonconvex active region (see Remark~\ref{rem:VarCrim} below), and a second discretisation based on \verb|GeoPDEs| \cite{vazquez2016,defalco2011,hughes2009} using isogeometric analysis, which can exactly represent conical geometries.

For the verification of the numerical schemes, we would like to use an explicit non-trivial stationary state solution as test case. However, the stationary state solution to the lipolytic scheme in Figure~\ref{fig:lipolysis} and Figure~\ref{fig:transacylation} is zero for the concentrations of  TG, DG, and MG as they were all hydrolysed into GL and FAs.

Therefore, the first part of this paper discusses a testing model system with a non-trivial positive equilibrium state, which we use for numerical testing purposes. To achieve this, we assume that besides TGs producing DGs and FAs in the active region, the reverse process of DGs forming TGs by the uptake of an FA happens in the reservoir region. This is a much oversimplified physiological shortcut as the build up of TGs involves another set of enzymes (e.g. diacylglycerol O-acyltransferase (DGAT) 1 and 2) and requires the activation of FAs by the coenzyme A (CoA). Also, the build-up of TGs does not happen inside the LD. Finally, for the sake of simplicity, we will not consider MGs and GLs as part the testing model, as they constitute  downstream products and can be post-calculated from knowing DGs.

We denote by $u := u(x,t)$ and $v := v(x,t)$ the (non-dimensional) concentrations of TGs and DGs in the reservoir region, respectively. Analogously, $w := w(x,t)$ and $z := z(x,t)$ for TGs and DGs in the active region, respectively. The active region is characterised by the presence of ATGL and we denote its concentration by $a_{3} := a_{3}(x,t)$. We then consider the chemical network shown on Figure~\ref{fig:chemfig}b.

The corresponding reaction-diffusion-interface model reads for arbitrary (non-dimensional) $T>0$
\begin{equation} \label{eqn:ParabolicModel}
	\begin{split}
		&\begin{cases}
			u_{t} - d_{3} \Delta u = \kappa v, & \quad \ \ (x,t) \in \Omega_{1} \times (0,T], \\
			d_{3} \partial_{\eta} u = \mu[w - u], & \quad \ \ (x,t) \in \Gamma \times (0,T],\\
			v_{t} - d_{2} \Delta v = -\kappa v, & \quad \ \ (x,t) \in \Omega_{1} \times (0,T],\\
			d_{2} \partial_{\eta} v = \nu[z - v], & \quad \ \ (x,t) \in \Gamma \times (0,T],\\
		\end{cases} \\
		&\begin{cases}
			w_{t} - d_{3} \Delta w = - f(w), & (x,t) \in \Omega_{2} \times (0,T],\\
			-d_{3} \partial_{\eta} w = -\mu[w - u], & (x,t) \in \Gamma \times (0,T], \\
			d_{3} \partial_{\eta} w = 0, & (x,t) \in \partial\Omega \times (0,T],\\
			z_{t}- d_{2} \Delta z = f(w), & (x,t) \in \Omega_{2} \times (0,T],\\
			-d_{2} \partial_{\eta} z = -\nu[z - v], & (x,t) \in \Gamma \times (0,T], \\
			d_{2} \partial_{\eta} z = 0, & (x,t) \in \partial\Omega \times (0,T],
		\end{cases}
	\end{split}
\end{equation}
where $d_{2},\ d_{3} > 0$ are constant diffusion coefficients, $\mu, \ \nu > 0$ are constant flux rates at the interface, $\kappa > 0$ is a constant feedback rate of DG into TG in the reservoir region, and
\begin{equation*}
	f(w) = \dfrac{v_{3} a_{3} w}{k_{3} + w} \qquad\text{with}\qquad v_{3}, \ k_{3} > 0,
\end{equation*}
is a Michaelis-Menten substrate-enzyme reaction rate describing the action of ATGL, $v_{3}$ is the maximum reaction velocity and $k_{3}$ is the Michaelis constant. Note that the homogeneous boundary conditions at $\partial\Omega$ are realistic since TGs and DGs are supposed to be too hydrophobic to leave the LDs.

System \eqref{eqn:ParabolicModel} is coupled with the initial data
\begin{equation} \label{eqn:ParabolicInitialData}
	\begin{cases}
		u(x,0) = u_{0}(x), \quad \ v(x,0) = v_{0}(x), \quad x \in \Omega_{1}, \\
		w(x,0) = w_{0}(x), \quad z(x,0) = z_{0}(x), \quad x \in \Omega_{2}.
	\end{cases}
\end{equation}
By adding the PDEs in system \eqref{eqn:ParabolicModel} and  using integration by parts, we remark that any solution to system \eqref{eqn:ParabolicModel} conserves the total mass of glycerol contained in TGs and DGs for all time $t \geq 0$,
\begin{equation} \label{eqn:MassCon}
	\int_{\Omega_{1}} [u(x,t) + v(x,t)] \dd{x} + \int_{\Omega_{2}} [w(x,t) + z(x,t)] \dd{x} = M_{0},
\end{equation}
where $M_{0} = \int_{\Omega_{1}} [u_{0}(x) + v_{0}(x)] \dd{x} + \int_{\Omega_{2}} [w_{0}(x) + z_{0}(x)] \dd{x} > 0$ is the given initial total mass.

Since the Michaelis-Menten kinetics does not allow for explicit stationary state solutions, we let the ATGL concentration $a_{3} > 0$ to be constant and consider its linear approximation,
\begin{equation} \label{eqn:LinearParabolic}
	f(w) = \rho w \qq{where} \rho = \frac{v_{3}}{k_{3}} a_{3} > 0 \qq{is constant,}
\end{equation}
which holds in the well-known Michaelis-Menten regime for small substrate concentrations, that is, when $w \ll k_{3}$. This is a plausible assumption in the active region where ATGL and phospholipids are expected to leave limited volume for TGs and DGs. We repeat that the reverse process $v \xrightarrow{\kappa} u$ in the reservoir region is an oversimplified model for the build-up of TGs from DGs and chosen linear for the sake of constructing a non-negative explicit stationary state solution for numerical testing purposes. For this reason, unless otherwise explicitly stated, we therefore refer to system \eqref{eqn:ParabolicModel} as the linear parabolic system where $f$ is as defined in \eqref{eqn:LinearParabolic}.

\subsubsection*{Overview}

Section~\ref{sec:math_prelim} contains the necessary mathematical preliminaries and the precise problem statement.

In Section~\ref{sec:wellposedness}, we study the linear parabolic system \eqref{eqn:ParabolicModel} and its corresponding linear stationary state (elliptic) system. We prove existence and uniqueness of solution for the parabolic and elliptic systems by using fixed point arguments following ideas of, e.g., \cite{egger2018}. Technical challenges are caused by the complex-balanced structure of the chemical reaction network on Figure~\ref{fig:chemfig}b. Moreover, we utilise the entropy method to prove the exponential convergence of the solution of system \eqref{eqn:ParabolicModel} to the complex-balanced equilibrium.  The main results are stated in Theorem~\ref{thm:WellposednessParabolicSystem}, Theorem~\ref{thm:WellposenessElliptic}, and Theorem~\ref{thm:ExponentialConvergence}.

Section~\ref{sec:numerical_analysis} contains results on the numerical schemes and simulations. Here, we compute the explicit radially symmetric stationary state solutions and use a finite element method to approximate the numerical solutions of the parabolic and elliptic problems. We demonstrate that the numerical schemes feature the desired convergence rates of discretisation errors.

In Section~\ref{sec:clustering}, we use numerical simulations to study the potential effects of ATGL clustering. Indeed, ATGL is known to be able to cluster and experimental data suggest that the distribution of ATGL over the surface of LDs can be quite heterogeneous. While ATGL heterogeneities cannot be addressed in the previous ODE models, we study a prototypical PDE model of lipolysis according to the processes depicted in Figure~\ref{fig:lipolysis} and Figure~\ref{fig:transacylation} and show simulations that ATGL clustering can significantly delay downstream metabolites compared to the same amount of unclustered ATGL.

Finally, we provide a summary and conclusions in Section~\ref{sec:conclusions}.

\section{Mathematical Preliminaries}\label{sec:math_prelim}

Let us start by introducing relevant mathematical notations and the problem statement. Let $\Omega \subset \mathbb{R}^{d}$ be a bounded Lipschitz domain. Let $\Omega_{1} \subset \Omega$ and define $\Omega_{2} = \Omega \setminus \overline{\Omega}_{1}$ such that $\Omega = \Omega_{1} \cup \Omega_{2}$ and $\Omega_{1} \cap \Omega_{2} = \emptyset$. Assume that $\Omega_{1}$ and $\Omega_{2}$ are bounded Lipschitz domains. Denote by $\Gamma := \overline{\Omega}_{1} \cap \overline{\Omega}_{2} \subset \Omega$ the interface which is assumed to be in the interior of $\Omega$. Define the Lebesgue product space $\mathbf{L}^{p}$ as $\mathbf{L}^{p} := L^{p}(\Omega_{1}) \times L^{p}(\Omega_{1}) \times L^{p}(\Omega_{2}) \times L^{p}(\Omega_{2})$ where $p \in [1, \infty)$, equipped with the norm
\[ \norm{\Phi}_{\mathbf{L}^{p}} = \qty(\norm{\varphi_{1}}_{L^{p}(\Omega_{1})}^{p} + \norm{\varphi_{2}}_{L^{p}(\Omega_{1})}^{p} + \norm{\varphi_{3}}_{L^{p}(\Omega_{2})}^{p} + \norm{\varphi_{4}}_{L^{p}(\Omega_{2})}^{p})^{\frac{1}{p}}, \]
for all $\Phi := (\varphi_{1}, \varphi_{2}, \varphi_{3}, \varphi_{4}) \in \mathbf{L}^{p}$. The Sobolev product space $\mathbf{H}^{m}$ is likewise defined as $\mathbf{H}^{m} := H^{m}(\Omega_{1}) \times H^{m}(\Omega_{1}) \times H^{m}(\Omega_{2}) \times H^{m}(\Omega_{2})$ where $m \geq 0$, equipped with the norm
\[ \norm{\Psi}_{\mathbf{H}^{m}} = \qty(\norm{\psi_{1}}_{H^{m}(\Omega_{1})}^{2} + \norm{\psi_{2}}_{H^{m}(\Omega_{1})}^{2} + \norm{\psi_{3}}_{H^{m}(\Omega_{2})}^{2} + \norm{\psi_{4}}_{H^{m}(\Omega_{2})}^{2})^{\frac{1}{2}}, \]
for all $\Psi := (\psi_{1}, \psi_{2}, \psi_{3}, \psi_{4}) \in \mathbf{H}^{m}$. Note that $\mathbf{H}^{0} = \mathbf{L}^{2}$ by convention. The classical Sobolev and Lebesgue spaces in $U$ are denoted by $H^{m}(U)$ and $L^{p}(U)$, respectively, where $U$ is either $\Omega_{1}$, $\Omega_{2}$, or $\Gamma$. The inner product in $L^{2}(U)$ is denoted by $(\cdot, \cdot)_{U}$. Finally, the dual space of $H^{m}(U)$ is denoted by $H^{m}(U)^{*}$.

We now consider the following definition of weak solutions: let $\mathbf{c} := (u,v,w,z)$ denote a tuple of functions where
\begin{align*}
	u, v \in L^{2}(0,T; H^{1}(\Omega_{1})) \cap H^{1}(0,T; H^{1}(\Omega_{1})^{*}), \\
	w, z \in L^{2}(0,T; H^{1}(\Omega_{2})) \cap H^{1}(0,T; H^{1}(\Omega_{2})^{*}).
\end{align*}
Then, a weak solution of the linear parabolic system \eqref{eqn:ParabolicModel} satisfies the weak formulation
\begin{equation} \label{eqn:VarEqn}
	\mathcal{B}(\dot{\mathbf{c}}, \Phi) + \mathcal{A}(\mathbf{c}, \Phi) = 0 \qq{for all test functions} \Phi \in \mathbf{H}^{1},
\end{equation}
where the linear form $\mathcal{B}$ is given by
\begin{equation} \label{eqn:BilinearForm1}
	\mathcal{B}(\dot{\mathbf{c}}, \Phi) := (\dot{u}, \varphi_{1})_{\Omega_{1}} + (\dot{v}, \varphi_{2})_{\Omega_{1}} + (\dot{w}, \varphi_{3})_{\Omega_{2}} + (\dot{z}, \varphi_{4})_{\Omega_{2}},
\end{equation}
the bilinear form $\mathcal{A}$ is given by
\begin{align}
	\mathcal{A}(\mathbf{c}, \Phi) &:= d_{3} (\nabla u, \nabla \varphi_{1})_{\Omega_{1}} + d_{2} (\nabla v, \nabla \varphi_{2})_{\Omega_{1}} + d_{3} (\nabla w, \nabla \varphi_{3})_{\Omega_{2}} \notag \\
	&+ d_{2} (\nabla z, \nabla \varphi_{4})_{\Omega_{2}} + \kappa (v, \varphi_{2} - \varphi_{1})_{\Omega_{1}}  + \rho (w, \varphi_{3} - \varphi_{4})_{\Omega_{2}} \notag \\
	&+ \mu(w - u, \varphi_{3} - \varphi_{1})_{\Gamma} + \nu(z - v, \varphi_{4} - \varphi_{2})_{\Gamma}, \label{eqn:BilinearForm2}
\end{align}
$\dot{\mathbf{c}} := (\dot{u}, \dot{v}, \dot{w}, \dot{z})$ denotes time derivatives, and in addition, the initial conditions in \eqref{eqn:ParabolicInitialData} and the total mass conservation law in \eqref{eqn:MassCon} hold.

\section{Well-posedness} \label{sec:wellposedness}

In this section, we discuss the existence, uniqueness, and continuous dependence on the initial datum of the weak solution of the linear parabolic system \eqref{eqn:ParabolicModel}.

\begin{lemma}[Continuity of $\mathcal{A}$]\label{lem:continuity}
	There exists a constant $K > 0$ depending only on the diffusion coefficients, reaction rates, flux constants, and the domains $\Omega_{1}$, $\Omega_{2}$, such that
	\begin{equation}\label{eqn:continuity}
		\abs{ \mathcal{A}(\mathbf{c}, \Phi) } \leq K \norm{\mathbf{c}}_{\mathbf{H}^{1}}\norm{\Phi}_{\mathbf{H}^{1}} \qq{for all} \mathbf{c}, \Phi \in \mathbf{H}^{1}.
	\end{equation}
\end{lemma}
\begin{proof}
	The proof follows from direct computations using the triangle and H\"{o}lder's inequalities, and the fact that the $L^{2}$-norm is bounded above by the $H^{1}$-norm, e.g.
	\[ |d_{3} (\nabla u, \nabla \varphi_{1})_{\Omega_{1}}| \leq d_{3} \norm{u}_{H^{1}(\Omega_{1})}\norm{\varphi_{1}}_{H^{1}(\Omega_{1})} \leq d_{3}\norm{\mathbf{c}}_{\mathbf{H}^{1}}\norm{\Phi}_{\mathbf{H}^{1}}, \]
	and the Trace Theorem, see \cite[Theorem 1, Section 5.5]{evans2010}, e.g.
	\begin{align*}
		&|\mu(w - u, \varphi_{3} - \varphi_{1})_{\Gamma}| \\
		&\qquad \leq K_{1}\mu \qty(\norm{w}_{H^{1}(\Omega_{2})} + \norm{u}_{H^{1}(\Omega_{1})})\qty(\norm{\varphi_{1}}_{H^{1}(\Omega_{1})} +\norm{\varphi_{3}}_{H^{1}(\Omega_{2})}),
	\end{align*}
	for some constant $K_{1} > 0$. Hence, we obtain the constant $K > 0$ such that estimate \eqref{eqn:continuity} holds.
\end{proof}

\begin{lemma}[G\r{a}rding Inequality]\label{lem:GardingIneq}
	There exist constants $\alpha > 0$ and $\beta > 0$ such that for all $t \in [0, T]$, we have
	\begin{equation}\label{eqn:GardingIneq}
		\mathcal{A}(\mathbf{c}, \mathbf{c}) + \alpha \norm{\mathbf{c}}_{\mathbf{L}^{2}}^{2} \geq \beta \norm{\mathbf{c}}_{\mathbf{H}^{1}}^{2}.
	\end{equation}
	One admissible choice is $\alpha = 1 + \max \qty{ \frac{\kappa}{2}, \frac{\rho }{2} } > 0$ and $\beta = \min \left\{ 1, d_{2}, d_{3}, \kappa ,\rho \right\} > 0$.
\end{lemma}
\begin{proof}
	By Cauchy-Schwarz and Young's inequalities, we obtain
	\begin{align*}
		&\mathcal{A}(\mathbf{c}, \mathbf{c}) \geq d_{3} \norm{\nabla u}_{L^{2}(\Omega_{1})}^{2} + d_{2} \norm{\nabla v}_{L^{2}(\Omega_{1})}^{2} + d_{3} \norm{\nabla w}_{L^{2}(\Omega_{2})}^{2} + d_{2} \norm{\nabla z}_{L^{2}(\Omega_{2})}^{2} \\
		&\quad + \norm{u}_{L^{2}(\Omega_{1})}^{2} + \kappa\norm{v}_{L^{2}(\Omega_{1})}^{2} + \rho \norm{w}_{L^{2}(\Omega_{2})}^{2} + \norm{z}_{L^{2}(\Omega_{2})}^{2}  \\
		&\quad - \qty( 1 + \frac{\kappa}{2} )\norm{u}_{L^{2}(\Omega_{1})}^{2} - \frac{\kappa}{2} \norm{v}_{L^{2}(\Omega_{1})}^{2} - \frac{\rho }{2} \norm{w}_{L^{2}(\Omega_{2})}^{2} - \qty( 1 + \frac{\rho }{2} )\norm{z}_{L^{2}(\Omega_{2})}^{2} \\
		&\quad \geq \beta \norm{\mathbf{c}}_{\mathbf{H}^{1}}^{2} - \alpha \norm{\mathbf{c}}_{\mathbf{L}^{2}}^{2},
	\end{align*}
	where $\beta = \min \left\{ 1, d_{2}, d_{3}, \kappa ,\rho \right\} > 0$ and $\alpha = 1 + \max \qty{ \frac{\kappa}{2}, \frac{\rho }{2} } > 0$, proving inequality \eqref{eqn:GardingIneq} as desired.
\end{proof}

The following theorem guarantees the existence, uniqueness, and continuous dependence from the initial datum of the global weak solution to system \eqref{eqn:ParabolicModel}. For clarity, the dual of $\mathbf{H}^{m}$ is denoted by $(\mathbf{H}^{m})^{*} := H^{m}(\Omega_{1})^{*} \times H^{m}(\Omega_{1})^{*}  \times H^{m}(\Omega_{2})^{*}  \times H^{m}(\Omega_{2})^{*} $.
\begin{theorem}[Well-posedness of the Parabolic System]\label{thm:WellposednessParabolicSystem}
	Given any initial datum $\mathbf{c}_{0} := (u_{0}, v_{0}, w_{0}, z_{0}) \in \mathbf{L}^{2}$, the linear system \eqref{eqn:ParabolicModel} possesses a unique weak solution $\mathbf{c} = (u,v,w,z)$ such that
	\[ u, v \in L^{2}(0,T; H^{1}(\Omega_{1})) \cap H^{1}(0,T; H^{1}(\Omega_{1})^{*}), \]
	and
	\[ w, z \in L^{2}(0,T; H^{1}(\Omega_{2})) \cap H^{1}(0,T; H^{1}(\Omega_{2})^{*}), \]
	satisfying the stability estimate for all $t \in [0, T]$,
	\begin{equation} \label{eqn:WellposednessEnergyEstimate}
		\norm{\mathbf{c}(t)}_{\mathbf{L}^{2}}^{2} + \int_{0}^{t} \norm{\mathbf{c}(\tau)}_{\mathbf{H}^{1}}^{2} \dd{\tau} + \int_{0}^{t} \norm{\dot{\mathbf{c}}(\tau)}_{(\mathbf{H}^{1})^{*}}^{2} \dd{\tau} \leq C \norm{\mathbf{c}_{0}}_{\mathbf{L}^{2}}^{2},
	\end{equation}
	for some constant $C > 0$ depending only on $T$, $K$ from the continuity constant of $\mathcal{A}$ in Lemma~\ref{lem:continuity}, and $\alpha, \beta$ from the G\r{a}rding inequality in Lemma~\ref{lem:GardingIneq}. If the initial datum is non-negative, then the solution remains non-negative for all time.
\end{theorem}
\begin{proof}
	The existence and uniqueness of the global weak solution follows from Section 3, Chapter XVIII of \cite{dautray2000}. In particular, since the bilinear form $\mathcal{A}$ is only weakly coercive (that is, it only satisfies a G\r{a}rding inequality as in Lemma~\ref{lem:GardingIneq}), we introduce the following change of variables: for a finite $T > 0$, set $\widetilde{u}(t) = e^{-\gamma t}u(t)$, $\widetilde{v}(t) = e^{-\gamma t}v(t)$, $\widetilde{w}(t) = e^{-\gamma t}w(t)$, and $\widetilde{z}(t) = e^{-\gamma t}z(t)$, with $\gamma \in \mathbb{R}$ and $t \in [0, T]$. Let $\widetilde{\mathbf{c}}(t) := (\widetilde{u}(t), \widetilde{v}(t), \widetilde{w}(t), \widetilde{z}(t)) \in \mathbf{H}^{1}$. Then we have that
	\[ \dot{\widetilde{\mathbf{c}}} = e^{-\gamma t} \dot{\mathbf{c}} - \gamma \widetilde{\mathbf{c}}, \]
	and $\widetilde{\mathbf{c}}$ satisfies the variational equation
	\begin{equation} \label{eqn:TransformedVariationalEquation}
		\mathcal{B}(\dot{\widetilde{\mathbf{c}}}, \Phi) + \widetilde{\mathcal{A}}(\widetilde{\mathbf{c}}, \Phi) = 0 \qq{for all} \Phi \in \mathbf{H}^{1},
	\end{equation}
	where the bilinear form $\widetilde{\mathcal{A}}$ is defined as
	\[ \widetilde{\mathcal{A}}(\widetilde{\mathbf{c}}, \Phi) := \mathcal{A}(\widetilde{\mathbf{c}}, \Phi) + \gamma[(\widetilde{u}, \varphi_{1})_{\Omega_{1}} + (\widetilde{v}, \varphi_{2})_{\Omega_{1}} + (\widetilde{w}, \varphi_{3})_{\Omega_{2}} + (\widetilde{z}, \varphi_{4})_{\Omega_{2}}]. \]
	By choosing $\gamma = \alpha$ in Lemma~\ref{lem:GardingIneq}, the bilinear form $\widetilde{\mathcal{A}}$ is coercive:
	\[ \widetilde{\mathcal{A}}(\widetilde{\mathbf{c}}, \widetilde{\mathbf{c}}) = \mathcal{A}(\widetilde{\mathbf{c}}, \widetilde{\mathbf{c}}) + \alpha \norm{\widetilde{\mathbf{c}}}_{\mathbf{L}^{2}}^{2} \geq \beta \norm{\widetilde{\mathbf{c}}}_{\mathbf{H}^{1}}^{2}. \]
	Also, note that the new variables satisfy the given initial data in \eqref{eqn:ParabolicInitialData}, that is, $\widetilde{\mathbf{c}}(0) = \mathbf{c}_{0}$. Thus, there exists a unique solution $\widetilde{\mathbf{c}} \in \mathbf{H}^{1}$ to the variational equation \eqref{eqn:TransformedVariationalEquation}, and therefore to the original problem \eqref{eqn:VarEqn}.

	With the existence and uniqueness of solution, we now prove the stability estimate in \eqref{eqn:WellposednessEnergyEstimate}. Multiplying the PDE for $u$ with $e^{-\lambda t}u$ for some $\lambda > 0$ (to be determined later) in system \eqref{eqn:ParabolicModel} and integrating with respect to space, we get
	\begin{align}
		\frac{1}{2} & \dv{}{t} \qty(e^{-\lambda t} \norm{u(t)}_{L^{2}(\Omega_{1})}^{2}) + e^{-\lambda t}\frac{\lambda}{2}\norm{u(t)}_{L^{2}(\Omega_{1})}^{2} +  e^{-\lambda t}d_{3}\norm{\nabla u(t)}_{L^{2}(\Omega_{1})}^{2} \notag \\
		& - e^{-\lambda t} \kappa(v(t), u(t))_{\Omega_{1}} - e^{-\lambda t} \mu(w(t)-u(t), u(t))_{\Gamma} = 0. \label{eqn:est_u}
	\end{align}
	Analogously for $v$, $w$, and $z$, we obtain
	\begin{align}
		\frac{1}{2} & \dv{}{t} \qty(e^{-\lambda t} \norm{v(t)}_{L^{2}(\Omega_{1})}^{2}) + e^{-\lambda t}\frac{\lambda}{2}\norm{v(t)}_{L^{2}(\Omega_{1})}^{2} +  e^{-\lambda t}d_{3}\norm{\nabla v(t)}_{L^{2}(\Omega_{1})}^{2} \notag \\
		& + e^{-\lambda t} \kappa \norm{v(t)}_{L^{2}(\Omega_{1})}^{2} - e^{-\lambda t} \nu(z(t)-v(t), v(t))_{\Gamma} = 0, \label{eqn:est_v}
	\end{align}
	\begin{align}
		\frac{1}{2} & \dv{}{t} \qty(e^{-\lambda t} \norm{w(t)}_{L^{2}(\Omega_{2})}^{2}) + e^{-\lambda t}\frac{\lambda}{2}\norm{w(t)}_{L^{2}(\Omega_{2})}^{2} +  e^{-\lambda t}d_{2}\norm{\nabla w(t)}_{L^{2}(\Omega_{2})}^{2} \notag \\
		& + e^{-\lambda t} \rho \norm{w(t)}_{L^{2}(\Omega_{2})}^{2} + e^{-\lambda t} \mu(w(t)-u(t), w(t))_{\Gamma} = 0, \label{eqn:est_w}
	\end{align}
	\begin{align}
		\frac{1}{2} & \dv{}{t} \qty(e^{-\lambda t} \norm{z(t)}_{L^{2}(\Omega_{2})}^{2}) + e^{-\lambda t}\frac{\lambda}{2}\norm{z(t)}_{L^{2}(\Omega_{2})}^{2} +  e^{-\lambda t}d_{2}\norm{\nabla z(t)}_{L^{2}(\Omega_{2})}^{2} \notag \\
		& - e^{-\lambda t} \rho (z(t), w(t))_{\Omega_{2}} + e^{-\lambda t} \nu(z(t)-v(t), z(t))_{\Gamma} = 0. \label{eqn:est_z}
	\end{align}
	Adding Eqns. \eqref{eqn:est_u}, \eqref{eqn:est_v}, \eqref{eqn:est_w}, and \eqref{eqn:est_z}, we obtain
	\[ \frac{1}{2}\dv{}{t} \qty(e^{-\lambda t} \norm{\mathbf{c}(t)}_{\mathbf{L}^{2}}^{2}) + e^{-\lambda t}\qty( \mathcal{A}(\mathbf{c}(t), \mathbf{c}(t)) +  \frac{\lambda}{2}\norm{\mathbf{c}(t)}_{\mathbf{L}^{2}}^{2} ) = 0. \]
	Recalling the G\r{a}rding inequality \eqref{eqn:GardingIneq} and choosing $\lambda = 2\alpha$, we have
	\[ \dv{}{t} \qty(e^{-2 \alpha t} \norm{\mathbf{c}(t)}_{\mathbf{L}^{2}}^{2}) + 2\beta e^{-2\alpha t}\norm{\mathbf{c}(t)}_{\mathbf{H}^{1}}^{2} \leq 0.  \]
	Integration with respect to time over $(0,t)$ and using $\mathbf{c}(0) = \mathbf{c}_{0}$, we obtain, for every $t \in (0, T)$,
	\[ e^{-2\alpha T} \norm{\mathbf{c}(t)}_{\mathbf{L}^{2}}^{2} + 2\beta \int_{0}^{t} e^{-2\alpha \tau} \norm{\mathbf{c}(\tau)}_{\mathbf{H}^{1}}^{2} \dd{\tau} \leq \norm{\mathbf{c}_{0}}_{\mathbf{L}^{2}}^{2}. \]
	Since
	\[ e^{-2\alpha T} \int_{0}^{t} \norm{\mathbf{c}(\tau)}_{\mathbf{H}^{1}}^{2} \dd{\tau} \leq \int_{0}^{t} e^{-2\alpha \tau} \norm{\mathbf{c}(\tau)}_{\mathbf{H}^{1}}^{2} \dd{\tau}, \]
	we therefore get
	\begin{equation}\label{eqn:parabolic_est1}
		\norm{\mathbf{c}(t)}_{\mathbf{L}^{2}}^{2} + \int_{0}^{t} \norm{\mathbf{c}(\tau)}_{\mathbf{H}^{1}}^{2} \dd{\tau} \leq \frac{e^{2\alpha T}}{\min\{1, 2\beta\}} \norm{\mathbf{c}_{0}}_{\mathbf{L}^{2}}^{2}.
	\end{equation}
	Now, observe that by the continuity of $\mathcal{A}$, we have
	\[ \abs{\mathcal{B}(\dot{\mathbf{c}}(t), \Phi)} = \abs{-\mathcal{A}(\mathbf{c}(t), \Phi)} \leq K \norm{\mathbf{c}(t)}_{\mathbf{H}^{1}} \norm{\Phi}_{\mathbf{H}^{1}}, \]
	and therefore, we obtain
	\[ \norm{\dot{\mathbf{c}}(t)}_{(\mathbf{H}^{1})^{*}}^{2} \leq K^{2} \norm{\mathbf{c}(t)}_{\mathbf{H}^{1}}^{2}. \]
	Integration with respect to time implies that
	\begin{equation} \label{eqn:parabolic_est2}
		\int_{0}^{t} \norm{\dot{\mathbf{c}}(\tau)}_{(\mathbf{H}^{1})^{*}}^{2} \dd{\tau} \leq K^{2} \int_{0}^{t} \norm{\mathbf{c}(\tau)}_{\mathbf{H}^{1}}^{2} \dd{\tau} \leq K^{2}\frac{e^{2\alpha T}}{\min\{1, 2\beta\}} \norm{\mathbf{c}_{0}}_{\mathbf{L}^{2}}^{2}.
	\end{equation}
	Hence, adding the estimates \eqref{eqn:parabolic_est1} and \eqref{eqn:parabolic_est2}, we have
	\[ \norm{\mathbf{c}(t)}_{\mathbf{L}^{2}}^{2} + \int_{0}^{t} \norm{\mathbf{c}(\tau)}_{\mathbf{H}^{1}}^{2} \dd{\tau} + \int_{0}^{t} \norm{\dot{\mathbf{c}}(\tau)}_{(\mathbf{H}^{1})^{*}}^{2} \dd{\tau} \leq C \norm{\mathbf{c}_{0}}_{\mathbf{L}^{2}}^{2}, \]
	where $C = (1 + K^{2})\frac{e^{2\alpha T}}{\min\{1, 2\beta\}} > 0$, proving the stability estimate \eqref{eqn:WellposednessEnergyEstimate}.

	Finally, let us show the non-negativity of the solution provided the given initial datum is non-negative. In particular, given the non-negative initial datum $\mathbf{c}_{0}$, we multiply the PDE for $u$ with $e^{-\lambda t} u^{-}$ where $u^{-} := -\min\{0, u\} \in H^{1}(\Omega_{1})$ denotes the negative part of $u$, and analogously following the steps as in the stability estimate proof above, we obtain
	\[ \frac{1}{2}\dv{}{t} \qty(e^{-\lambda t} \norm{\mathbf{c}^{-}(t)}_{\mathbf{L}^{2}}^{2}) + e^{-\lambda t}\qty( \mathcal{A}(\mathbf{c}^{-}(t), \mathbf{c}^{-}(t)) +  \frac{\lambda}{2}\norm{\mathbf{c}^{-}(t)}_{\mathbf{L}^{2}}^{2} ) = 0. \]
	As before, the G\r{a}rding inequality in Lemma~\ref{lem:GardingIneq} and the positivity of the exponential function imply that
	\[ \dv{}{t} \qty(e^{-2\alpha t} \norm{\mathbf{c}^{-}(t)}_{\mathbf{L}^{2}}^{2}) \leq 0. \]
	Integration with respect to time yields $\norm{\mathbf{c}^{-}(t)}_{\mathbf{L}^{2}}^{2} \leq e^{2\alpha t} \norm{\mathbf{c}^{-}(0)}_{\mathbf{L}^{2}}^{2}=0$ since we have that $\mathbf{c}^{-}(0) = 0$ due to non-negativity of $\mathbf{c}_{0}$. Hence, $\mathbf{c}^{-}(x, t) = 0$ a.e. in $\Omega$ and thus, for all $t \in [0, T]$, the solution $\mathbf{c}$ is non-negative a.e. in $\Omega$.
\end{proof}

\begin{lemma}[Interpolation Estimates] \label{lem:interpolationEstimates}
	Let $\Omega \subset \mathbb{R}^{d}$, where $d=2,3$, be a bounded and Lipschitz domain and let $f \in H^{1}(\Omega)$. Then for any $\varepsilon, \tilde{\varepsilon} > 0$, there exist constants $C_{\varepsilon}, C_{\tilde{\varepsilon}} > 0$ such that
	\[ \norm{f}_{L^{2}(\Omega)}^{2} \leq \varepsilon \norm{\nabla f}_{L^{2}(\Omega)}^{2} + C_{\varepsilon}\norm{f}_{L^{1}(\Omega)}^{2} \text{ and } \norm{f}_{L^{2}(\partial\Omega)}^{2} \leq \tilde{\varepsilon} \norm{\nabla f}_{L^{2}(\Omega)}^{2} + C_{\tilde{\varepsilon}}\norm{f}_{L^{1}(\Omega)}^{2}.\]
\end{lemma}
\begin{proof}
	For the first estimate, the Sobolev Embedding Theorem \cite[Theorem 4.12]{Adams2003} implies that since $f \in H^{1}(\Omega)$, then $f \in L^{6}(\Omega)$ and $\norm{f}_{L^{6}(\Omega)} \leq C \norm{f}_{H^{1}(\Omega)}$ for some constant $C > 0$. Using the $L^{p}$ interpolation inequality \cite[Theorem 2.11]{Adams2003}, since $f \in L^{6}(\Omega) \cap L^{1}(\Omega)$, then $u \in L^{2}(\Omega)$ and
	\[ \norm{f}_{L^{2}(\Omega)} \leq \norm{f}_{L^{1}(\Omega)}^{2/5} \norm{f}_{L^{6}(\Omega)}^{3/5}. \]
	Therefore, we have
	\[ \norm{f}_{L^{2}(\Omega)} \leq \norm{f}_{L^{1}(\Omega)}^{2/5} \norm{f}_{L^{6}(\Omega)}^{3/5} \leq C \norm{f}_{L^{1}(\Omega)}^{2/5} \norm{f}_{H^{1}(\Omega)}^{3/5}. \]
	Squaring both sides leads to
	\[ \norm{f}_{L^{2}(\Omega)}^{2} \leq C^{2} (\norm{f}_{L^{1}(\Omega)}^{2})^{2/5} (\norm{f}_{H^{1}(\Omega)}^{2})^{3/5}. \]
	Using a rescaled $\varepsilon$-Young's inequality for conjugate exponents \cite[Appendix B.2.d]{evans2010} on the right-hand side, for any $0 < \tilde{\varepsilon} < 1$, there exists a constant $C_{\tilde{\varepsilon}} > 0$ (with $C^2$ absorbed into $C_{\tilde{\varepsilon}}$) such that
	\begin{align*}
		\norm{f}_{L^{2}(\Omega)}^{2} &\leq \tilde{\varepsilon}\norm{f}_{H^{1}(\Omega)}^{2} + C_{\tilde{\varepsilon}} \norm{f}_{L^{1}(\Omega)}^{2} \\
		&= \tilde{\varepsilon}(\norm{f}_{L^{2}(\Omega)}^{2} + \norm{\nabla f}_{L^{2}(\Omega)}^{2}) + C_{\tilde{\varepsilon}} \norm{f}_{L^{1}(\Omega)}^{2},
	\end{align*}
	which implies that
	\[ (1 - \tilde{\varepsilon})\norm{f}_{L^{2}(\Omega)}^{2} \leq \tilde{\varepsilon}\norm{\nabla f}_{L^{2}(\Omega)}^{2} + C_{\tilde{\varepsilon}}\norm{f}_{L^{1}(\Omega)}^{2}. \]
	Hence, we obtain with $\varepsilon = \tilde{\varepsilon}/(1 - \tilde{\varepsilon}) > 0$
	\[ \norm{f}_{L^{2}(\Omega)}^{2} \leq  \varepsilon \norm{\nabla f}_{L^{2}(\Omega)}^{2} + C_{\varepsilon} \norm{f}_{L^{1}(\Omega)}^{2}. \]

	The second estimate follows from \cite[Theorem 1.5.1.10]{grisvard2011} with $p=2$ and the interpolation estimate above.
\end{proof}

\begin{theorem}[Global $\mathbf{L}^{2}$ bounds]
	Given the non-negative initial datum $\mathbf{c}_{0} \in \mathbf{L}^{2}$, there exists a constant $C > 0$ such that for all $t \geq 0$, we have
	\begin{equation} \label{eqn:GlobalBound}
		\norm{\mathbf{c}(t)}_{\mathbf{L}^{2}}^{2} \leq C,
	\end{equation}
	that is, the global weak solution to system \eqref{eqn:ParabolicModel} is uniformly bounded in time.
\end{theorem}
\begin{proof}
	Take the time derivative of the quadratic function
	\[ \mathcal{H}(t) = \frac{1}{2}\qty( \norm{u(t)}_{L^{2}(\Omega_{1})}^{2} + \norm{v(t)}_{L^{2}(\Omega_{1})}^{2} + \norm{w(t)}_{L^{2}(\Omega_{2})}^{2} + \norm{z(t)}_{L^{2}(\Omega_{2})}^{2} ), \]
	to obtain (suppressing the variable $t$ in the notations)
	\begin{align}
		\dv{}{t}\mathcal{H} &= -d_{3}\norm{\nabla u}_{L^{2}(\Omega_{1})}^{2} - d_{2}\norm{\nabla v}_{L^{2}(\Omega_{1})}^{2} - d_{3}\norm{\nabla w}_{L^{2}(\Omega_{2})}^{2} - d_{2}\norm{\nabla z}_{L^{2}(\Omega_{2})}^{2} \notag \\
		&\quad -\kappa\norm{v}_{L^{2}(\Omega_{1})}^{2} - \rho \norm{w}_{L^{2}(\Omega_{2})}^{2} + \kappa(u, v)_{\Omega_{1}} + \rho (w, z)_{\Omega_{2}} \notag \\
		&\quad + \mu \norm{w - u}_{L^{2}(\Gamma)}^{2} + \nu\norm{z - v}_{L^{2}(\Gamma)}^{2}. \label{eqn:HFunc}
	\end{align}
	Let us estimate the inner products and the $L^{2}$-norms at the interface $\Gamma$. Observe that by Cauchy-Schwarz and Young's inequalities, we get
	\begin{align*}
		\kappa(u, v)_{\Omega_{1}} &\leq \frac{\kappa}{2}\qty(\norm{u}_{L^{2}(\Omega_{1})}^{2} + \norm{v}_{L^{2}(\Omega_{1})}^{2}), \\
		\rho(w, z)_{\Omega_{2}} &\leq \frac{\rho}{2}\qty(\norm{w}_{L^{2}(\Omega_{2})}^{2} + \norm{z}_{L^{2}(\Omega_{2})}^{2}), \\
		\mu \norm{w - u}_{L^{2}(\Gamma)}^{2} &\leq 2\mu \qty(\norm{w}_{L^{2}(\Gamma)}^{2} + \norm{u}_{L^{2}(\Gamma)}^{2}), \\
		\nu \norm{z - v}_{L^{2}(\Gamma)}^{2} &\leq 2\nu \qty(\norm{z}_{L^{2}(\Gamma)}^{2} + \norm{v}_{L^{2}(\Gamma)}^{2}).
	\end{align*}
	From Lemma~\ref{lem:interpolationEstimates}, for all $u, v \in H^{1}(\Omega_{1})$ and $w, z \in H^{1}(\Omega_{2})$ and for any $\varepsilon_{i} > 0$, $i=1,2,\ldots,6$, there exist corresponding constants $C_{\varepsilon_{i}} > 0$, $i=1,2,\ldots,6$, such that
	\begin{align*}
		\norm{u}_{L^{2}(\Gamma)}^{2} &\leq \varepsilon_{1} \norm{\nabla u}_{L^{2}(\Omega_{1})}^{2} + C_{\varepsilon_{1}}\norm{u}_{L^{1}(\Omega_{1})}^{2}, \\
		\norm{v}_{L^{2}(\Gamma)}^{2} &\leq \varepsilon_{2} \norm{\nabla v}_{L^{2}(\Omega_{1})}^{2} + C_{\varepsilon_{2}}\norm{v}_{L^{1}(\Omega_{1})}^{2}, \\
		\norm{w}_{L^{2}(\Gamma)}^{2} &\leq \varepsilon_{3} \norm{\nabla w}_{L^{2}(\Omega_{2})}^{2} + C_{\varepsilon_{3}}\norm{w}_{L^{1}(\Omega_{2})}^{2}, \\
		\norm{z}_{L^{2}(\Gamma)}^{2} &\leq \varepsilon_{4} \norm{\nabla z}_{L^{2}(\Omega_{2})}^{2} + C_{\varepsilon_{4}}\norm{z}_{L^{1}(\Omega_{2})}^{2},\\
		\norm{u}_{L^{2}(\Omega_{1})}^{2} &\leq  \varepsilon_{5} \norm{\nabla u}_{L^{2}(\Omega_{1})}^{2} + C_{\varepsilon_{5}} \norm{u}_{L^{1}(\Omega_{1})}^{2},\\
		\norm{z}_{L^{2}(\Omega_{2})}^{2} &\leq \varepsilon_{6} \norm{\nabla z}_{L^{2}(\Omega_{2})}^{2} + C_{\varepsilon_{6}} \norm{z}_{L^{1}(\Omega_{2})}^{2}.
	\end{align*}
    Due to the non-negativity of solutions and the total mass conservation law in \eqref{eqn:MassCon}, we also know that
	\[ \norm{u}_{L^{1}(\Omega_{1})}, \ \norm{v}_{L^{1}(\Omega_{1})}, \ \norm{w}_{L^{1}(\Omega_{2})},\ \norm{z}_{L^{1}(\Omega_{2})} \leq M_{0}. \]
	Therefore, going back to \eqref{eqn:HFunc}, and using the interpolation estimates above, we obtain
	\begin{align*}
		\dv{}{t}\mathcal{H} &\leq - d_{3}\norm{\nabla u}_{L^{2}(\Omega_{1})}^{2} - d_{2}\norm{\nabla v}_{L^{2}(\Omega_{1})}^{2} - d_{3}\norm{\nabla w}_{L^{2}(\Omega_{2})}^{2} - d_{2}\norm{\nabla z}_{L^{2}(\Omega_{2})}^{2} \\
		&\quad +\frac{\kappa}{2}\norm{u}_{L^{2}(\Omega_{1})}^{2} - \frac{\kappa}{2}\norm{v}_{L^{2}(\Omega_{1})}^{2} - \frac{\rho}{2}\norm{w}_{L^{2}(\Omega_{2})}^{2} + \frac{\rho}{2}\norm{z}_{L^{2}(\Omega_{2})}^{2} \\
		&\quad + 2\mu\norm{u}_{L^{2}(\Gamma)}^{2} + 2\nu\norm{v}_{L^{2}(\Gamma)}^{2} + 2\mu\norm{w}_{L^{2}(\Gamma)}^{2} + 2\nu\norm{z}_{L^{2}(\Gamma)}^{2} \\
		&\leq -\qty(\frac{d_{3}}{2} - 2\mu\varepsilon_{1})\norm{\nabla u}_{L^{2}(\Omega_{1})}^{2} - (d_{2} - 2\nu\varepsilon_{2})\norm{\nabla v}_{L^{2}(\Omega_{1})}^{2} \\
		&\quad - (d_{3} - 2\mu\varepsilon_{3})\norm{\nabla w}_{L^{2}(\Omega_{2})}^{2} - \qty(\frac{d_{2}}{2} - 2\nu\varepsilon_{4})\norm{\nabla z}_{L^{2}(\Omega_{2})}^{2} \\
		&\quad - \qty(\frac{d_{3}}{2}\frac{1}{\varepsilon_{5}} - \frac{\kappa}{2})\norm{u}_{L^{2}(\Omega_{1})}^{2} - \frac{\kappa}{2}\norm{v}_{L^{2}(\Omega_{1})}^{2} \\
		&\quad - \frac{\rho}{2}\norm{w}_{L^{2}(\Omega_{2})}^{2} - \qty(\frac{d_{2}}{2} \frac{1}{\varepsilon_{6}} - \frac{\rho}{2})\norm{z}_{L^{2}(\Omega_{2})}^{2} \\
		&\quad + (2\mu C_{\varepsilon_{1}} + C_{\varepsilon_{5}} d_{3}/2) \norm{u}_{L^{1}(\Omega_{1})}^{2} + 2\nu C_{\varepsilon_{2}}\norm{v}_{L^{1}(\Omega_{1})}^{2} \\
		&\quad + 2\mu C_{\varepsilon_{3}}\norm{w}_{L^{1}(\Omega_{2})}^{2} + (2\nu C_{\varepsilon_{4}} + C_{\varepsilon_{6}} d_{2}/2)\norm{z}_{L^{1}(\Omega_{2})}^{2}.
	\end{align*}
	Note that $\frac{d_{3}}{2} - 2\mu\varepsilon_{1} > 0$ provided $\varepsilon_{1} < \frac{d_{3}}{4\mu}$. Thus, choose $0 < \varepsilon_{1} < \frac{d_{3}}{4\mu}$. Analogously, choose $0 < \varepsilon_{2} < \frac{d_{2}}{2\nu}$, $0 < \varepsilon_{3} < \frac{d_{3}}{2\mu}$, $0 < \varepsilon_{4} < \frac{d_{2}}{4\nu}$, $0 < \varepsilon_{5} < \frac{d_{3}}{\kappa}$, and $0 < \varepsilon_{6} < \frac{d_{2}}{\rho}$. Therefore, we have that
	\[ \dv{}{t}\mathcal{H} + a \mathcal{H} \leq b, \]
	where
    \[ a = \min\qty{ \frac{d_{3}}{2}\frac{1}{\varepsilon_{5}} - \frac{\kappa}{2}, \frac{\kappa}{2}, \frac{\rho}{2}, \frac{d_{2}}{2} \frac{1}{\varepsilon_{6}} - \frac{\rho}{2} } > 0, \]
    and
    \[ b = 2M_{0}^{2}(\mu(C_{\varepsilon_{1}} + C_{\varepsilon_{3}}) + \nu(C_{\varepsilon_{2}} + C_{\varepsilon_{4}}) + C_{\varepsilon_{5}}d_{3} + C_{\varepsilon_{6}}d_{2}) > 0, \]
    are both constants. By a classical Gronwall argument, we obtain the estimate
	\[ \mathcal{H}(t) \leq e^{-a t} \mathcal{H}(0) + \frac{b}{a}(e^{-at} - 1). \]
	Since $e^{-at} \leq 1$ for all $t \geq 0$, we have $\norm{\mathbf{c}(t)}_{\mathbf{L}^{2}}^{2} \leq C$ for all $t \geq 0$ where $C = \mathcal{H}(0) > 0$ as desired.
\end{proof}

\begin{theorem}[Regularity]
	If $\Omega_{1}$ and $\Omega_{2}$ are bounded $C^{2}$-domains and $\mathbf{c}_{0} \in \mathbf{H}^{1}$, then the weak solution $\mathbf{c} = (u, v, w, z)$ to system \eqref{eqn:ParabolicModel} in Theorem~\ref{thm:WellposednessParabolicSystem} satisfies the following regularities:
	\begin{align*}
		u, v \in L^{2}(0, T; H^{2}(\Omega_{1})) \cap H^{1}(0, T; L^{2}(\Omega_{1})), \\
		w, z \in L^{2}(0, T; H^{2}(\Omega_{2})) \cap H^{1}(0, T; L^{2}(\Omega_{2})).
	\end{align*}
\end{theorem}
\begin{proof}
	The proof follows from classical arguments since $u_0, v_0 \in H^{1}(\Omega_{1})$ and $w_0, z_0 \in H^{1}(\Omega_{2})$, see for instance \cite[Section 9.4]{salsa2022}.
\end{proof}

We now prove the well-posedness of the corresponding elliptic problem of system \eqref{eqn:ParabolicModel}. Our main result is summarised in the following theorem.
\begin{theorem}[Well-posedness of the Elliptic System] \label{thm:WellposenessElliptic}
	Suppose $\Omega_{1}$ and $\Omega_{2}$ are bounded $C^{4}$-domains. For any given positive total mass $M_{0} > 0$, there exists a unique positive equilibrium $\mathbf{c}_{\infty}(M_0) := (u_{\infty}, v_{\infty}, w_{\infty}, z_{\infty})$ solving the corresponding elliptic system of \eqref{eqn:ParabolicModel}:
	\begin{equation} \label{eqn:EllipticModel1}
		\begin{split}
			& \begin{cases}
				- d_{3} \Delta u_{\infty} = \kappa v_{\infty}, & \quad \ \ x \in \Omega_{1}, \\
				d_{3} \partial_{\eta} u_{\infty} = \mu[w_{\infty} - u_{\infty}], & \quad \ \ x \in \Gamma,\\
				- d_{2} \Delta v_{\infty} = -\kappa v_{\infty}, & \quad \ \ x \in \Omega_{1},\\
				d_{2} \partial_{\eta} v_{\infty} = \nu[z_{\infty} - v_{\infty}], & \quad \ \ x \in \Gamma,\\
			\end{cases} \\
			& \begin{cases}
				- d_{3} \Delta w_{\infty} = - \rho w_{\infty}, & x \in \Omega_{2},\\
				-d_{3} \partial_{\eta} w_{\infty} = -\mu[w_{\infty} - u_{\infty}], & x \in \Gamma, \\
				d_{3} \partial_{\eta} w_{\infty} = 0, & x \in \partial\Omega,\\
				- d_{2} \Delta z_{\infty} = \rho w_{\infty}, & x \in \Omega_{2},\\
				-d_{2} \partial_{\eta} z_{\infty} = -\nu[z_{\infty} - v_{\infty}], & x \in \Gamma, \\
				d_{2} \partial_{\eta} z_{\infty} = 0, & x \in \partial\Omega,
			\end{cases}
		\end{split}
	\end{equation}
	where the total mass $M_0$ corresponding to the total mass conservation law \eqref{eqn:MassCon} of the evolution problem is given by
	\begin{equation} \label{eqn:MassElliptic}
		\int_{\Omega_{1}} [u_{\infty}(x) + v_{\infty}(x)] \dd{x} + \int_{\Omega_{2}} [w_{\infty}(x) + z_{\infty}(x)] \dd{x} = M_{0}.
	\end{equation}
	Moreover, we have
	\[ u_{\infty}, v_{\infty} \in C^{2}(\overline{\Omega}_{1}), \quad w_{\infty}, z_{\infty} \in C^{2}(\overline{\Omega}_{2}), \]
	and for some constants $0 < m \leq M < \infty$,
	\begin{align*}
		0 < m \leq u_{\infty}(x), v_{\infty}(x) \leq M \qq{for all} x \in \overline{\Omega}_{1}, \\
		0 < m \leq w_{\infty}(x), z_{\infty}(x) \leq M \qq{for all} x \in \overline{\Omega}_{2}.
	\end{align*}
\end{theorem}
\begin{proof}
	The existence of solutions to \eqref{eqn:EllipticModel1} is far from trivial due to the complex-balanced coupling of four equations in two subdomains (recall Figure~\ref{fig:chemfig}b). We are unaware of a general existence theory for such systems; regardless of \eqref{eqn:EllipticModel1} being linear. In the following, we use some ideas from \cite{kfellnerbtang2017} and prove existence via a fixed point argument applied to an iteration of auxiliary solutions. One possible choice of an auxiliary system would preserve the total mass throughout the iterations, yet not the non-negativity of the solutions as such. Therefore, we choose the following partially decoupled auxiliary system, which features non-negative solutions for given non-negative input:
	\begin{equation} \label{eqn:EllipticModel2}
		\begin{split}
			& \begin{cases}
				- d_{3} \Delta u = \kappa v, & \quad \ \ x \in \Omega_{1}, \\
				d_{3} \partial_{\eta} u = \mu[w_{0} - u], & \quad \ \ x \in \Gamma,\\
				- d_{2} \Delta v = -\kappa v, & \quad \ \ x \in \Omega_{1},\\
				d_{2} \partial_{\eta} v = \nu[z_{0} - v], & \quad \ \ x \in \Gamma,\\
			\end{cases} \\
			& \begin{cases}
				- d_{3} \Delta w = - \rho w, & x \in \Omega_{2},\\
				-d_{3} \partial_{\eta} w = -\mu[w - u_{0}], & x \in \Gamma, \\
				d_{3} \partial_{\eta} w = 0, & x \in \partial\Omega,\\
				- d_{2} \Delta z = \rho w, & x \in \Omega_{2},\\
				-d_{2} \partial_{\eta} z = -\nu[z - v_{0}], & x \in \Gamma, \\
				d_{2} \partial_{\eta} z = 0, & x \in \partial\Omega,
			\end{cases}
		\end{split}
	\end{equation}
	where the input functions $\mathbf{c}_{0} = (u_{0}, v_{0}, w_{0}, z_{0}) \in \mathbf{Y}$ are given from the space
	\[ \mathbf{Y} := \{ \mathbf{c} = (u, v, w, z) \in \mathbf{H}^{3}: u, v, w, z \geq 0 \}. \]
	Note that the solution $\mathbf{c} = (u, v, w, z)$ cannot be expected  to have the same total mass as the previous iteration $\mathbf{c}_{0} = (u_{0}, v_{0}, w_{0}, z_{0})$. The solution $\mathbf{c}$ will even be zero if the trace of $\mathbf{c}_{0}$ vanishes at $\Gamma$.

	In order to prove the existence of solutions to \eqref{eqn:EllipticModel1} with positive total mass $M_0$ as fixed point of an iteration scheme, we therefore show positivity of the solutions $\mathbf{c}$ subject to positive $\mathbf{c}_{0}$ by using the strong maximum principle, which requires higher regularity and classical solutions. In particular, we choose, as an initial approximation, $\mathbf{c}_{0}$ to be positive, piecewise constant functions in the two subdomains such that for any $M_{0} > 0$, the total mass conservation law holds,
	\[ \int_{\Omega_{1}} [u_{0} + v_{0}] \dd{x} + \int_{\Omega_{2}} [w_{0} + z_{0}] \dd{x} = M_{0} > 0. \]

	The existence and uniqueness of a non-negative solution $\mathbf{c}$ to the auxiliary system \eqref{eqn:EllipticModel2} is shown as follows. Let us first consider the PDE for $v$ which is independent of the unknown $u$ in $\Omega_{1}$:
	\[ -d_{2} \Delta v + \kappa v = 0 \qq{in} \Omega_{1}, \quad d_{2} \partial_{\eta}v + \nu v = \nu z_{0} \qq{on} \Gamma. \]
	The Lax-Milgram theorem guarantees the existence and uniqueness of the weak solution $v \in H^{1}(\Omega_{1})$ to the variational problem: given $z_{0} \in H^{3}(\Omega_{2})$, find $v \in H^{1}(\Omega_{1})$ such that
	\[ \int_{\Omega_{1}} d_{2} \nabla v \cdot \nabla \varphi_{2} \dd{x} + \int_{\Gamma} \nu v \varphi_{2} \dd{s} + \int_{\Omega_{1}} \kappa v \varphi_{2} \dd{x} = \int_{\Gamma} \nu z_{0} \varphi_{2} \dd{s}, \]
	for all test functions $\varphi_{2} \in H^{1}(\Omega_{1})$. To show non-negativity given $z_{0} \geq 0$, observe that using $\varphi_{2} = v^{-} := -\min\{v, 0\}$, the negative part of $v$ as test function, and noting that $v v^{-} = -(v^{-})^{2}$ and $\nabla v \cdot \nabla v^{-} = - \chi_{\{v < 0\}} \abs{\nabla v}^{2}$, we obtain
	\[ -d_{2} \int_{\Omega_{1}} \chi_{\{v \leq 0\}} |\nabla v|^{2} \dd{x} - \kappa \int_{\Omega_{1}} (v^{-})^{2} \dd{x} - \nu \int_{\Gamma} (v^{-})^{2} \dd{s} = \int_{\Gamma} z_{0} v^{-} \dd{s}. \]
	Since the left-hand side is non-positive and the right-hand side is non-negative, then both sides are equal to zero and thus, $\int_{\Omega_{1}} (v^{-})^{2} \dd{x} = 0$ and $v \geq 0$ is non-negative in $\Omega_{1}$. Since $z_{0} \in H^{3 - \frac{1}{2}}(\Gamma) = H^{2 + \frac{1}{2}}(\Gamma)$ and $\Omega_{1}$ is a bounded $C^{4}$-domain, by maximal regularity for elliptic systems \cite[Theorem 8.31]{salsa2022}, we have that $v \in H^{4}(\Omega_{1})$ and the stability estimate
	\[ \norm{v}_{H^{4}(\Omega_{1})} \leq C \norm{z_{0}}_{H^{3}(\Omega_{2})}, \]
	for some constant $C > 0$ holds. By Sobolev Embedding Theorem for $n=2,3$, we get $v \in C^{2,\alpha}(\overline{\Omega}_{1})$ for some $\alpha \in (0, 1)$, and thus, $v \in C^{2}(\overline{\Omega}_{1})$.
	Hence, the strong maximum principle \cite{pao1992} implies that $v$ is strictly positive on $\overline{\Omega}_{1}$ whenever $z_0$ is strictly positive on $\Gamma$.

	With this unique positive solution $v$, we now consider the PDE for $u$:
	\[ - d_{3} \Delta u = \kappa v \qq{in} \Omega_{1}, \qquad d_{3} \partial_{\eta} u = \mu[w_{0} - u] \qq{on} \Gamma. \]
	As before, the Lax-Milgram Theorem guarantees the existence and uniqueness of a weak solution $u \in H^{1}(\Omega_{1})$ to the variational problem: find $u \in H^{1}(\Omega_{1})$ such that
	\[ \int_{\Omega_{1}} d_{3} \nabla u \cdot \nabla \varphi_{1} \dd{x} + \int_{\Gamma} \mu u \varphi_{1} \dd{s} = \int_{\Omega_{1}} \kappa v \varphi_{1} \dd{x} + \int_{\Gamma} \mu w_{0} \varphi_{1} \dd{s}, \]
	for all test functions $\varphi_{1} \in H^{1}(\Omega_{1})$. Since $v \in H^{4}(\Omega_{1}) \subset H^{2}(\Omega_{1})$ and $w_{0} \in H^{3 - \frac{1}{2}}(\Gamma) = H^{2 + \frac{1}{2}}(\Gamma)$, by maximal regularity, $u \in H^{4}(\Omega_{1})$ and
	\[ \norm{u}_{H^{4}(\Omega_{1})} \leq C\qty{ \norm{w_{0}}_{H^{3}(\Omega_{2})} + \norm{z_{0}}_{H^{3}(\Omega_{2})} }, \]
	for some constant $C > 0$. Hence, following the same embedding arguments as for $v$, we get $u \in C^{2}(\overline{\Omega}_{1})$. To prove non-negativity of the unique solution we choose $\varphi_{1} = u^{-}$ as test function and obtain
	\[ -\int_{\Omega_{1}} \chi_{\{u < 0\}} d_{3} \abs{\nabla u}^{2} \dd{x} - \int_{\Gamma} \mu (u^{-})^{2} \dd{s} = \int_{\Omega_{1}} \kappa v u^{-} \dd{x} + \int_{\Gamma} \mu w_{0} u^{-}. \]
	Note that the left-hand side is again non-positive while the right-hand side is non-negative since $v, w_{0} \geq 0$. Hence, both sides must be equal to zero and as a consequence, $\int_{\Omega_{1}} \chi_{\{u < 0\}} \abs{\nabla u}^{2} \dd{x} = 0$ and $\int_{\Gamma} \mu (u^{-})^{2} \dd{s} = 0$, which implies that $u^{-}$ is constant in $\Omega_{1}$ and $u^{-} = 0$ on $\Gamma$. By continuity (up to the boundary), we have that $u^{-} = 0$ in $\Omega_{1}$, and hence, $u \geq 0$ in $\Omega_{1}$. Finally, the strong maximum principles implies that $u$ is strictly positive in $\overline{\Omega}_{1}$ provided that either $v$ is strictly positive in $\Omega_{1}$ or $w_0$ is strictly positive on $\Gamma$.

	Following analogous arguments for $w$ and $z$ in $\Omega_{2}$, we therefore obtain a non-negative weak solution $\mathbf{c} = (u,v,w,z) \in \mathbf{H}^{3}$ to the auxiliary system \eqref{eqn:EllipticModel2}, and in addition, due to maximal regularity, we have $\mathbf{c} \in \mathbf{H}^{4}$ and the \textit{a priori} estimate
	\begin{equation} \label{eqn:Estimate1}
		\norm{\mathbf{c}}_{\mathbf{H}^{4}} \leq C \norm{\mathbf{c}_{0}}_{\mathbf{H}^{3}},
	\end{equation}
	holds for some constant $C > 0$. By Sobolev embeddings for $n=2,3$, we have that $\mathbf{c}$ is a classical solution to the auxiliary system \eqref{eqn:EllipticModel2}. Finally, by the strong maximum principle, strictly positive input functions $\mathbf{c}_0$ yield a strictly positive solution $\mathbf{c}$.

	However, we cannot expect the solution $\mathbf{c}$ to represent the same total mass $M_0 > 0$ as the initial guess $\mathbf{c}_{0}$. Nevertheless, the positivity of $\mathbf{c}$ implies that
	\[ \int_{\Omega_{1}} [u + v] \dd{x} + \int_{\Omega_{2}} [w + z] \dd{x} = M_{1}>0, \]
	for some $M_{1} > 0$. We can now rescale the input functions $\mathbf{c}_{0}$ and the solutions to \eqref{eqn:EllipticModel2} by choosing a multiplier $\lambda = M_{0} / M_{1}> 0$: by replacing $\mathbf{c}_{0}$ with $\lambda \mathbf{c}_0$ as input in the auxiliary system \eqref{eqn:EllipticModel2}, we obtain (by linearity) $\lambda$-times the previous solution. Hence, the rescaled solution $\mathbf{c} \in \mathbf{Y}$ contains the desired total mass
	\[ \int_{\Omega_{1}} [u + v] \dd{x} + \int_{\Omega_{2}} [w + z] \dd{x} = M_{0}. \]

	Define the mapping $\mathcal{I}: \mathbf{Y} \to \mathbf{Y}$ by $\mathcal{I}\mathbf{c}_{0} = \mathbf{c}$. Due to the \textit{a priori} estimate \eqref{eqn:Estimate1}, $\mathcal{I}$ is a compact mapping. By the Schauder fixed point theorem, there exists a fixed point denoted by $\mathbf{c}_{\infty} := (u_{\infty}, v_{\infty}, w_{\infty}, z_{\infty})$ of the mapping $\mathcal{I}$, and this fixed point is thus the non-negative weak solution of the elliptic system \eqref{eqn:EllipticModel1} with total mass $M_0$. The equilibrium solution $\mathbf{c}_{\infty}$ enjoys higher regularity due to the maximal regularity, that is, $\mathbf{c}_{\infty} \in \mathbf{H}^{4}$. Thus, as above, the Sobolev Embedding Theorem implies that for $n=2, 3$, we obtain
	\[ u_{\infty}, v_{\infty} \in C^{2}(\overline{\Omega}_{1}) \qq{and} w_{\infty}, z_{\infty} \in C^{2}(\overline{\Omega}_{2}). \]
	Due to strong maximum principle, we also conclude that given $M_{0} > 0$, we have that $\mathbf{c}_{\infty} > 0$. In case $M_{0} = 0$, we recover the trivial solution $\mathbf{c}_{\infty} = 0$.

	The upper bound $M > 0$ for the equilibrium solutions is due to the continuity of the functions and the compactness of their corresponding domains. The positive lower bound $m >0$ follows from the strong maximum principle like for to the solutions of the auxiliary system \eqref{eqn:EllipticModel2}.

	Finally, we prove the uniqueness of the equilibrium solution to system \eqref{eqn:EllipticModel1} via the entropy structure of the parabolic system \eqref{eqn:ParabolicModel}. From Appendix~\ref{sec:A1}, recall the \textit{relative entropy functional} $\mathcal{E}$ in equation~\eqref{eqn:RelativeEntropy} and the \textit{entropy-production functional} $\mathcal{D}$ in equation~\eqref{eqn:EntropyProduction}. Let $\mathbf{c}_{\infty}^{(1)} := (u_{\infty}^{(1)}, v_{\infty}^{(1)}, w_{\infty}^{(1)}, z_{\infty}^{(1)})$ and $\mathbf{c}_{\infty}^{(2)} := (u_{\infty}^{(2)}, v_{\infty}^{(2)}, w_{\infty}^{(2)}, z_{\infty}^{(2)})$ be two solutions of the stationary state system \eqref{eqn:EllipticModel1}. By being equilibrium solutions, we see that in Lemma~\ref{lem:EntropyProductionFunc}, we have $\mathcal{D}(\mathbf{c}_{\infty}^{(1)}|\mathbf{c}_{\infty}^{(2)}) = -\dv{t} \mathcal{E}(\mathbf{c}_{\infty}^{(1)}|\mathbf{c}_{\infty}^{(2)}) = 0$ independent of time. Thus, from the entropy-production functional \eqref{eqn:EntropyProduction}, we obtain
	\[ \frac{u_{\infty}^{(1)}}{u_{\infty}^{(2)}} \equiv \frac{v_{\infty}^{(1)}}{v_{\infty}^{(2)}} \equiv \frac{w_{\infty}^{(1)}}{w_{\infty}^{(2)}} \equiv \frac{z_{\infty}^{(1)}}{z_{\infty}^{(2)}} \equiv k, \]
	for some constant $0 \ne k \in \mathbb{R}$. If the equilibrium solutions $\mathbf{c}_{\infty}^{(1)}$ and $\mathbf{c}_{\infty}^{(2)}$ satisfy the same total mass conservation law, then we see that $ k \equiv 1$ and thus, $ \mathbf{c}_{\infty}^{(1)} \equiv \mathbf{c}_{\infty}^{(2)} $, which proves uniqueness of the equilibrium solution for a fixed total mass.
\end{proof}

We end this section with the following property of the global weak solution of system \eqref{eqn:ParabolicModel}.
\begin{theorem}[Exponential Convergence to Equilibrium] \label{thm:ExponentialConvergence}
	Under the assumptions of Theorem~\ref{thm:WellposenessElliptic}, we further conclude that every global weak solution $\mathbf{c} = (u,v,w,z)$ to system \eqref{eqn:ParabolicModel} with given fixed initial total mass $M_{0} > 0$ converges exponentially to $\mathbf{c}_{\infty} = (u_{\infty}, v_{\infty}, w_{\infty}, z_{\infty})$ in the sense that
	\[ \begin{split}
		&\int_{\Omega_{1}} \qty[ \frac{|u - u_{\infty}|^{2}}{u_{\infty}} + \frac{|v - v_{\infty}|^{2}}{v_{\infty}}] \dd{x} + \int_{\Omega_{2}} \qty[ \frac{|w - w_{\infty}|^{2}}{w_{\infty}} + \frac{|z - z_{\infty}|^{2}}{z_{\infty}} ] \dd{x} \\
        &\qquad \qquad \leq C \exp(-\lambda t),
	\end{split} \]
	for all $t > 0$, where the constants $\lambda > 0$ and $C > 0$ can be explicitly computed.
\end{theorem}
\begin{proof}
	By definition \eqref{eqn:RelativeEntropy} and \eqref{eqn:EntropyProduction}, we have $\dv{}{t} \mathcal{E}(\mathbf{c} - \mathbf{c}_{\infty}|\mathbf{c}_{\infty}) + \mathcal{D}(\mathbf{c} - \mathbf{c}_{\infty}|\mathbf{c}_{\infty}) = 0$. We use the result in Lemma 2.4 of \cite{kfellnerbtang2017}: given a fixed total mass $M_{0} > 0$, the following entropy entropy-production estimate
	\begin{equation} \label{eqn:Entropyproduction}
		\mathcal{D}(\mathbf{c} - \mathbf{c}_{\infty}|\mathbf{c}_{\infty}) \geq \lambda \mathcal{E}(\mathbf{c} - \mathbf{c}_{\infty}|\mathbf{c}_{\infty}),
	\end{equation}
	holds for any non-negative measurable functions $\mathbf{c} = (u,v,w,z)$ satisfying the mass conservation law \eqref{eqn:MassCon}, where $\mathbf{c}_{\infty} = (u_{\infty}, v_{\infty}, w_{\infty}, z_{\infty} )$ is as defined in Theorem~\ref{thm:WellposenessElliptic}, and the constant $\lambda$ can be explicitly estimated. With estimate \eqref{eqn:EntropyProduction}, we have that
	\[ \dv{}{t} \mathcal{E}(\mathbf{c} - \mathbf{c}_{\infty}|\mathbf{c}_{\infty}) + \lambda\mathcal{E}(\mathbf{c} - \mathbf{c}_{\infty}|\mathbf{c}_{\infty}) \leq 0, \]
	and by a classical Gronwall argument, we get $\mathcal{E}(\mathbf{c} - \mathbf{c}_{\infty}|\mathbf{c}_{\infty}) \leq C \exp(-\lambda t)$ as desired.
\end{proof}

\section{Numerical Simulations} \label{sec:numerical_analysis}

We first consider the equilibrium system \eqref{eqn:EllipticModel1}. For convenience, we work only in two dimensions ($d=2$), however, one can easily extend it to three dimensions. We then consider the variational problem: find $\mathbf{c} \in \mathbf{H}^{1}$ such that
\begin{equation} \label{eqn:EllipticVarForm}
	\mathcal{A}(\mathbf{c}, \Phi) = 0 \qq{for all test functions} \Phi \in \mathbf{H}^{1},
\end{equation}
where the bilinear form is given by
\begin{align}
	\mathcal{A}(\mathbf{c}, \Phi) &:= d_{3} (\nabla u, \nabla \varphi_{1})_{\Omega_{1}} + d_{2} (\nabla v, \nabla \varphi_{2})_{\Omega_{1}} + d_{3} (\nabla w, \nabla \varphi_{3})_{\Omega_{2}} \notag \\
	&+ d_{2} (\nabla z, \nabla \varphi_{4})_{\Omega_{2}} + \kappa (v, \varphi_{2} - \varphi_{1})_{\Omega_{1}}  + \rho (w, \varphi_{3} - \varphi_{4})_{\Omega_{2}} \notag \\
	&+ \mu(w - u, \varphi_{3} - \varphi_{1})_{\Gamma} + \nu(z - v, \varphi_{4} - \varphi_{2})_{\Gamma}, \label{eqn:BilinearFormElliptic}
\end{align}
and such that, in addition, the total mass conservation law in \eqref{eqn:MassCon} is satisfied. The bilinear form \eqref{eqn:BilinearFormElliptic} is the same as the bilinear form in \eqref{eqn:BilinearForm2} yet without time dependence as we are considering an equilibrium state. Note that the existence and uniqueness of the nontrivial positive weak solution of \eqref{eqn:EllipticVarForm} is already established in Theorem~\ref{thm:WellposenessElliptic}. We remark again that the bilinear form $\mathcal{A}$ in \eqref{eqn:BilinearFormElliptic} is continuous and not coercive, however, it satisfies the G\r{a}rding inequality \eqref{eqn:GardingIneq} in Lemma \ref{lem:GardingIneq}.

Following the framework in \cite[Chapter 2]{QuarteroniValli1999}, let $\Omega_{1,h}$ and $\Omega_{2,h}$ be polygonal approximations of the subdomains $\Omega_{1}$ and $\Omega_{2}$, respectively, sharing the same polygonal interface $\Gamma_{h} := \partial\Omega_{1,h} \cap \partial\Omega_{2,h}$. Let $\mathcal{T}_{1,h}$ and $\mathcal{T}_{2,h}$ be shape-regular triangulations of $\Omega_{1,h}$ and $\Omega_{2,h}$, respectively. We assume that the two triangulations are constructed such that they match along $\Gamma_h$, that is, the nodes and edges of $\mathcal{T}_{1,h}$ and $\mathcal{T}_{2,h}$ coincide at $\Gamma_{h}$. Thus, consider $\mathcal{T}_{h} := \mathcal{T}_{1,h} \cup \mathcal{T}_{2,h}$ as a conforming shape-regular triangulation of $\Omega_{h} := \overline{\Omega}_{1,h} \cup \Omega_{2,h}$. We also assume that $\Gamma_{h} \subset \Omega_{h}$. The mesh size is denoted by $h = \max \{ h_{T} : T \in \mathcal{T}_{h} \}$. Finally, define the finite dimensional product space $\mathbf{V}_{h}^{k} := V_{1,h}^{k} \times V_{1,h}^{k} \times V_{2,h}^{k} \times V_{2,h}^{k}$ where
\begin{align*}
	V_{1,h}^{k} &:= \{ \varphi_{1} \in C(\Omega_{1,h}) : \varphi_{1}|_{T} \in \mathbb{P}^{k}(T) \qq{for all} T \in \mathcal{T}_{1,h} \}, \\
	V_{2,h}^{k} &:= \{ \varphi_{2} \in C(\Omega_{2,h}) : \varphi_{2}|_{T} \in \mathbb{P}^{k}(T) \qq{for all} T \in \mathcal{T}_{2,h} \},
\end{align*}
and $\mathbb{P}^{k}$ denotes the space of polynomial functions of degree less than or equal to $k \in \mathbb{N}$.

\begin{remark}[Variational Crimes]\label{rem:VarCrim}
	In the \texttt{FEniCS} implementation, we approximate the domains $\Omega_{1}$ and $\Omega_{2}$ by $\Omega_{1,h}$ and $\Omega_{2,h}$, respectively, where in particular, $\Omega_{2,h} \not\subset \Omega_{2}$ (see Figure~\ref{fig:var_crime}). Moreover, we commit another variational crime by using quadrature rules for the integrals.
\end{remark}
\begin{figure}[htbp]
	\centering
	\includegraphics[width = 1\textwidth]{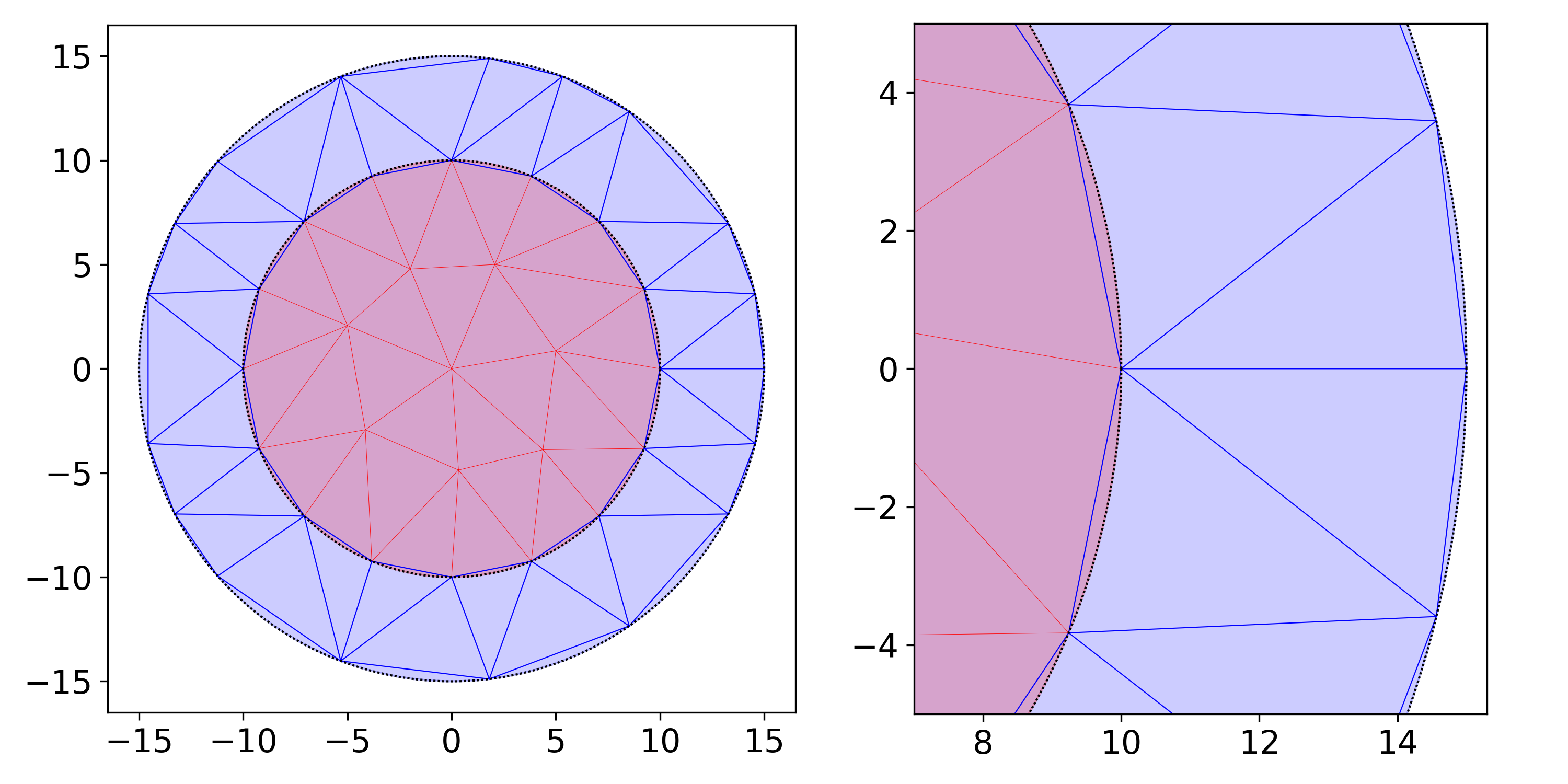}
	\caption{Illustration that $\Omega_{2,h} \not \subset \Omega_{2}$ due to the nonconvexity of the active region $\Omega_{2}$. Here, we set $R = 10$ and $\delta = 5$}
	\label{fig:var_crime}
\end{figure}

For $\mathbf{c}_{h} := (u_{h}, v_{h}, w_{h}, z_{h}) \in \mathbf{V}_{h}^{k}$, we use the norm
\[ \norm{\mathbf{c}_{h}}_{\mathbf{V}_{h}^{k}} = \qty(\norm{u_{h}}_{H^{1}(\Omega_{1,h})}^{2} + \norm{v_{h}}_{H^{1}(\Omega_{1,h})}^{2} + \norm{w_{h}}_{H^{1}(\Omega_{2,h})}^{2} + \norm{z_{h}}_{H^{1}(\Omega_{2,h})}^{2})^{\frac{1}{2}}. \]
Consider the following discrete variational problem corresponding to \eqref{eqn:EllipticVarForm}: find $\mathbf{c}_{h} \in \mathbf{V}_{h}^{k}$ such that
\begin{align}
	\mathcal{A}_{h}(\mathbf{c}_{h}, \Phi) &:= d_{3} (\nabla u_{h}, \nabla \varphi_{1})_{\Omega_{1,h}} + d_{2} (\nabla v_{h}, \nabla \varphi_{2})_{\Omega_{1,h}} \notag \\
	&+ d_{3} (\nabla w_{h}, \nabla \varphi_{3})_{\Omega_{2,h}} + d_{2} (\nabla z_{h}, \nabla \varphi_{4})_{\Omega_{2,h}} \notag \\
	&+ \kappa (v_{h}, \varphi_{2} - \varphi_{1})_{\Omega_{1,h}} + \rho(w_{h}, \varphi_{3} - \varphi_{4})_{\Omega_{2,h}} \notag \\
	&+ \mu(w_{h} - u_{h}, \varphi_{3} - \varphi_{1})_{\Gamma_{h}} + \nu(z_{h} - v_{h}, \varphi_{4} - \varphi_{2})_{\Gamma_{h}} \notag \\
	& = 0 \qq{for all test functions} \Phi \in \mathbf{V}_{h}^{k}, \label{eqn:VarProb_Elliptic_Discrete}
\end{align}
and in addition, the discrete total mass conservation law
\begin{equation} \label{eqn:Discrete_MassCon}
	\int_{\Omega_{1,h}} (u_{h} + v_{h}) \dd{x} + \int_{\Omega_{2,h}} (w_{h} + z_{h}) \dd{x} = M_{0,h},
\end{equation}
holds, where $M_{0,h} = \int_{\Omega_{1,h}} (u_{0,h} + v_{0,h}) \dd{x} + \int_{\Omega_{2,h}} (w_{0,h} + z_{0,h}) \dd{x} > 0$ is the given initial total mass of the discrete system. Observe that $\mathcal{A}_{h}$ is the same bilinear form as in the bilinear form $\mathcal{A}$ in \eqref{eqn:BilinearFormElliptic}, except that it acts on a different function space.

In order to have a comparison to the numerical solutions, we first compute for the explicit radially symmetric solutions of the eqilibrium system \eqref{eqn:EllipticModel1}. For functions depending only on the radius $r$, with $r^{2}=x^{2}+y^{2}$ for all $(x, y) \in \Omega$, the Laplace operator reduces to $\Delta = \dv[2]{}{r} + \frac{1}{r} \dv{}{r}$. Therefore, the elliptic system \eqref{eqn:EllipticModel1} is transformed into the following second order system of boundary value problems
\begin{equation} \label{eqn:ExactSolutions}
	\begin{split}
		&\begin{cases}
			d_{3}(u'' + \frac{1}{r} u') = -\kappa v, & r \in (0, R), \\
			d_{3}u'(R) = \mu[w(R) - u(R)], & \\
			d_{2}(v'' + \frac{1}{r} v') = \kappa v, & r \in (0, R), \\
			d_{2}v'(R) = \nu[z(R) - v(R)],
		\end{cases} \\
		&\begin{cases}
			d_{3}(w'' + \frac{1}{r} w') = \rho w, & r \in (R, R+\delta), \\
			d_{3}w'(R) = \mu[w(R) - u(R)], & d_{3}w'(R+\delta) = 0, \\
			d_{2}(z'' + \frac{1}{r}z') = -\rho w, & r \in (R, R+\delta), \\
			d_{2}z'(R) = \nu[z(R) - v(R)], & d_{2}z'(R+\delta) = 0,
		\end{cases}
	\end{split}
\end{equation}
where $' = \dv{}{r}$. Note that $u, v : (0,R) \to \mathbb{R}^{+}$ and $w, z: (R, R+\delta) \to \mathbb{R}^{+}$ due to symmetry. Since the system parameters are assumed to be positive constants, we expect to obtain radially symmetric solutions, however, this is not always the case for nonconstant parameters. This is also the reason behind the numerical simulations being implemented in two dimensions in preparation for the nonhomogeneous parameters.
\begin{theorem}[Radially Symmetric Solutions]
	The solutions of the boundary value problem \eqref{eqn:ExactSolutions} are given by
	\begin{align}
		& v(r) = k_{1} I_{0}(\alpha r), && u(r) = \frac{C_{1}}{d_{3}} - \frac{d_{2}}{d_{3}}v(r), \label{eqn:ExactSolutions_uv} \\
		& w(r) = k_{2} I_{0}(\beta r) + k_{3} K_{0}(\beta r), && z(r) = \frac{C_{2}}{d_{2}} - \frac{d_{3}}{d_{2}}w(\beta r), \label{eqn:ExactSolutions_wz}
	\end{align}
	where $\alpha^{2} = \kappa/d_{2}$, $\beta^{2} = \rho /d_{3}$, $k_{1}$, $k_{2}$, $k_{3}$ are explicitly computed constants, $I_{0}$ and $K_{0}$ are the modified Bessel functions of the first and second kind (of order zero), respectively, and $C_{1}, C_{2}$ are the constant total substrate concentrations in $\Omega_{1}$ and $\Omega_{2}$, respectively.
\end{theorem}
\begin{proof}
	The derivation of the explicit solutions is straightforward. Observe that adding the ODEs for $u$ and $v$ (similarly, for $w$ and $z$) results to the ODE $y'' + \frac{1}{r}y' = 0$ for $y = d_{3}u + d_{2}v$ with fundamental solutions $\ln r$ and one. Since we want non-negative solutions, we disregard $\ln r$. Hence, we obtain the scaled conservation of total concentrations in $\Omega_{1}$ and $\Omega_{2}$
	\begin{equation} \label{eqn:ExactSolutions_MassCon1}
		d_{3}u(r) + d_{2}v(r) = C_{1} \qq{and} d_{3}w(r) + d_{2}z(r) = C_{2}.
	\end{equation}

	We now solve the ordinary differential system for $u$ and $v$. The general solution of the ODE for $v$ is given by $v(r) = k_{1} I_{0}(\alpha r) + \widetilde{k}_{1}K_{0}(\alpha r)$, where $\alpha^{2} = \kappa/d_{2}$, $I_{j}$ and $K_{j}$ are the \textit{modified Bessel functions} \cite[Section 10.25]{olver2010} of order $j \in \mathbb{Z}$ (first kind and second kind, respectively). Since $K_{0}$ is unbounded at the origin, we take $\widetilde{k}_{1} = 0$ so that we obtain $v(r) = k_{1} I_{0}(\alpha r)$. The constant $k_{1}$ can be explicitly computed using the boundary condition for $v$ in terms of the value of $z$ at the interface, which we will determine later, in particular, we get $k_{1} = z(R)/[\frac{d_{2}}{\nu}\alpha I_{1}(\alpha R) + I_{0}(\alpha R)]$. Due to the scaled conservation of total concentrations in \eqref{eqn:ExactSolutions_MassCon1}, we therefore obtain the exact solutions of $u$ and $v$ in \eqref{eqn:ExactSolutions_uv}. Solving the ODE for $w$ and arguing analogously as above, we finally obtain the exact solutions for $w$ and $z$ in \eqref{eqn:ExactSolutions_wz} with $\beta^{2} = \rho /d_{3}$ where the constants $k_{1}$ and $k_{2}$ are explicitly given by solving the linear system
	\begin{equation*}
		\begin{bmatrix}
			A_{11} & A_{12} \\
			A_{21} & A_{22}
		\end{bmatrix} \begin{bmatrix}
			k_{1} \\ k_{2}
		\end{bmatrix} = \begin{bmatrix}
			0 \\ \frac{1}{2}(C_{1} + C_{2})
		\end{bmatrix},
	\end{equation*}
	where
	\begin{align*}
		&A_{11} = d_{2} \alpha I_{1}(\alpha R), \\
		&A_{12} = d_{3} \beta K_{1}(\beta R) \qty[ \frac{I_{1}(\beta R)}{K_{1}(\beta R)} - \frac{I_{1}(\beta (R+\delta))}{K_{1}(\beta(R+\delta))}], \\
		&A_{21} = \frac{d_{3}^{2}}{\mu}\alpha I_{1}(\alpha R) + d_{2}I_{0}(\alpha R) + \frac{1}{2}d_{2}\alpha I_{1}(\alpha R) \qty(\frac{d_{3}}{\mu} - \frac{d_{2}}{\nu}), && \\
		&A_{22} = d_{3}K_{0}(\beta R) \qty[ \frac{I_{0}(\beta R)}{K_{0}(\beta R)} + \frac{I_{1}(\beta (R+\delta))}{K_{1}(\beta(R+\delta))}].
	\end{align*}
	Note that the two by two coefficient matrix above has rank two and hence, uniquely solvable for $k_{1}$ and $k_{2}$. Finally, we have $k_{3} = k_{2} \frac{I_{1}(\beta (R+\delta))}{K_{1}(\beta(R+\delta))}$.
\end{proof}

\begin{figure}[htb]
	\centering
	\includegraphics[width=0.9\textwidth]{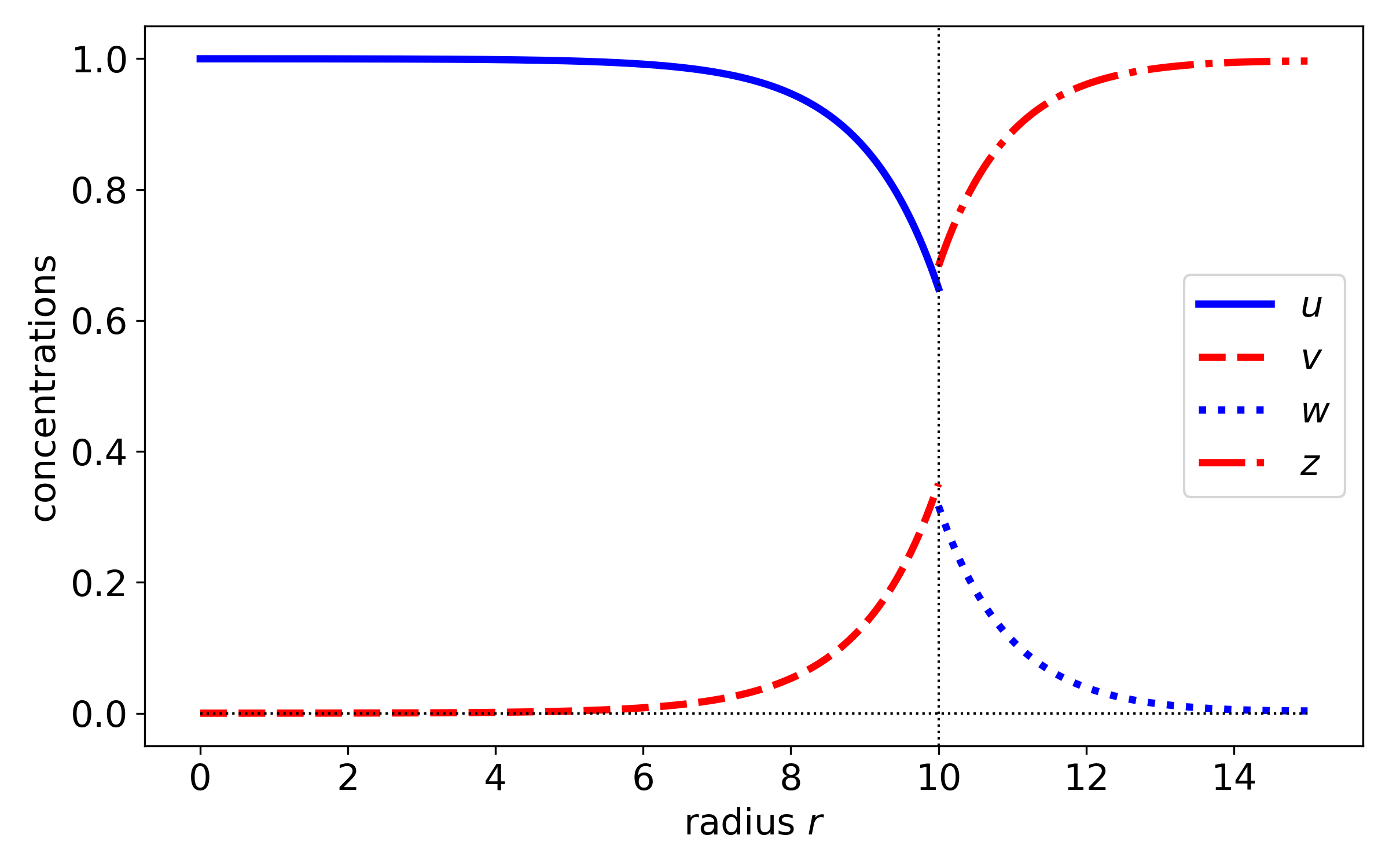}
	\caption{Profiles of the radially symmetric exact solutions in the reservoir and active regions. The triglyceride concentrations $u$ and $w$ are decreasing as the radius increases while the diglyceride concentrations $v$ and $z$ are increasing. The discontinuity at the interface is due to the Robin boundary condition}
	\label{fig:ExactSolutions}
\end{figure}

We now present some numerical simulations of the test model in system \eqref{eqn:EllipticModel1}. To this end, we use generic values for the constant parameters in the system, as we are more interested in the qualitative behaviours of the solutions than the quantitative. Hence, let all constant parameters ($d_{3}$, $d_{2}$, $\mu$, $\nu$, $\kappa$, $\rho $, $C_{1}$, $C_{2}$) in the system be equal to unity, the radius of the reservoir region $\Omega_{1}$ be $R = 10$, and the thickness of the active region $\Omega_{2}$ be $\delta = 5$. We implement the numerical simulations in \texttt{FEniCS}, a Python-based finite element computing software, using linear Lagrange basis functions $\mathbb{P}^{1}$. Figure~\ref{fig:ExactSolutions} and Figure~\ref{fig:FEMSolutions} show the profiles of the radially symmetric explicit solution and the finite element solution in two dimensions, respectively.
\begin{figure}[htb]
	\centering
	\includegraphics[width=0.5\linewidth]{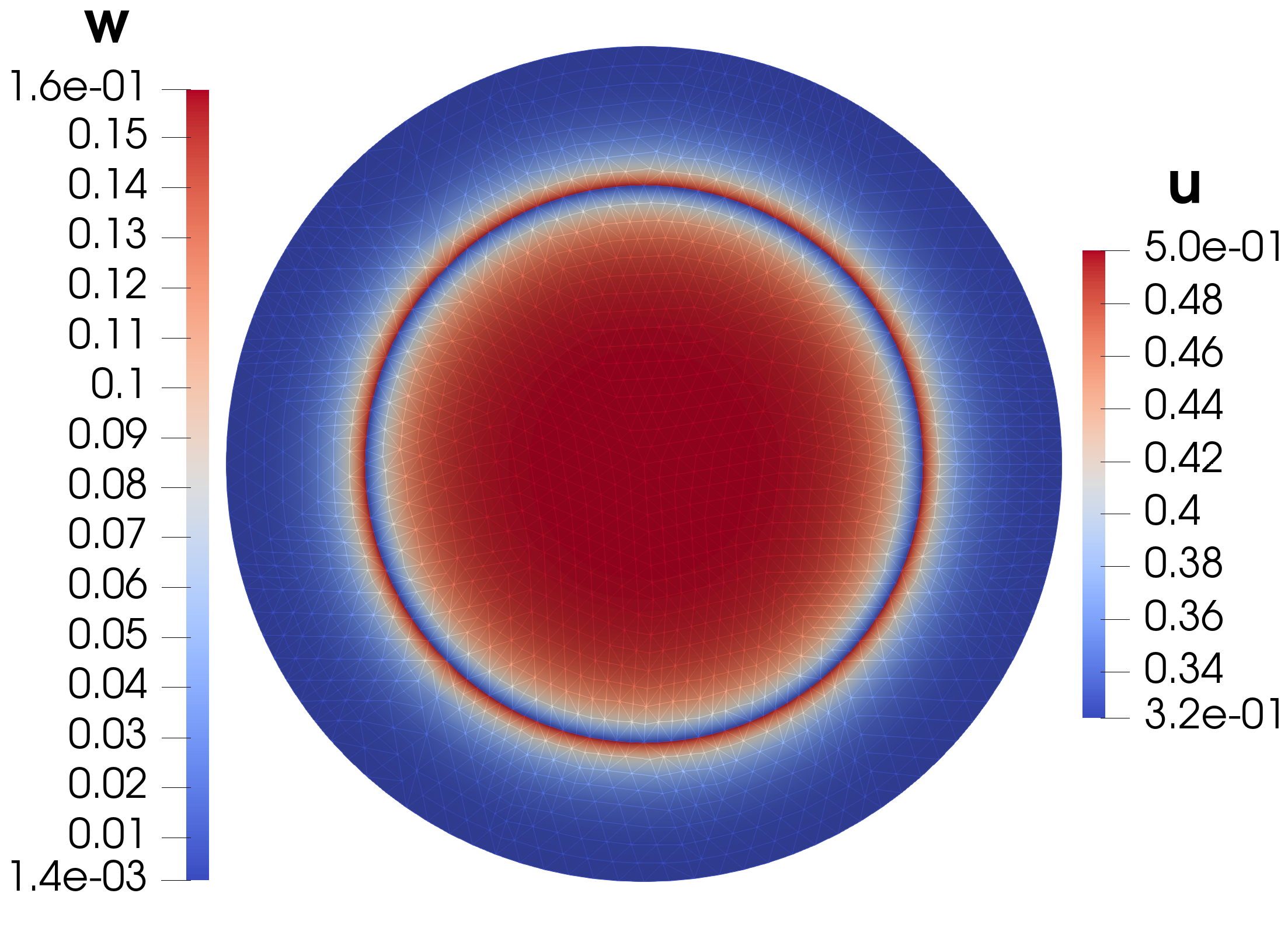}\includegraphics[width=0.5\linewidth]{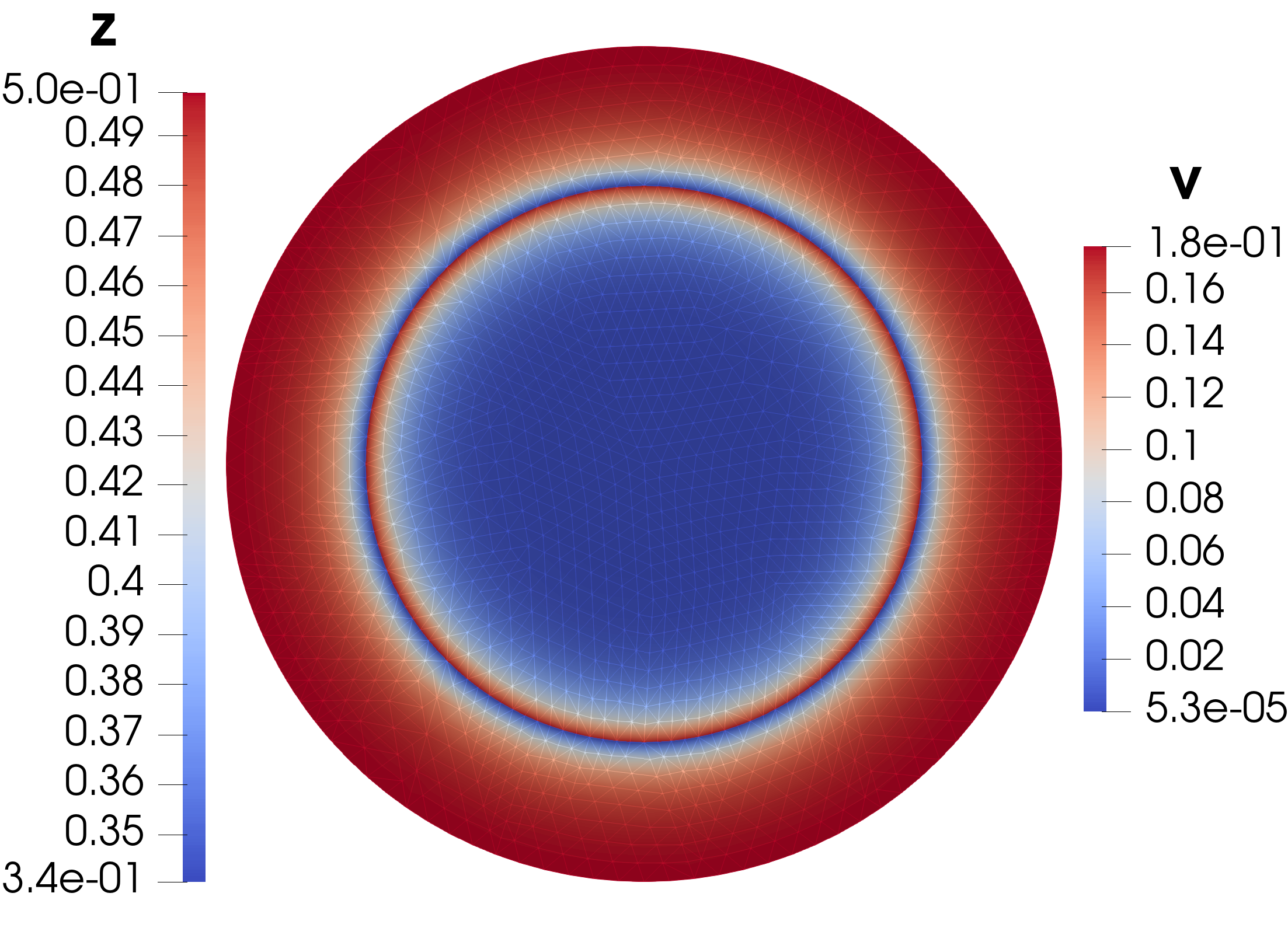}
	\caption{The finite element solution of the linear elliptic problem exhibits the same monotonicity and symmetry properties of the exact solutions. Note that the colour scales are different for each function}
	\label{fig:FEMSolutions}
\end{figure}

\begin{remark}\label{rem:ndof-mesh}
	In two-dimensions, we expect the estimated order of convergence in the $H^{1}$- and $L^{2}$-norms with respect to the number of degrees of freedom (\verb|ndof|) to be $k/2$ and $(k+1)/2$, respectively. Here, $k$ is the polynomial degree of the basis functions.
\end{remark}

We also obtain a numerical verification of the estimated order of convergence of the discretisation errors in the $H^{1}$- and $L^{2}$-norms. Figure~\ref{fig:ErrorsL2H1_rhalf} shows the quadratic convergence in the $L^{2}$- and linear convergence in the $H^{1}$-norms of the discretisation error for each of the unknown variables with respect to the mesh size. The subfigures show the discretisation error plotted against the mesh size, both in logarithmic scales. In each subfigure, we observe from the top right to the bottom left a decreasing error with decreasing mesh size. In terms of the number of degrees of freedom (\verb|ndof|), this corresponds to the estimated order of convergence of $0.5$ and $1.0$ with respect to the $H^{1}$- and $L^{2}$-norms, respectively (see Remark~\ref{rem:ndof-mesh}).
\begin{figure}[htb]
	\centering
	\includegraphics[width=0.9\linewidth]{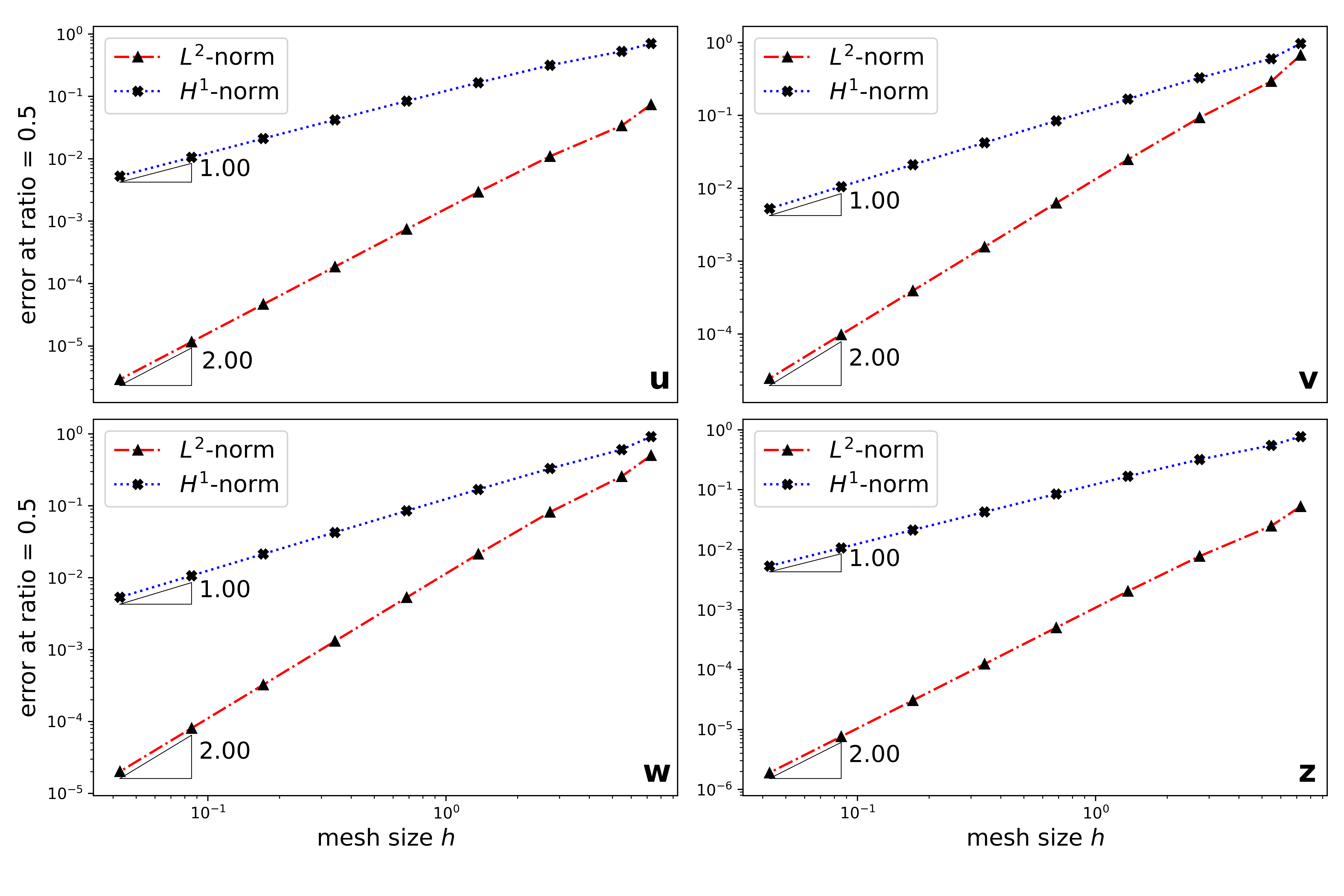}
	\caption{At ratio $0.5$, we observe the linear and quadratic order of convergence for the discretisation errors in $H^{1}$- and $L^{2}$-norms against the mesh size, resp. for all variables}
	\label{fig:ErrorsL2H1_rhalf}
\end{figure}

\begin{figure}[htb]
	\centering
	\includegraphics[width=0.9\linewidth]{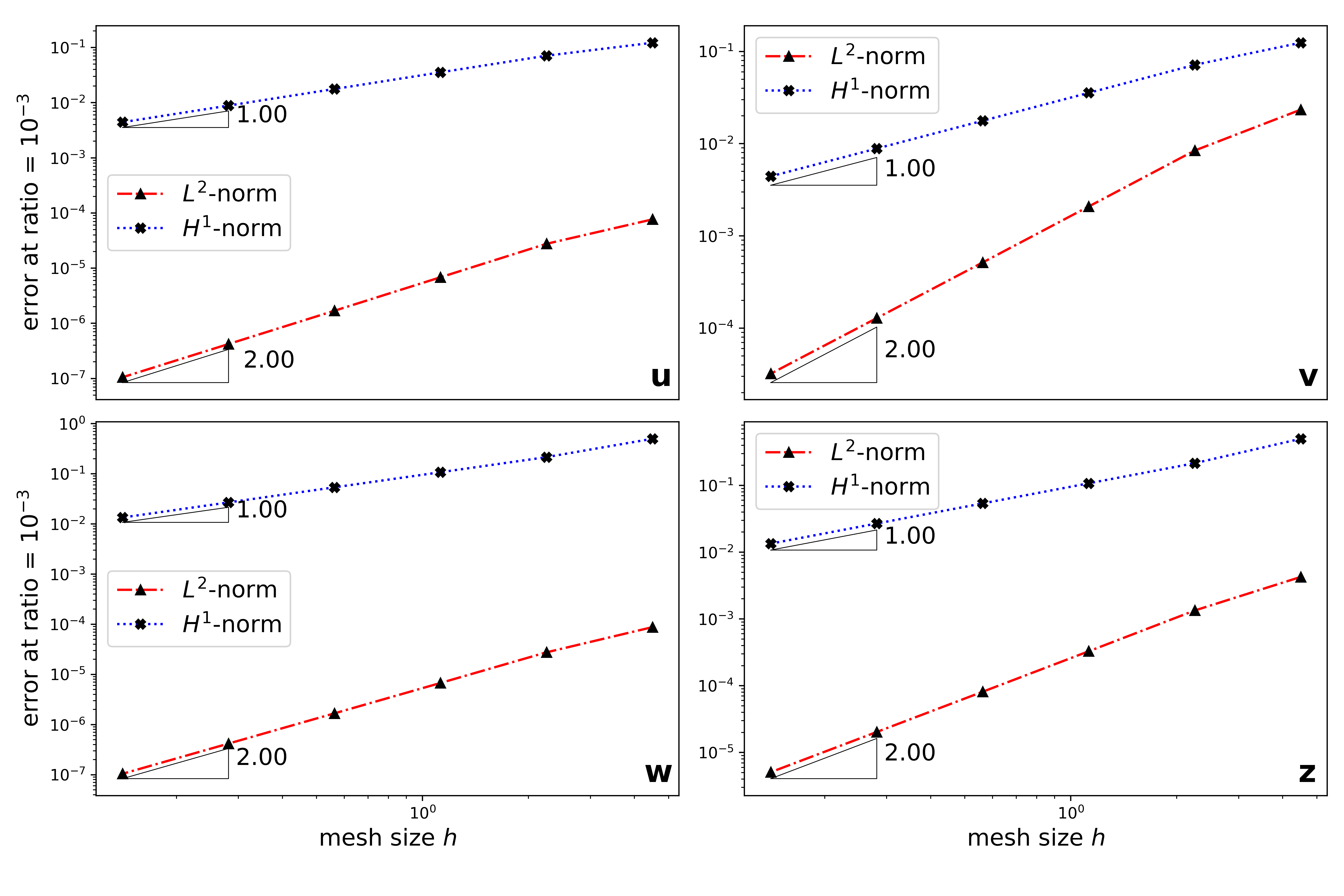}
	\caption{For the ratio $\delta/R$ of order $10^{-3}$, we still observe linear convergence for $H^{1}$- and quadratic convergence for $L^{2}$-norm of the discretisation errors against the mesh size}
	\label{fig:ErrorsL2H1_ratio1e-3}
\end{figure}

Next, we investigate the case when the ratio of the annular thickness $\delta$ and the radius $R$ of the reservoir region decreases. Here, we fix the radius $R=10$ (for consistency of plotting) and decrease the annular thickness $\delta$ down to a ratio of $10^{-3}$, which corresponds to the nondimensionalised geometry of LDs of size 5 $\mu$m in relation to a hypothesised molecular reach of ATGL of about 5 nm. In Figure~\ref{fig:ErrorsL2H1_ratio1e-3}, we still obtain the quadratic and linear convergence in the $L^{2}$- and $H^{1}$-norm of the discretisation error, respectively. The numerical solutions are shown in Figure~\ref{fig:uv_sol_r3}, which have the same qualitative behaviours as in the numerical solutions in Figure~\ref{fig:FEMSolutions}.

\begin{figure}[htb]
	\centering
	\includegraphics[width=0.45\linewidth]{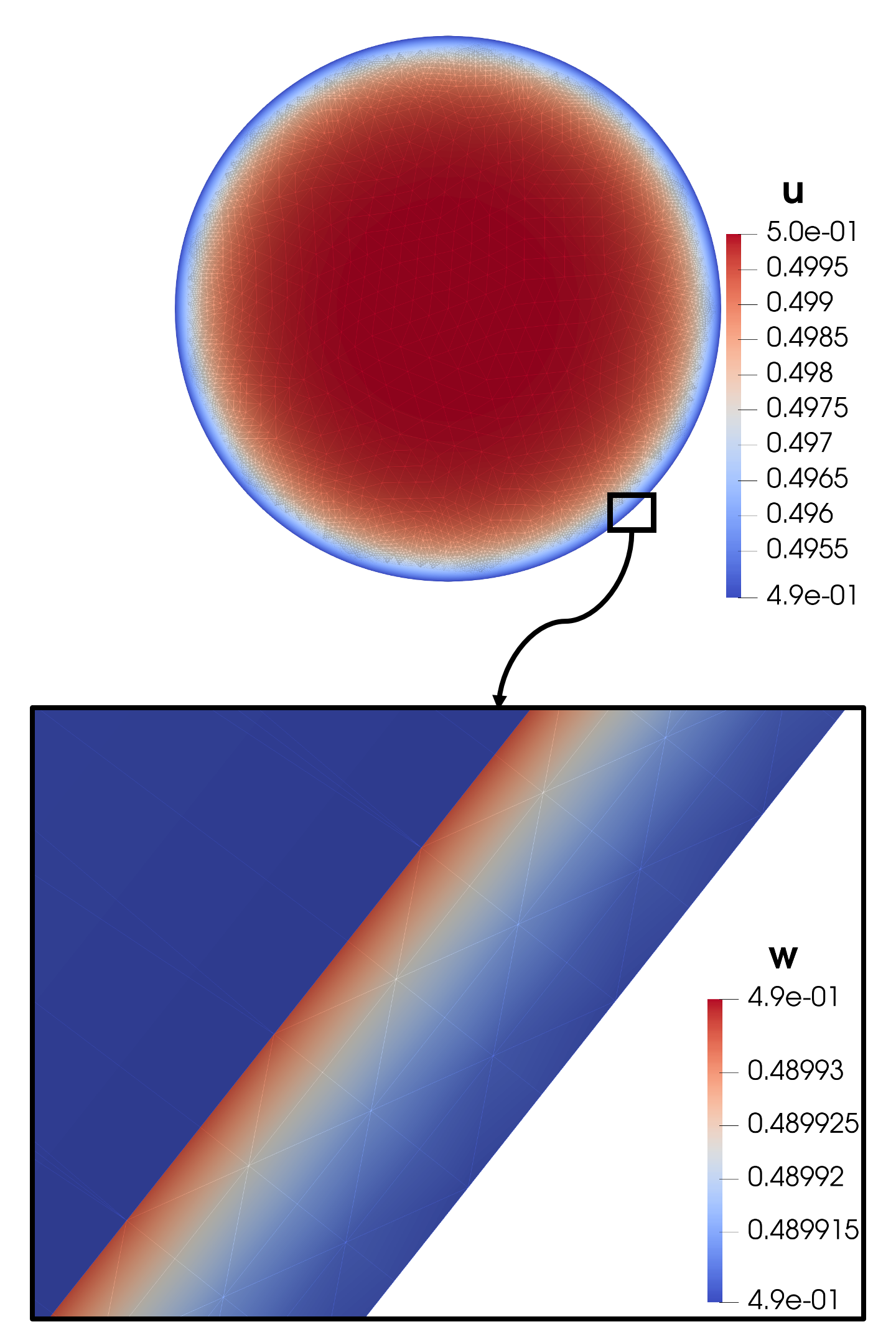}\includegraphics[width=0.45\linewidth]{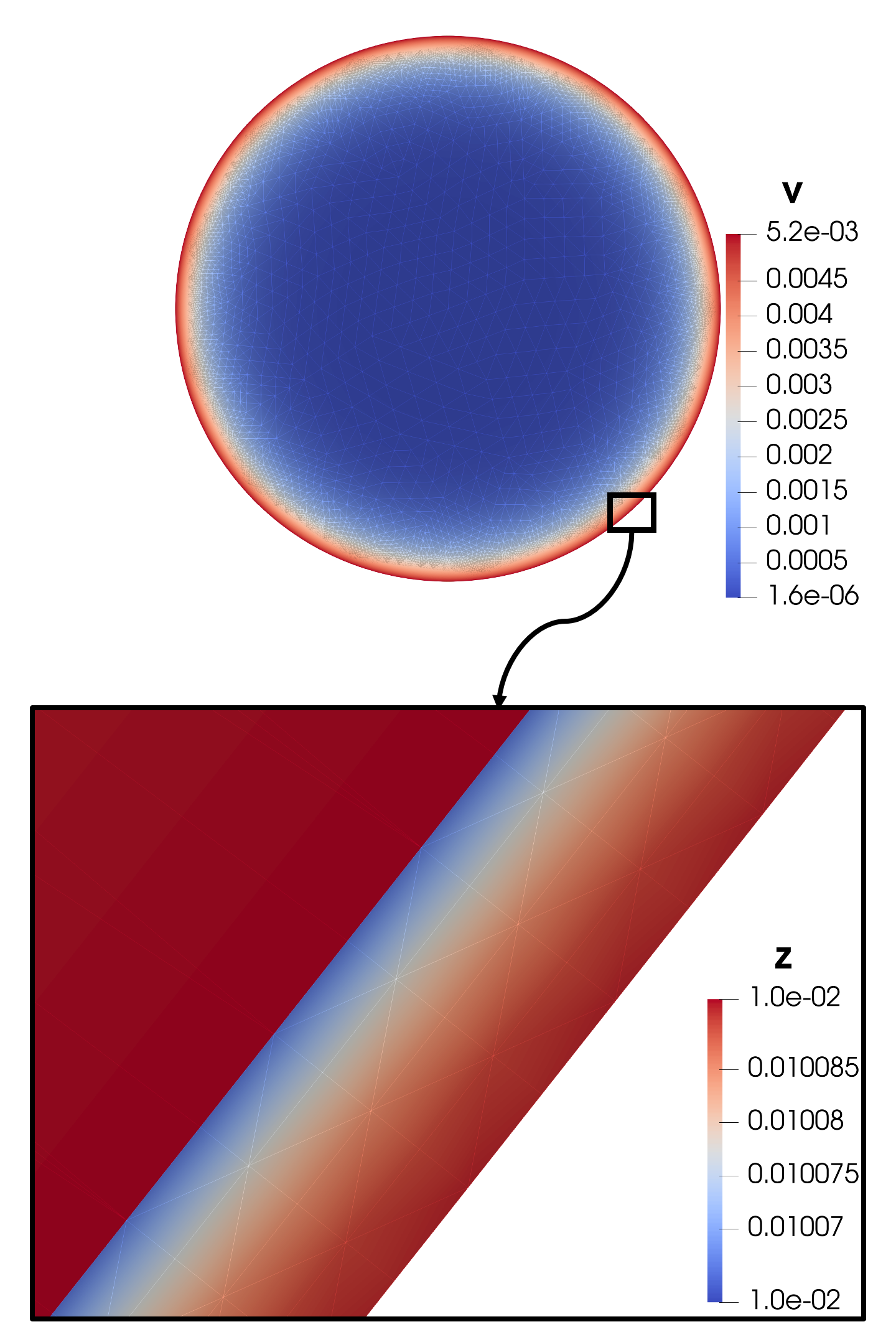}
	\caption{Numerical solutions for the ratio $\delta/R$ of order $10^{-3}$. For $w$ and $z$, only a portion of the strip is shown due to the thinness of the active region}
	\label{fig:uv_sol_r3}
\end{figure}

To demonstrate the exponential convergence of the global solutions to the equilibrium states, we implement a backward Euler time discretisation and then a classical finite element method to solve the resulting linear elliptic problem. In Figure~\ref{fig:entropy}, the exponential decay of the relative entropy functional $\mathcal{E}(\mathbf{c} - \mathbf{c}_{\infty}|\mathbf{c}_{\infty})$ in logarithmic scale is plotted against time. Here, we still use the generic values for the constant parameters with the addition that the parabolic system is solved on the time interval $[0, 10]$ with step size $0.01$. We observe that the entropy functional $\mathcal{E}$ approaches the graph of the linear function $y = -\lambda t + C$ as predicted in Theorem~\ref{thm:ExponentialConvergence}. For the given set of constant parameters, we have $\lambda \approxeq 1.79$ and $C \approxeq 107.6$.
\begin{figure}[htb]
	\centering
	\includegraphics[width=0.8\linewidth]{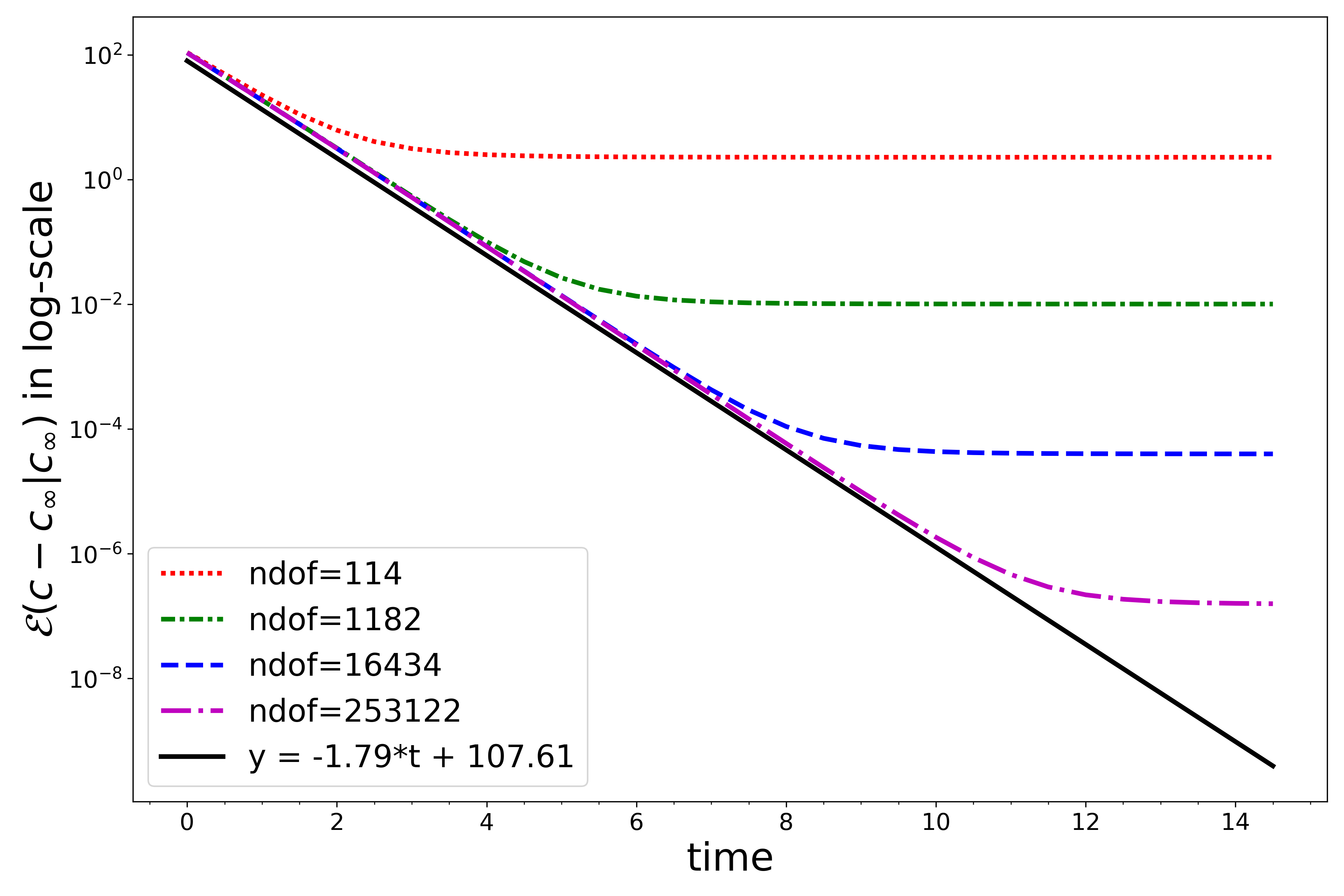}
	\caption{The relative entropy functional $\mathcal{E}(\mathbf{c} - \mathbf{c}_{\infty}|\mathbf{c}_{\infty})$ is exponentially decaying through time. Observe that the exponential decay is independent of the number of degrees of freedom (and hence, of mesh size)}
	\label{fig:entropy}
\end{figure}

Due to the geometry of the computational domain, and for comparison purposes, we also implemented a numerical scheme using \verb|GeoPDEs|, an open source and free computing Matlab package, which uses a method called \textit{isogeometric analysis} \cite{vazquez2016,defalco2011}. A detailed introduction into the method can be found in \cite{hughes2009}. The main difference of isogeometric analysis to the classical finite element method is the use of NURBS \textit{(Non-Uniform Rational B-Splines)} as basis functions that can exactly represent conical geometries. With this, our computational domains (reservoir and active regions) are exactly constructed while in the classical finite element implementation, these domains are only approximated (by triangulations) using polynomials as basis functions.

With appropriate change of basis functions, we solve the discrete variational problem \eqref{eqn:VarProb_Elliptic_Discrete} by considering it as an eigenproblem: we want to find the eigenfunction corresponding to the eigenvalue zero. Note that the eigenfunction is only determined up to a multiplicative factor, which we fix by imposing the total mass conservation law in \eqref{eqn:Discrete_MassCon} to obtain the solution.

For the rest of this section, we now use \verb|GeoPDEs| in the numerical simulations. We expect to obtain a better rate of convergence as the polynomial degree of the basis functions increases. In Figure~\ref{fig:ord_h1_err_all_ndof} and Figure~\ref{fig:ord_l2_err_all_ndof}, we observe that the estimated order of convergence increases by approximately one-half as the polynomial degree increases by one. The figures show the $H^{1}$- and $L^{2}$-norms of the discretisation error plotted against the number of degrees of freedom for polynomial degrees $k=1,2,3$. With respect to the mesh size, the estimated order of convergence for the $H^{1}$-norm error is therefore equal to $1$, $2$, $3$ while for the $L^{2}$-norm error, we have $2$, $3$, and $4$, for polynomial degrees $k=1, 2, 3$, respectively (see Remark~\ref{rem:ndof-mesh}). For $k=1$, we reiterate the results in Figure \ref{fig:ErrorsL2H1_rhalf} and Figure~\ref{fig:ErrorsL2H1_ratio1e-3}.
\begin{figure}[htb]
	\centering
	\includegraphics[width=0.9\linewidth]{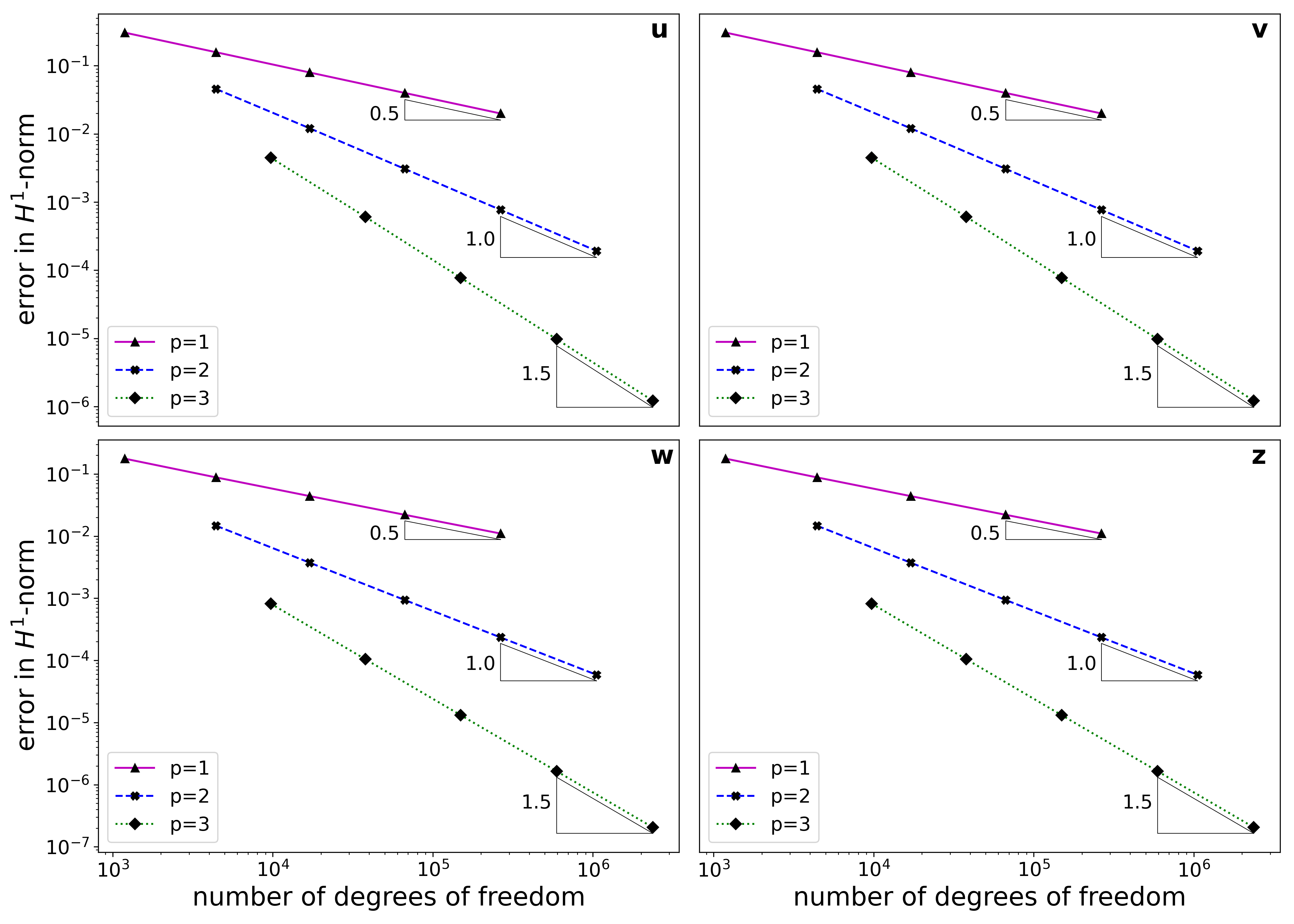}
	\caption{The order of convergence for the discretisation error in the $H^{1}$-norm plotted against the number of degrees of freedom increases by one-half as the degree of the NURBS basis functions increases by one}
	\label{fig:ord_h1_err_all_ndof}
\end{figure}

\begin{figure}[htb]
	\centering
	\includegraphics[width=0.9\linewidth]{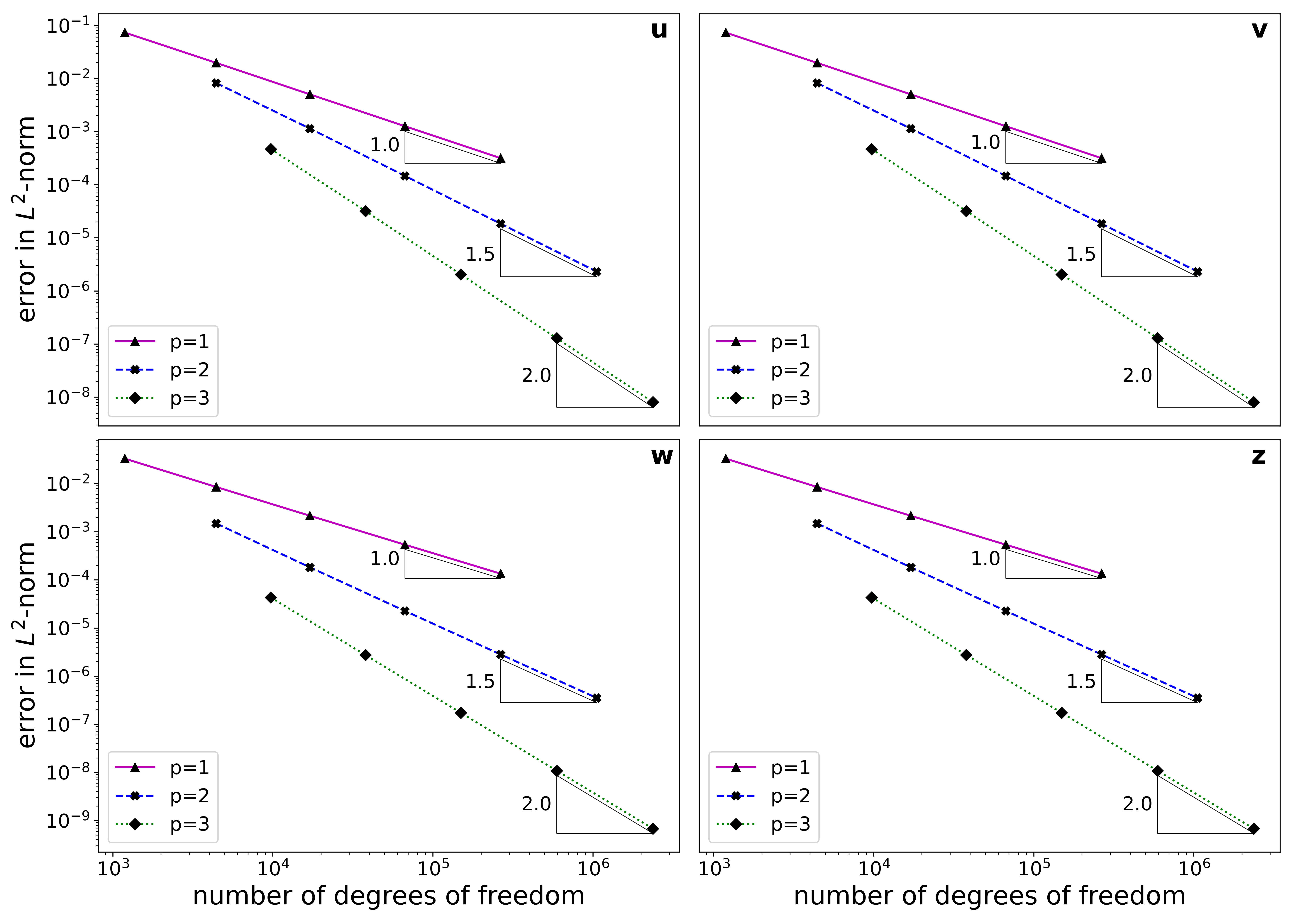}
	\caption{The order of convergence for the discretisation error in the $L^{2}$-norm plotted against the number of degrees of freedom increases by one-half as the degree of the NURBS basis functions increases by one}
	\label{fig:ord_l2_err_all_ndof}
\end{figure}

Finally, we present numerical simulations for the nonlinear elliptic system where the forward process of TG degradation into DG and FA is modelled by the Michaelis-Menten kinetics. In particular, we consider the variational problem: find $\mathbf{c}_{h} \in \mathbf{V}_{h}^{k}$ such that it holds
\begin{equation}\label{eqn:VarProb_Elliptic_Discrete_Nonlinear}
	\begin{split}
		\mathcal{A}_{h}(\mathbf{c}_{h}, \Phi) &= d_{3} (\nabla u_{h}, \nabla \varphi_{1})_{\Omega_{1,h}} + d_{2} (\nabla v_{h}, \nabla \varphi_{2})_{\Omega_{1,h}} \\
		&+ d_{3} (\nabla w_{h}, \nabla \varphi_{3})_{\Omega_{2,h}} + d_{2} (\nabla z_{h}, \nabla \varphi_{4})_{\Omega_{2,h}} \\
		&+ \kappa (v_{h}, \varphi_{2} - \varphi_{1})_{\Omega_{1,h}} + \qty(\frac{v_{3} w_{h}}{k_{3} + w_{h}}, \varphi_{3} - \varphi_{4})_{\Omega_{2,h}} \\
		&+ \mu(w_{h} - u_{h}, \varphi_{3} - \varphi_{1})_{\Gamma_{h}} + \nu(z_{h} - v_{h}, \varphi_{4} - \varphi_{2})_{\Gamma_{h}} \\
		&= 0 \qq{for all test functions} \Phi \in \mathbf{V}_{h}^{k},
	\end{split}
\end{equation}
and in addition, the discrete total mass conservation law in \eqref{eqn:Discrete_MassCon} holds. We solve the nonlinear system \eqref{eqn:VarProb_Elliptic_Discrete_Nonlinear} using a simple Picard iteration method. In particular, we linearise system \eqref{eqn:VarProb_Elliptic_Discrete_Nonlinear} by replacing the unknown state in the denominator of the Michaelis-Menten term with a previously known solution, more precisely, at iteration step $n$,  we take $\frac{v_{3} w_{h}^{n}}{k_{3} + w_{h}^{n-1}}$ where $w_{h}^{n-1}$ is a known solution from the previous iteration. Figure~\ref{fig:r3_nonlin_err} shows the convergence of the solution after twelve iterations using the $H^{1}$- and $L^{2}$-norms of the difference of the previous $\mathbf{c}_{h}^{n-1}$ and the current $\mathbf{c}_{h}^{n}$ solutions. Here, we initialise the iteration by first solving the linear problem in system \eqref{eqn:VarProb_Elliptic_Discrete} and then taking the solution $w_{h}^{0}$ as the initial iterate in the nonlinear system~\eqref{eqn:VarProb_Elliptic_Discrete_Nonlinear}.

\begin{figure}[htb]
	\centering
	\includegraphics[width=0.9\linewidth]{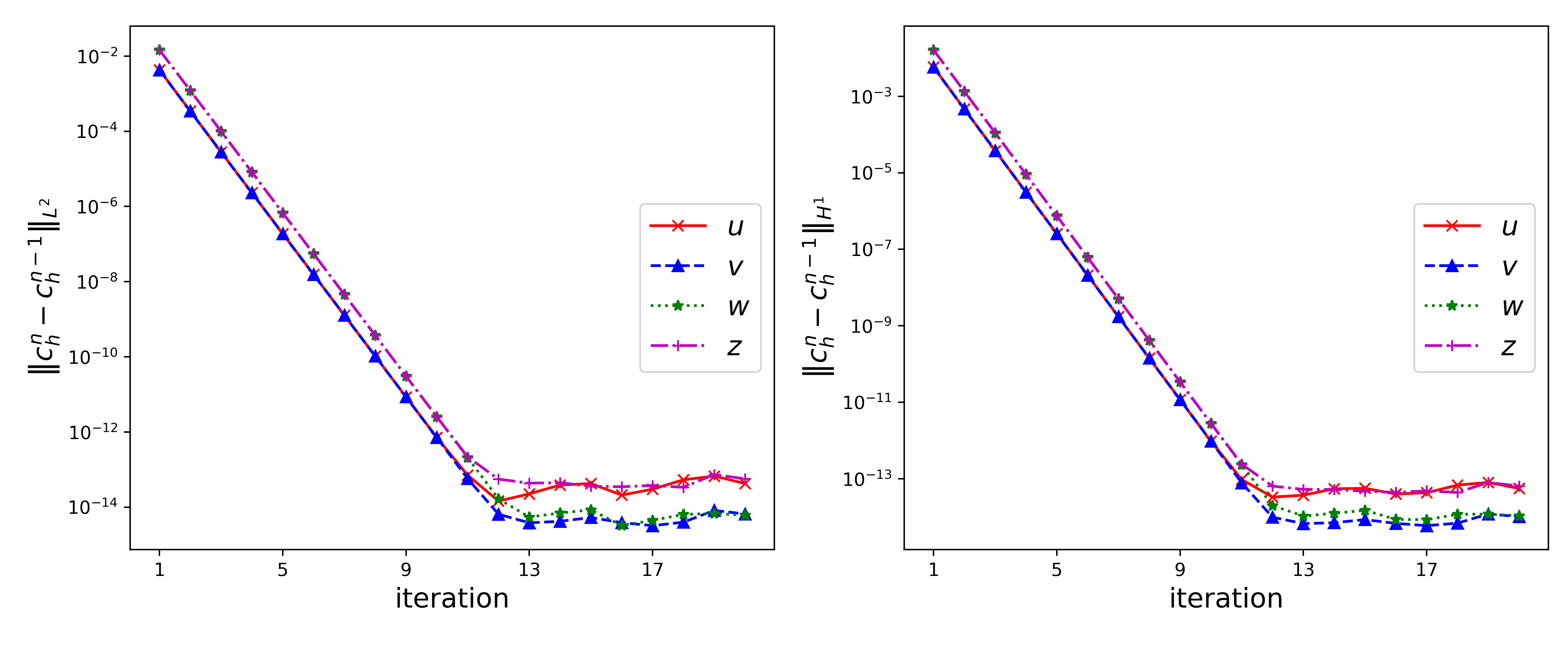}
	\caption{Convergence of the Picard iteration for the nonlinear elliptic system at ratio = 0.5 in the $L^{2}$- and $H^{1}$-norms. Observe that we obtain convergence after twelve iterations for all unknown states}
	\label{fig:r3_nonlin_err}
\end{figure}

\section{Effects of ATGL Clustering on Lipolysis}\label{sec:clustering}

In this section, we illustrate the effects of ATGL clustering on the surface of LDs via numerical examples. ATGL molecules are known to be able to form sizeable clusters and there is experimental evidence of spatially heterogeneous distributions of ATGL on LDs \cite{Kulminskaya-2021,Padmanabha-2018}. However, a quantitative assessment of the effects of ATGL heterogeneities to lipolysis is far beyond the currently available experimental protocols, which are only able to quantitatively measure substrate and enzyme concentrations of entire populations of LDs.

In the following, we present a prototypical model for lipolysis, which processes TGs over DGs and MGs to GLs and three FAs according to hydrolysis in Figure~\ref{fig:lipolysis} and transacylation in Figure~\ref{fig:transacylation}. The PDE model reads for all $t \in (0, T]$ and some arbitrary observation time $T > 0$
\begin{equation}\label{eqn:ParabolicModelRealistic}
	\begin{split}
		\begin{cases}
			\pdv{}{t} q_{3i} - d_{3} \Delta q_{3i}  = 0, & x \in \Omega_{1}, \\
			d_{3} \partial_{\eta} q_{3i} = \mu[q_{3} - q_{3i}], & x \in \Gamma, \\
			\pdv{}{t} q_{3} - d_{3} \Delta q_{3} = - \dfrac{v_{3}a_{3}(x) q_{3}}{k_{3} + q_{3}} + \sigma a_{3}(x)q_{2}^{2}, & x \in \Omega_{2}, \\
			d_{3} \partial_{\eta} q_{3} = \mu[q_{3} - q_{3i}], & x \in \Gamma, \\
			d_{3} \partial_{\eta} q_{3} = 0, & x \in \partial\Omega, \\
			\pdv{}{t} q_{2} - d_{2} \Delta q_{2} = +\dfrac{v_{3}a_{3}(x) q_{3}}{k_{3} + q_{3}} - \dfrac{v_{2} e_{2} q_{2}}{k_{2} + q_{2}} - 2\sigma a_{3}(x) q_{2}^{2}, & x \in \Omega_{2}, \\
			d_{2} \partial_{\eta} q_{2} = 0, & x \in \partial \Omega, \\
			d_{2} \partial_{\eta} q_{2} = 0, & x \in \Gamma, \\
			\pdv{}{t} q_{1} - d_{1} \Delta q_{1} = -\dfrac{v_{1} m_{1} q_{1}}{k_{1} + q_{1}} + \dfrac{v_{2} e_{2} q_{2}}{k_{2} + q_{2}} + \sigma a_{3}(x) q_{2}^{2}, & x \in \Omega_{2}, \\
			d_{1} \partial_{\eta} q_{1} = 0, & x \in \Gamma, \\
			d_{1} \partial_{\eta} q_{1} = 0, & x \in \partial\Omega, \\
			\pdv{}{t} q_{0} - d_{0} \Delta q_{0} = +\dfrac{v_{1} m_{1} q_{1}}{k_{1} + q_{1}}, & x \in \Omega_{2}, \\
			d_{0} \partial_{\eta} q_{0} = 0, & x \in \Gamma, \\
			d_{0} \partial_{\eta} q_{0} = 0, & x \in \partial\Omega,
		\end{cases}
	\end{split}
\end{equation}
where $q_{3i}$ and $q_3$ denote the TG concentrations in the reservoir region $\Omega_{1}$ and the active layer $\Omega_{2}$. The exchange of TGs over the interface between reservoir region and active layer is governed by the constant flux rate $\mu > 0$ and the TG diffusion constant $d_{3} > 0$. Moreover, $q_{2}$, $q_{1}$ and $q_{0}$ are the concentrations of DGs, MGs, and GLs in the active region with constant diffusion coefficients $d_{2}$, $d_{1}$, $d_{0} > 0$.

Note that DGs and MGs are confined to the active region by homogeneous Neumann boundary conditions as we assume that the increased polarity of DGs and MGs compared to the unpolar TGs (with their symmetrical configuration of three FAs) prevents them from leaving the active region near the surface of the LD. Hence, the reservoir region is occupied only by TGs. This matches the observations that \textit{in vivo} DGs and MGs are negligible inside LDs except under pathological conditions.

Next, $v_{3}$, $k_{3}>0$ are Michaelis-Menten kinetic parameters of TG hydrolysis catalysed by the AGTL concentration $a_{3}(x)$, for which we detail below three different spatial distributions. The AGTL concentration $a_3(x)$ also facilitates DG transacylation (recall Figure~\ref{fig:transacylation}), which is modelled according to the mass action law (MAL) with a parameter $\sigma>0$. Finally, the constant parameters $v_{2}$, $k_{2}>0$ and $v_{1}$, $k_{1}>0$ denote the Michaelis-Menten kinetic parameters of DG hydrolysis catalysed by the HSL concentration $e_{2} > 0$ and MG hydrolysis catalysed by MGL with concentration $m_{1} > 0$. Note that we do not consider HSL and MGL to cluster.

System \eqref{eqn:ParabolicModelRealistic} is subject to the TG only initial data
\begin{equation} \label{eqn:ParabolicModelRealistic_Init}
	\begin{cases}
		q_{3i}(x, 0) = q_{3i}^{0}(x), & x \in \Omega_{1}, \\
		q_{3}(x, 0) = q_{3}^{0}(x), & x \in \Omega_{2}, \\
		q_{2}(x, 0) = 0, & x \in \Omega_{2}, \\
		q_{1}(x, 0) = 0, & x \in \Omega_{2}, \\
		q_{0}(x, 0) = 0, & x \in \Omega_{2}.
	\end{cases}
\end{equation}
Observe that solutions to system \eqref{eqn:ParabolicModelRealistic} conserve the total amount of glycerol, that is for all $0 < t \leq T$
\begin{multline}
	\int_{\Omega_{1}} q_{3i}(x,t) \dd{x} + \int_{\Omega_{2}} [ q_{3}(x,t) + q_{2}(x,t) + q_{1}(x,t) + q_{0}(x,t)] \dd{x} \\
	= \int_{\Omega_{1}} q_{3i}^{0}(x) \dd{x} + \int_{\Omega_{2}} q_{3}^{0}(x) \dd{x}.
\end{multline}

For the ATGL concentration $a_{3}(x)=a_{3}(r, \theta)$, we use the following three radial Gaussian profiles to study the effects of clustering: the radially homogeneous distribution,
\begin{equation*}
	a_{3}^{(1)}(r,\theta) = \frac{1}{s\sqrt{2\pi}} e^{-\frac{1}{2}\qty(\frac{r - \overline{r}}{s})}, \quad \theta \in [0, 2\pi],
\end{equation*}
the two clusters distribution,
\begin{equation*}
	a_{3}^{(2)}(r,\theta) =
	\begin{cases}
		\frac{\alpha_{1}}{s\sqrt{2\pi}} e^{-\frac{1}{2}\qty(\frac{r - \overline{r}}{s})}, & \theta \in [-\pi/8, \pi/8] \cup [7\pi/8, 9\pi/8], \\
		0, & \text{elsewhere},
	\end{cases}
\end{equation*}
and the four clusters distribution,
\begin{equation*}
	a_{3}^{(4)}(r,\theta) =
	\begin{cases}
		\frac{\alpha_{2}}{s\sqrt{2\pi}} e^{-\frac{1}{2}\qty(\frac{r - \overline{r}}{s})}, & \theta \in \begin{aligned}
			& [0, \pi/4] \cup [\pi/2, 3\pi/4] \cup \\
			& [\pi, 5\pi/4] \cup [3\pi/2, 7\pi/4],
		\end{aligned} \\
		0, & \text{elsewhere},
	\end{cases}
\end{equation*}
where  $\alpha_{1}$, $\alpha_{2}$ are chosen so that the total amount of ATGL is equal, i.e.
\[ \int_{\Omega_{2}} a_{3}^{(1)} \dd{x} = \int_{\Omega_{2}} a_{3}^{(2)} \dd{x} = \int_{\Omega_{2}} a_{3}^{(4)} \dd{x}.
\]

\begin{figure}[htb]
	\centering
	\includegraphics[width=0.33\linewidth]{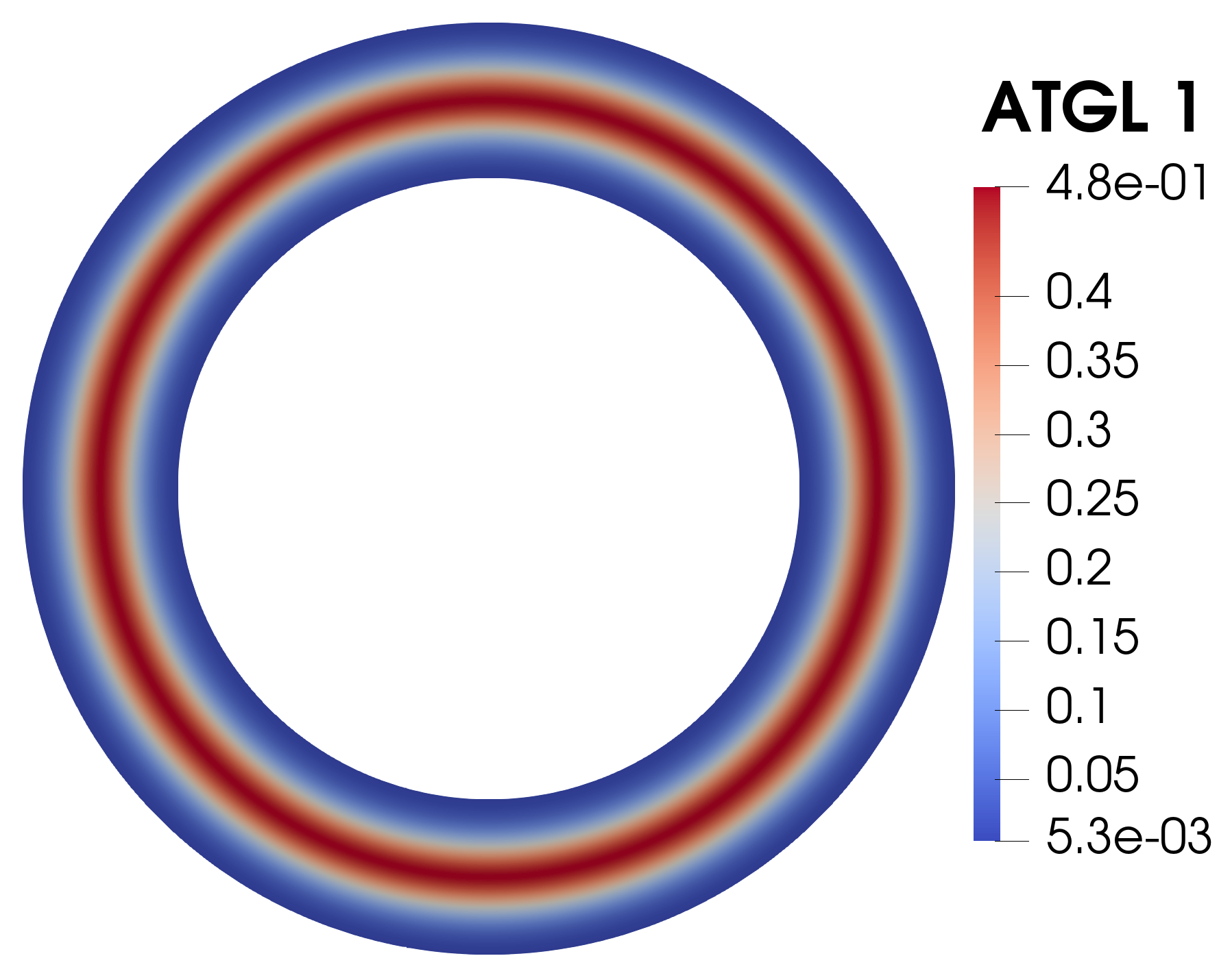}\includegraphics[width=0.33\linewidth]{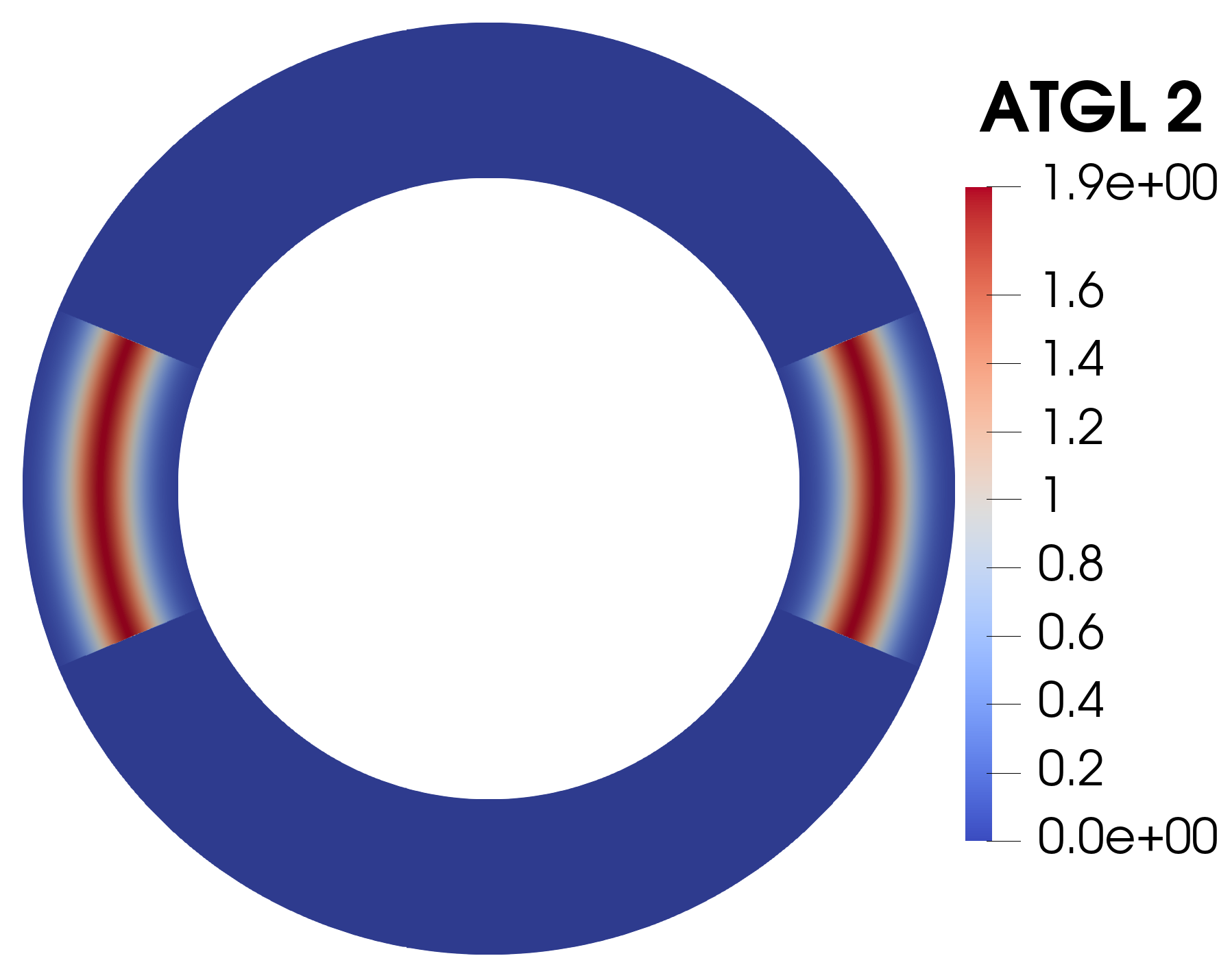}\includegraphics[width=0.33\linewidth]{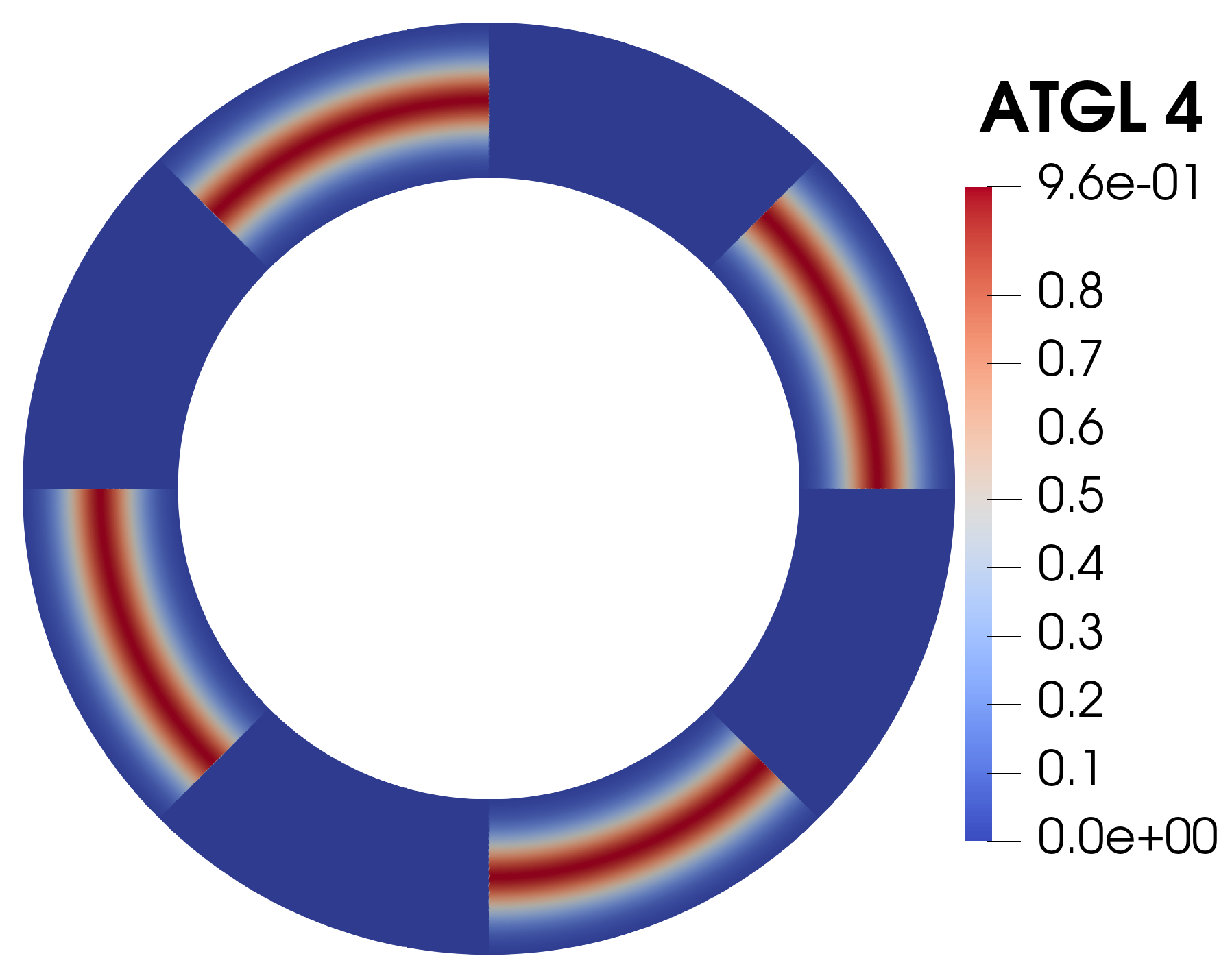}
	\caption{The three ATGL concentration profiles $a_{3}^{(1)}$, $a_{3}^{(2)}$, and $a_{3}^{(4)}$ in the active region}
	\label{fig:atgl_profiles}
\end{figure}

Figure~\ref{fig:atgl_profiles} plots the three ATGL distributions in the active region using $r \in [R, R+\delta] = [10, 15]$, $\overline{r} = 25/2$, and $s = 5/6$. We also use these parameter values in the succeeding numerical simulations.

Moreover in the numerics, all other parameters of the model \eqref{eqn:ParabolicModelRealistic} are normalised to unity with the exception of the Michaelis-Menten velocities $v_2$ and $v_1$ of HSL and MGL, which are set to two and three, respectively, following a suggestion of biochemists. We also remark that system \eqref{eqn:ParabolicModelRealistic} is solved by a backward Euler time discretisation and by a one-step Picard linearisation to the resulting nonlinear elliptic system. The resulting stationary system is therefore linear and is implemented in \verb|FEniCS|.

In the following, we study the effects of ATGL clustering by comparing the activity of the lipolytic process, which we measure by the percentage of MGs produced since the beginning in relation to the conserved total amount of glycerol in the system. Note that in the lipolytic cascade TG $\to$ DG $\to$ MG $\to$ GL, MGs are the first downstream product unaffected by ATGL and that every glycerol being part of an initial TG will eventually be processed to be part of an MG, implying that the percentage of produced MGs will always reach 100\% in the large time limit.

\begin{figure}[htb]
	\centering
	\includegraphics[width=0.9\linewidth]{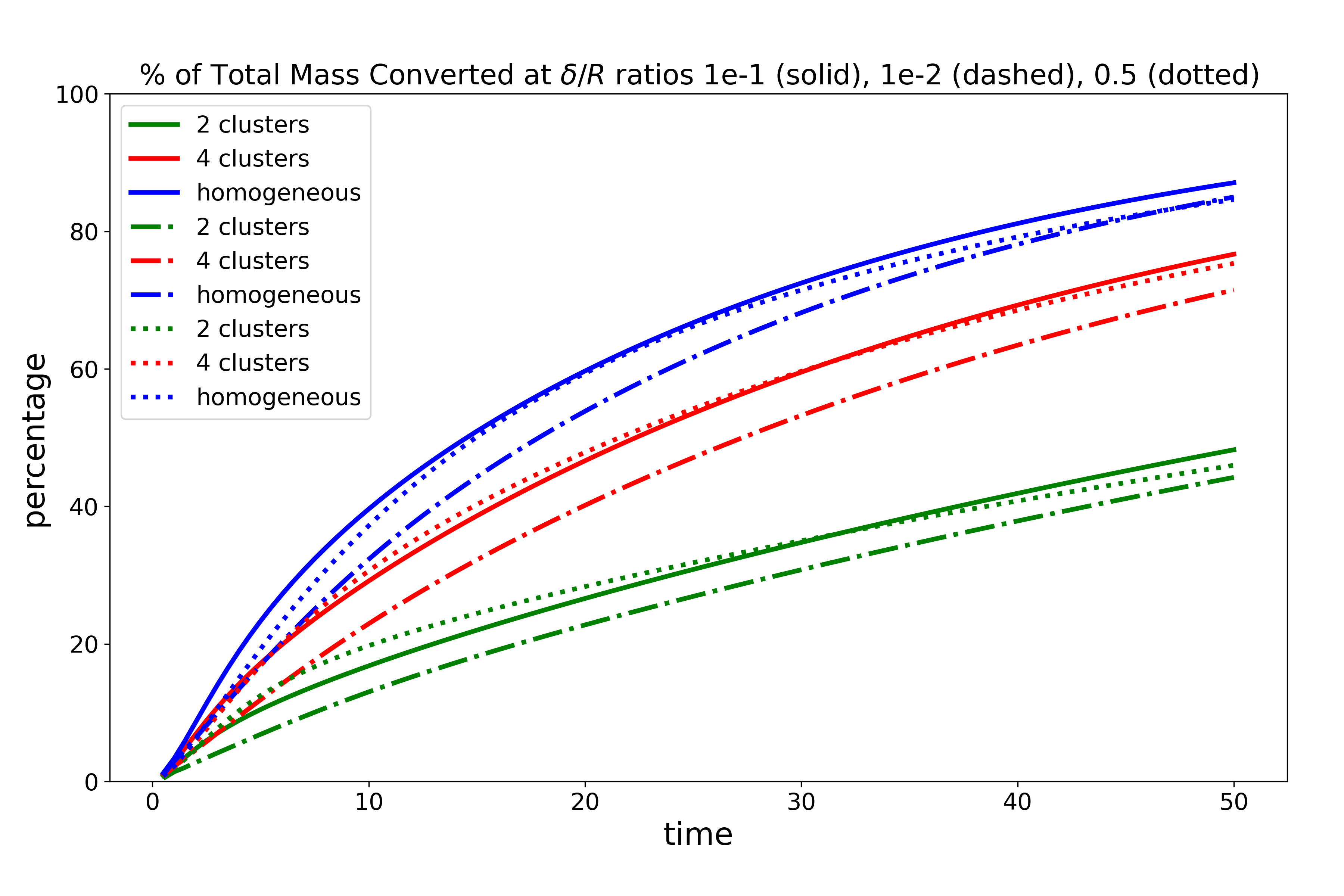}
	\caption{A comparison of the evolution of the percentage of produced MGs, where the same total amount of ATGL is distributed either radially symmetric or aggregated in two resp. four clusters. The dotted, solid, and dashed lines compare different relative thicknesses of the active region}
	\label{fig:clustering}
\end{figure}

Figure~\ref{fig:clustering} compares the evolution of the percentages of produced MGs for three different spatial distributions of the same total ATGL amount and for three different relative thicknesses of the active layer $\delta/R$. The three blue lines show that the radially symmetric Gaussian ATGL profile $a_{3}^{(1)}$ yields the quickest build up of MGs. The dotted, solid, and dashed blue lines compare $\delta/R$ ratios of $0.5$, $0.1$ and $0.01$ and illustrate intricate minor variations caused by the relative thickness of the active layer. Moreover, Figure~\ref{fig:clustering} shows that the two ATGL clusters of $a_{3}^{(2)}$ yield a significantly slowed down  production of MGs (green lines), while the four ATGL clusters of $a_{3}^{(4)}$ (red lines) yield increased production, yet not as fast as in the radially symmetric case without ATGL clustering (blue lines).

The interpretation of both the effects of ATGL clustering and relative thickness of the active layer is nontrivial, which is majorly due to the different nonlinear reaction rates of DG hydrolysis (Michaelis-Menten is bounded by a linear function) and DG transacylation 2DG~$\to$~TG~+~MG, which is modelled by a quadratic function. Depending on whether the DG concentration is relatively small or large, DG hydrolysis will dominate DG transacylation and vice versa \cite{elias2023}. The minor variations caused by the three different relative thicknesses (dotted, solid, and dashed lines) are the result of minor variations in the evolution of the DG concentrations, the diffusion of the substrates within the active region and the effects of DG transacylation.

\begin{figure}[htb]
	\centering
	\includegraphics[width=0.9\linewidth]{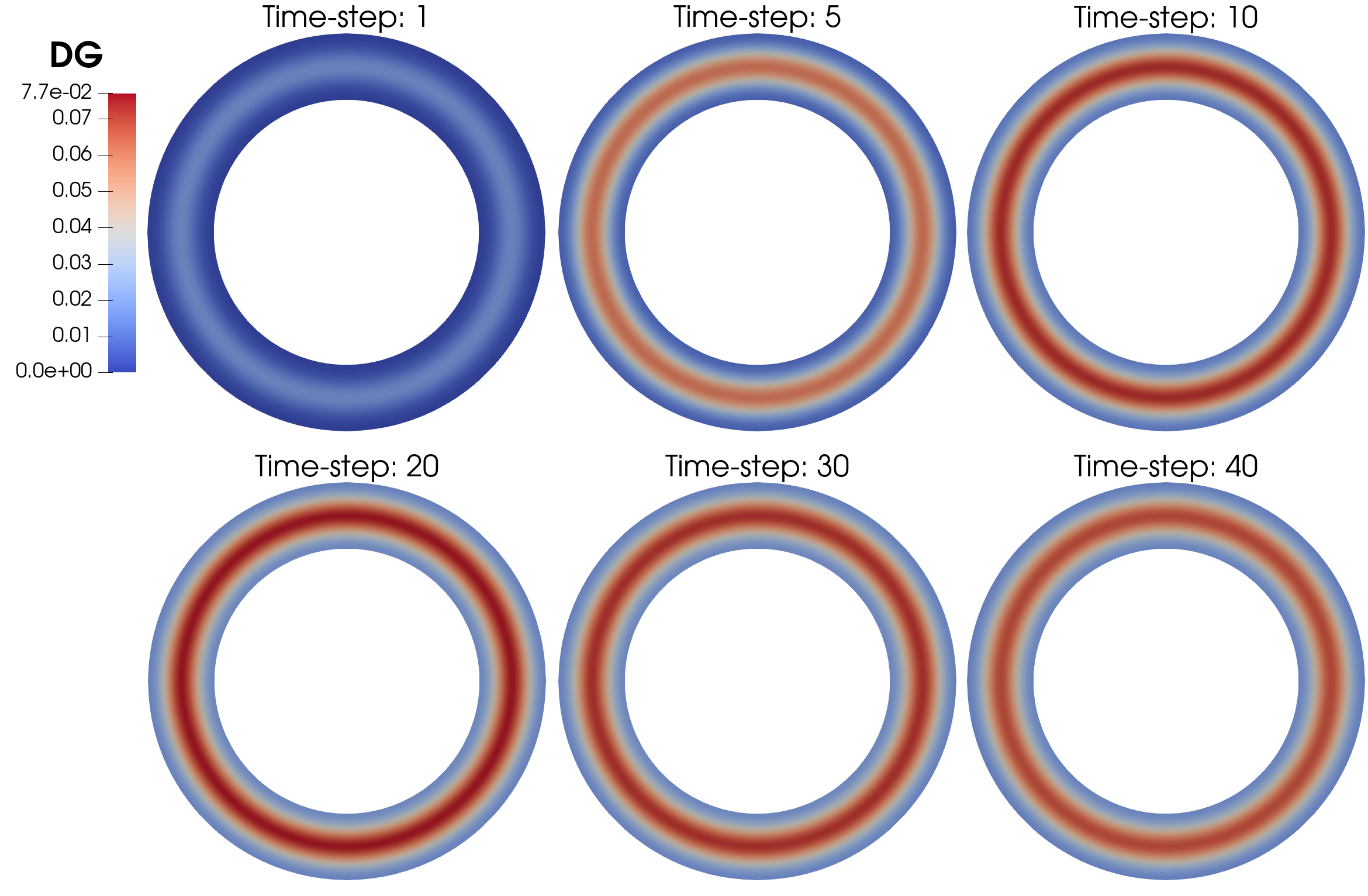}
	\caption{Snapshots of the DG concentration in case of the radially symmetric ATGL distribution. Observe that initially, DG concentration is zero but builds up from time step one proportional to the ATGL profile until reaching maximum levels and slowly decaying later}
	\label{fig:dg_homogeneous}
\end{figure}

The major differences between zero, two or four ATGL clusters correspond to the significant changes in DG concentrations caused by ATGL clustering. As illustration, we present snapshots of the numerical DG concentrations in the case where $\delta / R = 0.5$ for the homogeneous ATGL distribution in Figure~\ref{fig:dg_homogeneous} compared to two ATGL clusters Figure~\ref{fig:dg_2clusters} and four ATGL clusters in Figure~\ref{fig:dg_4clusters}, which shows that the case of two ATGL cluster yields the largest DG concentrations, causing the strongest slow down in Figure~\ref{fig:clustering} (green lines). We encounter here a situation, where increased DG transacylation due to larger DG concentrations slows down lipolysis by taking away substrate from DG hydrolysis, which is the faster process to produce MGs \cite{elias2023}.

As supplementary materials, the full animated versions can be found in the Github repository \href{https://github.com/reymartsalcedo/lipolysis}{https://github.com/reymartsalcedo/lipolysis}. In the three cases of ATGL distributions, we observe an initial build of DGs up to maximum levels, which follows the profile of the ATGL distribution, and then a slow decays of DGs as they are degraded into MGs by either DG hydrolysis or DG transacylation.

The numerical example of Figure~\ref{fig:clustering} demonstrates that clustering of ATGL and associated heterogeneities have significant impact on the efficiency of the lipolytic cascade. A more specific discussion of the effects of heterogeneities and/or the effects of transacylation as partial feedback process, which becomes increasingly important with accumulating DG concentrations, is still impeded by the lack of reliable kinetic parameters for ATGL, HSL, and MGL. However, there is a current work in progress, which aims to identify the kinetic parameters of ATGL from specially designed biochemical experiments using purified ATGL acting on a well-defined substrate of artificial lipid droplets.

\begin{figure}[htb]
	\centering
	\includegraphics[width=0.9\linewidth]{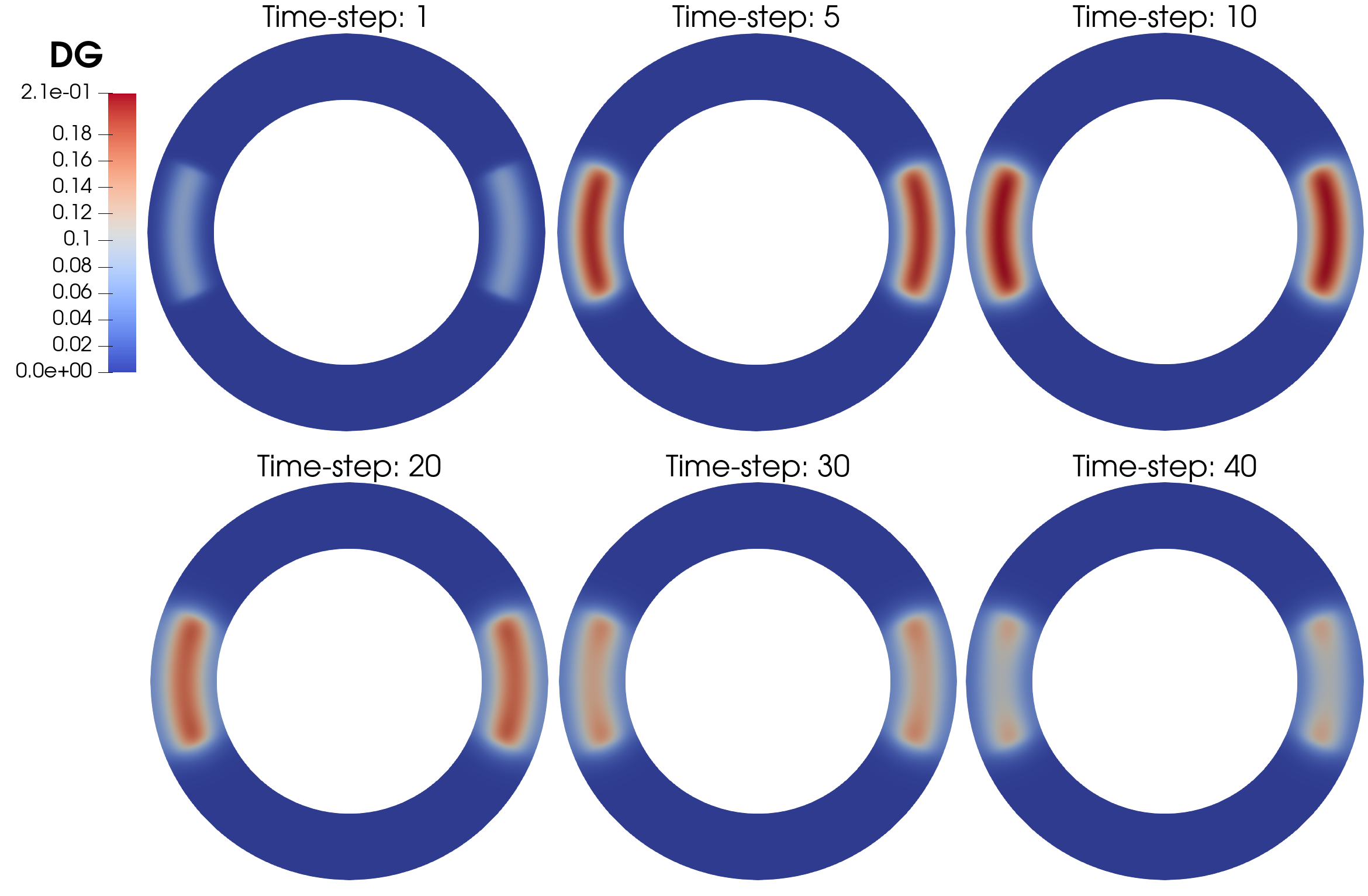}
	\caption{Snapshots of the DG concentration in case of two  ATGL clusters. Compared to the radially symmetric case, the larger ATGL concentrations localised in two cluster yield larger maximum levels of DG concentrations}
	\label{fig:dg_2clusters}
\end{figure}

\begin{figure}[htb]
	\centering
	\includegraphics[width=0.9\linewidth]{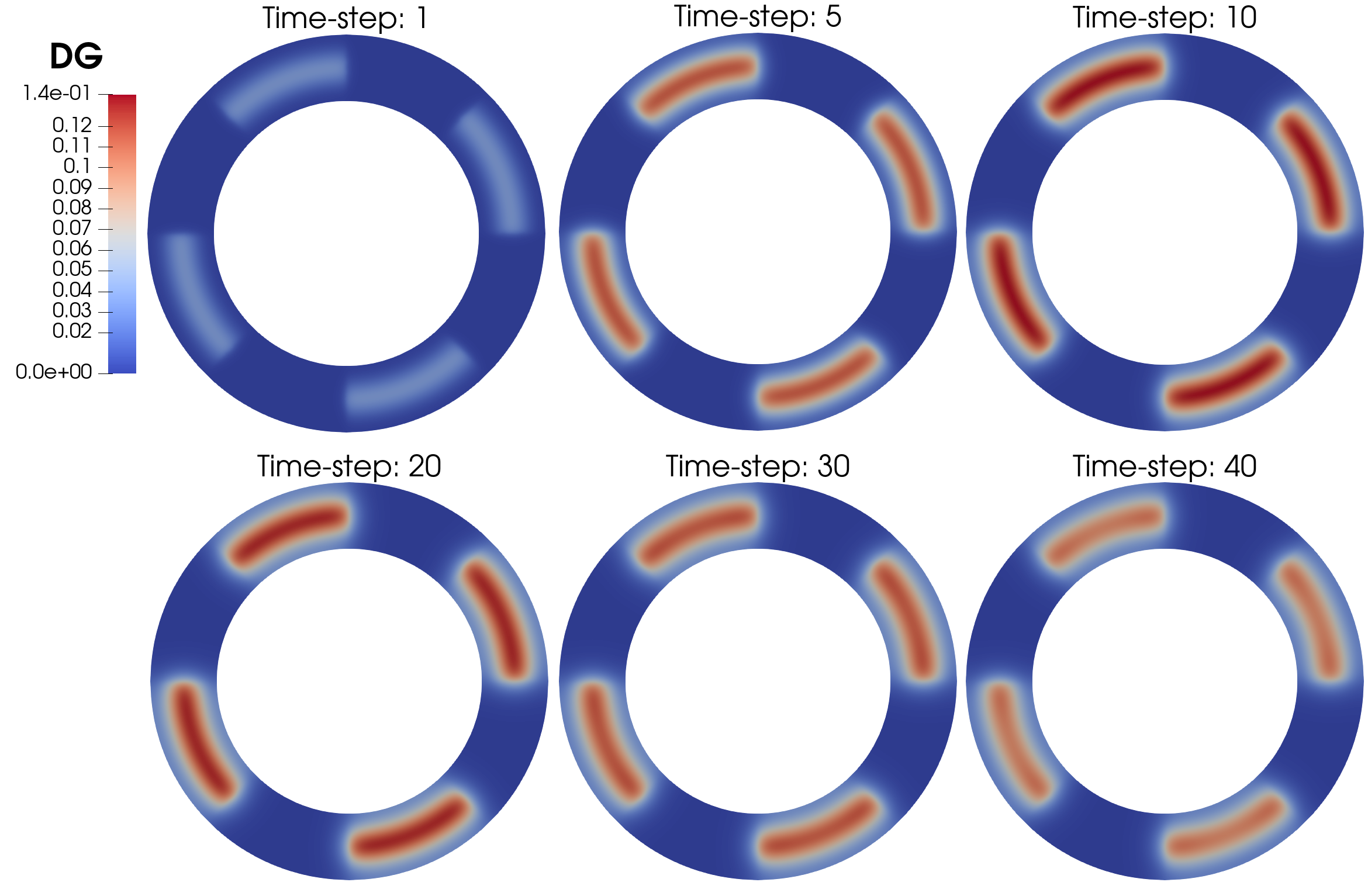}
	\caption{Snapshots of the DG concentration in case of four ATGL clusters. The maximum levels of DG concentration are in between the radially symmetric and the two cluster case}
	\label{fig:dg_4clusters}
\end{figure}

\begin{remark}[Machine Specifications]
	All numerical simulations are implemented using Python 3.7.10 for \verb|FEniCS| (version 2018.1.0) and Matlab R2021a for \verb|GeoPDEs| (version 3.2.2), Manjaro Linux 64-bit (Kernel: Linux 6.1.65-1-MANJARO), Quad Core Intel Xeon E3-1220 v5, 32 GB RAM. Visualisations are done in ParaView (\href{https://www.paraview.org/}{https://www.paraview.org/}).
\end{remark}

\section{Conclusions}\label{sec:conclusions}

In this paper, we presented a first PDE model for lipolysis on lipid droplets by postulating a thin active layer near the surface, where the enzymes can get into contact with the substrates. The PDE model in system \eqref{eqn:ParabolicModel} served as numerical testing case in preparation for a prototypical model of lipolysis in system \eqref{eqn:ParabolicModelRealistic} incorporating hydrolysis and DG transacylation. For the linear parabolic system \eqref{eqn:ParabolicModel}, the classical theory is applicable to prove the well-posedness of the problem, while for the corresponding elliptic system, difficulties arise from the noncoercivity of the bilinear form, which prevents the use of the classical Lax-Milgram Theorem. Instead, we adapt the arguments of \cite{kfellnerbtang2017} for which a fixed point argument was used to prove existence of solutions and the entropy method for the uniqueness. We also used the entropy method to prove the exponential convergence of solutions to the complex-balanced equilibrium.

We have provided numerical simulations to demonstrate the theoretical results. In particular, we implemented a classical finite element method using the computing package \verb|FEniCS| and an isogeometric method using \verb|GeoPDEs|. The main reason of using two different methods is to compare the numerical results. While the results from the \verb|FEniCS| implementation are already enough for a numerical test model, the results from \verb|GeoPDEs| improved the numerical errors in the sense that the computational domain is exactly represented and hence, the error only comes from the finite-dimensional approximation of the continuous problem. Finally, we explored, through numerical simulations, the effects of ATGL clustering in the active region, which shows a significant slow down of lipolysis.

The proof for the existence and uniqueness of discrete solutions is an open problem in this paper as well as for the error estimates. Nonetheless, we have numerically shown the linear and quadratic convergence of the discretisation error in the $H^{1}$- and $L^{2}$- norms. Additionally, we have shown the improved convergence for using higher order NURBS basis functions.

%\section*{Acknowledgments}
%Reymart Salcedo Lagunero would like to thank the co-authors from the University of the Bundeswehr Munich (and members of their institute) during his six-month research visit.

\section*{Declarations}
\textbf{Funding} Reymart Salcedo Lagunero was funded by the International Research Training Group IGDK 1754 Munich - Graz. This article was published in Open Access with financial support from the University of Graz.

\vspace{0.1in}
\noindent\textbf{Conflict of Interest} The authors have no relevant financial or non-financial interests to disclose.

\vspace{0.1in}
\noindent\textbf{Ethics Approval and Consent to Participate} Not applicable in this article.

\vspace{0.1in}
\noindent\textbf{Consent for Publication} Not applicable in this article.

\vspace{0.1in}
\noindent\textbf{Data and Materials Availability Statement} Not applicable in this article.

\vspace{0.1in}
\noindent\textbf{Code Availability} Sample minimal implementation code in \verb|FEniCS| can be found in the Github repository \href{https://github.com/reymartsalcedo/lipolysis}{https://github.com/reymartsalcedo/lipolysis}.

\vspace{0.1in}
\noindent\textbf{Author's Contributions} \textbf{Thomas Apel:} Supervision, Resources, Validation, Writing - Review and Editing \textbf{Klemens Fellner:} Conceptualisation, Methodology, Supervision, Formal Analysis, Resources, Writing - Review and Editing \textbf{Volker Kempf:} Software (\verb|FEniCS|), Validation, Writing - Review and Editing \textbf{Reymart Salcedo Lagunero:} Writing - Original Draft - Review and Editing, Formal Analysis, Visualisation, Data Curation, Software, Investigation \textbf{Philipp Zilk:} Software (\verb|GeoPDEs|), Validation, Writing - Review and Editing. All authors read and approved the final manuscript.

\begin{appendices}
\section{Relative Entropy and Entropy-Production Functionals}\label{sec:A1}

We note that since the parabolic system \eqref{eqn:ParabolicModel} describes the evolution of a complex-balanced chemical reaction network, it features the \textit{relative entropy functional} $\mathcal{E}$ given by:
	\begin{equation} \label{eqn:RelativeEntropy}
		\mathcal{E}(\mathbf{c}_{1}|\mathbf{c}_{2}) = \int_{\Omega_{1}} \qty( \frac{u_{1}^{2}}{u_{2}} + \frac{v_{1}^{2}}{v_{2}} ) \dd{x} + \int_{\Omega_{2}} \qty( \frac{w_{1}^{2}}{w_{2}} + \frac{z_{1}^{2}}{z_{2}} ) \dd{x},
	\end{equation}
	where $\mathbf{c}_{1} := (u_{1}, v_{1}, w_{1}, z_{1})$ and $\mathbf{c}_{2} := (u_{2}, v_{2}, w_{2}, z_{2})$ are any two solutions of the parabolic system with respect to possibly different initial data. For comparisons, see \cite{fellner2017,kfellnerbtang2017}. The evolution of the relative entropy functional computes according to the \textit{entropy-production functional} $\mathcal{D}$ as follows:
    \begin{lemma}[Entropy-Production Functional] \label{lem:EntropyProductionFunc}
        The entropy-production functional is given by
            \begin{align} \label{eqn:EntropyProduction}
                &\mathcal{D}(\mathbf{c}_{1}|\mathbf{c}_{2}) = - \dv{}{t} \mathcal{E}(\mathbf{c}_{1}|\mathbf{c}_{2}) \notag \\
                &= \int_{\Omega_{1}} 2d_{3} u_{2} \qty| \nabla\qty(\frac{u_{1}}{u_{2}}) |^{2} \dd{x} + \int_{\Omega_{1}} 2d_{2} v_{2} \qty| \nabla\qty(\frac{v_{1}}{v_{2}}) |^{2} \dd{x} \notag \\
                & + \int_{\Omega_{1}} \kappa v_{2} \qty| \frac{u_{1}}{u_{2}} - \frac{v_{1}}{v_{2}} |^{2} \dd{x} + \int_{\Omega_{2}} 2d_{3} w_{2} \qty| \nabla\qty(\frac{w_{1}}{w_{2}}) |^{2} \dd{x} \notag \\
                & + \int_{\Omega_{2}} 2d_{2} z_{2} \qty| \nabla\qty(\frac{z_{1}}{z_{2}}) |^{2} \dd{x} + \int_{\Omega_{2}} \rho w_{2} \qty| \frac{w_{1}}{w_{2}} - \frac{z_{1}}{z_{2}} |^{2} \dd{x} \notag \\
                & + \int_{\Gamma} \mu (w_{2} + u_{2}) \qty| \frac{u_{1}}{u_{2}} - \frac{w_{1}}{w_{2}} |^{2} \dd{s} + \int_{\Gamma} \nu (z_{2} + v_{2}) \qty| \frac{v_{1}}{v_{2}} - \frac{z_{1}}{z_{2}} |^{2} \dd{s}.
            \end{align}
    \end{lemma}

    \begin{proof}
    	Note that
    	\begin{align*}
    		\dv{}{t} & \mathcal{E}(\mathbf{c}_{1}|\mathbf{c}_{2}) = \int_{\Omega_{1}} \left[ 2 \frac{u_{1}}{u_{2}}(\partial_{t}u_{1}) - \frac{u_{1}^{2}}{u_{2}^{2}}(\partial_{t}u_{2}) + 2 \frac{v_{1}}{v_{2}}(\partial_{t}v_{1}) - \frac{v_{1}^{2}}{v_{2}^{2}}(\partial_{t}v_{2}) \right] \dd{x} \\
    		&+ \int_{\Omega_{2}} \left[ 2 \frac{w_{1}}{w_{2}}(\partial_{t}w_{1}) - \frac{w_{1}^{2}}{w_{2}^{2}}(\partial_{t}w_{2}) + 2 \frac{z_{1}}{z_{2}}(\partial_{t}z_{1}) - \frac{z_{1}^{2}}{z_{2}^{2}}(\partial_{t}z_{2}) \right] \dd{x}.
    	\end{align*}
    	Insert the time derivatives of $\mathbf{c}_{1}$ to obtain
    	\begin{align*}
    		\dv{}{t} & \mathcal{E}(\mathbf{c}_{1}|\mathbf{c}_{2}) = \int_{\Omega_{1}} \left[ 2 \frac{u_{1}}{u_{2}}(d_{3}\Delta u_{1} + \kappa v_{1}) - \frac{u_{1}^{2}}{u_{2}^{2}}(d_{3}\Delta u_{2} + \kappa v_{2}) \right] \dd{x} \\
    		& + \int_{\Omega_{1}} \left[ 2 \frac{v_{1}}{v_{2}}(d_{2} \Delta v_{1} - \kappa v_{1}) - \frac{v_{1}^{2}}{v_{2}^{2}}(d_{2} \Delta v_{2} - \kappa v_{2}) \right] \dd{x} \\
    		&+ \int_{\Omega_{2}} \left[ 2 \frac{w_{1}}{w_{2}}(d_{3}\Delta w_{1} - \rho w_{1}) - \frac{w_{1}^{2}}{w_{2}^{2}}(d_{3}\Delta w_{2} - \rho w_{2}) \right] \dd{x} \\
    		& + \int_{\Omega_{2}} \left[ 2 \frac{z_{1}}{z_{2}}(d_{2}\Delta z_{1} + \rho w_{1}) - \frac{z_{1}^{2}}{z_{2}^{2}}(d_{2}\Delta z_{2} + \rho w_{2}) \right] \dd{x} \\
    		&= \int_{\Omega_{1}} \left[ \qty( 2\frac{u_{1}}{u_{2}}d_{3}\Delta u_{1} - \frac{u_{1}^{2}}{u_{2}^{2}}d_{3}\Delta u_{2} )  + \qty( 2\frac{u_{1}}{u_{2}}(\kappa v_{1}) - \frac{u_{1}^{2}}{u_{2}^{2}} (\kappa v_{2})) \right] \dd{x} \\
    		& + \int_{\Omega_{1}} \left[ \qty( 2\frac{v_{1}}{v_{2}}d_{2}\Delta v_{1} -  \frac{v_{1}^{2}}{v_{2}^{2}}d_{2}\Delta v_{2} )  + \qty( 2\frac{v_{1}}{v_{2}}(-\kappa v_{1}) - \frac{v_{1}^{2}}{v_{2}^{2}} (-\kappa v_{2})) \right] \dd{x} \\
    		&+ \int_{\Omega_{2}} \left[ \qty( 2\frac{w_{1}}{w_{2}}d_{3}\Delta w_{1} - \frac{w_{1}^{2}}{w_{2}^{2}}d_{3}\Delta w_{2} )  + \qty( 2\frac{w_{1}}{w_{2}}(-\rho  w_{1}) - \frac{w_{1}^{2}}{w_{2}^{2}} (-\rho  w_{2})) \right] \dd{x} \\
    		& + \int_{\Omega_{2}} \left[ \qty( 2\frac{z_{1}}{z_{2}}d_{2}\Delta z_{1} - \frac{z_{1}^{2}}{z_{2}^{2}}d_{2}\Delta z_{2} )  + \qty( 2\frac{z_{1}}{z_{2}}(\rho  w_{1}) - \frac{z_{1}^{2}}{z_{2}^{2}} (\rho  w_{2})) \right] \dd{x}.
    	\end{align*}
    	In the following, we use the identities:
    	\[ \nabla \qty( \frac{u_{1}^{2}}{u_{2}^{2}} ) = \nabla \qty(\qty( \frac{u_{1}}{u_{2}} )^{2}) = 2 \frac{u_{1}}{u_{2}} \nabla\qty( \frac{u_{1}}{u_{2}} ) \textrm{ and } \nabla\qty( \frac{u_{1}}{u_{2}} ) = \frac{u_{2}\nabla u_{1}  - u_{1} \nabla u_{2}}{u_{2}^{2}}. \]
    	Let us consider each of the terms. Using integration by parts, we get
    	\begin{align*}
    		&\int_{\Omega_{1}} \qty( 2\frac{u_{1}}{u_{2}}d_{3}\Delta u_{1} -  \frac{u_{1}^{2}}{u_{2}^{2}}d_{3}\Delta u_{2} ) \dd{x} \\
    		&= -\int_{\Omega_{1}} 2d_{3} \nabla u_{1} \cdot \nabla\qty(\frac{u_{1}}{u_{2}}) \dd{x} + \int_{\Gamma} 2\frac{u_{1}}{u_{2}} (d_{3} \partial_{\eta} u_{1}) \dd{s} \\
    		&\qquad + \int_{\Omega_{1}} d_{3} \nabla u_{2} \cdot \nabla \qty( \frac{u_{1}^{2}}{u_{2}^{2}} ) \dd{x} - \int_{\Gamma} \frac{u_{1}^{2}}{u_{2}^{2}} (d_{3}\partial_{\eta} u_{2}) \dd{s} \\
    		&= -\int_{\Omega_{1}} 2d_{3} \nabla u_{1} \cdot \nabla\qty(\frac{u_{1}}{u_{2}}) \dd{x} + \int_{\Gamma} 2\frac{u_{1}}{u_{2}} \mu(w_{1} - u_{1}) \dd{s} \\
    		&\qquad + \int_{\Omega_{1}} 2d_{3} \frac{u_{1}}{u_{2}} \nabla u_{2} \cdot \nabla \qty( \frac{u_{1}}{u_{2}} ) \dd{x} - \int_{\Gamma} \frac{u_{1}^{2}}{u_{2}^{2}} \mu(w_{2} - u_{2}) \dd{s} \\
    		&= -2d_{3} \int_{\Omega_{1}} \qty( \nabla u_{1}  - \frac{u_{1}}{u_{2}} \nabla u_{2} ) \cdot \nabla\qty(\frac{u_{1}}{u_{2}}) \dd{x} \\
    		&\qquad + \int_{\Gamma} \qty( 2\frac{u_{1}}{u_{2}} \mu(w_{1} - u_{1}) - \frac{u_{1}^{2}}{u_{2}^{2}} \mu(w_{2} - u_{2}) ) \dd{s} \\
    		&= -2d_{3} \int_{\Omega_{1}} u_{2}\qty( \frac{u_{2}\nabla u_{1}  - u_{1} \nabla u_{2}}{u_{2}^{2}} ) \cdot \nabla\qty(\frac{u_{1}}{u_{2}}) \dd{x} \\
    		&\qquad + \int_{\Gamma} \qty( 2\frac{u_{1}}{u_{2}} \mu(w_{1} - u_{1}) - \frac{u_{1}^{2}}{u_{2}^{2}} \mu(w_{2} - u_{2}) ) \dd{s} \\
    		&= -2d_{3} \int_{\Omega_{1}} u_{2}\nabla\qty(\frac{u_{1}}{u_{2}}) \cdot \nabla\qty(\frac{u_{1}}{u_{2}}) \dd{x} \\
    		&\qquad + \int_{\Gamma} \qty( 2\frac{u_{1}}{u_{2}} \mu(w_{1} - u_{1}) - \frac{u_{1}^{2}}{u_{2}^{2}} \mu(w_{2} - u_{2}) ) \dd{s} \\
    		&= -2d_{3} \int_{\Omega_{1}} u_{2}\qty|\nabla\qty(\frac{u_{1}}{u_{2}})|^{2} \dd{x}  + \int_{\Gamma} \qty( 2\frac{u_{1}}{u_{2}} \mu(w_{1} - u_{1}) - \frac{u_{1}^{2}}{u_{2}^{2}} \mu(w_{2} - u_{2}) ) \dd{s}.
    	\end{align*}
    	Analogously, we therefore get for $u$ and $v$:
    	\begin{align*}
    		&\int_{\Omega_{1}} \qty( 2\frac{u_{1}}{u_{2}}d_{3}\Delta u_{1} -  \frac{u_{1}^{2}}{u_{2}^{2}}d_{3}\Delta u_{2} ) \dd{x} \\
    		&= -2d_{3} \int_{\Omega_{1}} u_{2}\qty|\nabla\qty(\frac{u_{1}}{u_{2}})|^{2} \dd{x} + \int_{\Gamma} \qty( 2\frac{u_{1}}{u_{2}} \mu(w_{1} - u_{1}) - \frac{u_{1}^{2}}{u_{2}^{2}} \mu(w_{2} - u_{2}) ) \dd{s}, \\
    		&\int_{\Omega_{1}} \qty( 2\frac{v_{1}}{v_{2}}d_{2}\Delta v_{1} -  \frac{v_{1}^{2}}{v_{2}^{2}}d_{2}\Delta v_{2} ) \dd{x} \\
    		&= -2d_{2} \int_{\Omega_{1}} v_{2}\qty|\nabla\qty(\frac{v_{1}}{v_{2}})|^{2} \dd{x} + \int_{\Gamma} \qty( 2\frac{v_{1}}{v_{2}} \nu(z_{1} - v_{1}) - \frac{v_{1}^{2}}{v_{2}^{2}} \nu(z_{2} - v_{2}) ) \dd{s},
    	\end{align*}
    	and for $w$ and $z$:
    	\begin{align*}
    		&\int_{\Omega_{2}} \qty( 2\frac{w_{1}}{w_{2}}d_{3}\Delta w_{1} -  \frac{w_{1}^{2}}{w_{2}^{2}}d_{3}\Delta w_{2} ) \dd{x} \\
    		&= -2d_{3} \int_{\Omega_{2}} w_{2}\qty|\nabla\qty(\frac{w_{1}}{w_{2}})|^{2} \dd{x} + \int_{\Gamma} \qty( -2\frac{w_{1}}{w_{2}} \mu(w_{1} - u_{1}) + \frac{w_{1}^{2}}{w_{2}^{2}} \mu(w_{2} - u_{2}) ) \dd{s}, \\
    		&\int_{\Omega_{2}} \qty( 2\frac{z_{1}}{z_{2}}d_{2}\Delta z_{1} -  \frac{z_{1}^{2}}{z_{2}^{2}}d_{2}\Delta z_{2} ) \dd{x} \\
    		&= -2d_{2} \int_{\Omega_{2}} z_{2}\qty|\nabla\qty(\frac{z_{1}}{z_{2}})|^{2} \dd{x} + \int_{\Gamma} \qty( -2\frac{z_{1}}{z_{2}} \nu(z_{1} - v_{1}) + \frac{z_{1}^{2}}{z_{2}^{2}} \nu(z_{2} - v_{2}) ) \dd{s}.
    	\end{align*}
    	Rewrite the following terms as follows:
    	\begin{align*}
    		&2\frac{u_{1}}{u_{2}}(\kappa v_{1}) - \frac{u_{1}^{2}}{u_{2}^{2}} (\kappa v_{2}) =  2\frac{u_{1}}{u_{2}}\frac{v_{1}}{v_{2}}(\kappa v_{2}) - \frac{u_{1}^{2}}{u_{2}^{2}} (\kappa v_{2}), \\
    		&2\frac{v_{1}}{v_{2}}(-\kappa v_{1}) - \frac{v_{1}^{2}}{v_{2}^{2}} (-\kappa v_{2}) =  -2\frac{v_{1}^{2}}{v_{2}^{2}}(\kappa v_{2}) + \frac{v_{1}^{2}}{v_{2}^{2}} (\kappa v_{2})), \\
    		&2\frac{w_{1}}{w_{2}}(-\rho  w_{1}) - \frac{w_{1}^{2}}{w_{2}^{2}} (-\rho  w_{2}) = -2\frac{w_{1}^{2}}{w_{2}^{2}}(\rho  w_{2}) + \frac{w_{1}^{2}}{w_{2}^{2}} (\rho  w_{2}), \\
    		&2\frac{z_{1}}{z_{2}}(\rho  w_{1}) - \frac{z_{1}^{2}}{z_{2}^{2}} (\rho  w_{2}) = 2\frac{z_{1}}{z_{2}}\frac{w_{1}}{w_{2}}(\rho  w_{2}) - \frac{z_{1}^{2}}{z_{2}^{2}} (\rho  w_{2}).
    	\end{align*}
    	Hence, we have the following simplifications:
    	\begin{align*}
    		\int_{\Omega_{1}} & \left[ \qty(2\frac{u_{1}}{u_{2}}(\kappa v_{1}) - \frac{u_{1}^{2}}{u_{2}^{2}} (\kappa v_{2})) + \qty( 2\frac{v_{1}}{v_{2}}(-\kappa v_{1}) - \frac{v_{1}^{2}}{v_{2}^{2}} (-\kappa v_{2}) ) \right] \dd{x} \\
    		&= \int_{\Omega_{1}} \left[ 2\frac{u_{1}}{u_{2}}\frac{v_{1}}{v_{2}}(\kappa v_{2}) - \frac{u_{1}^{2}}{u_{2}^{2}} (\kappa v_{2}) -2\frac{v_{1}^{2}}{v_{2}^{2}}(\kappa v_{2}) + \frac{v_{1}^{2}}{v_{2}^{2}} (\kappa v_{2}) \right] \dd{x} \\
    		&= \int_{\Omega_{1}} \left[ 2\frac{u_{1}}{u_{2}}\frac{v_{1}}{v_{2}}(\kappa v_{2}) - \frac{u_{1}^{2}}{u_{2}^{2}} (\kappa v_{2}) -\frac{v_{1}^{2}}{v_{2}^{2}}(\kappa v_{2}) \right] \dd{x} \\
    		&= -\int_{\Omega_{1}} \kappa v_{2} \qty| \frac{u_{1}}{u_{2}} - \frac{v_{1}}{v_{2}}|^{2} \dd{x}.
    	\end{align*}
    	Analogously, and hence,
    	\begin{align*}
    		\int_{\Omega_{1}} & \left[ \qty(2\frac{u_{1}}{u_{2}}(\kappa v_{1}) - \frac{u_{1}^{2}}{u_{2}^{2}} (\kappa v_{2})) + \qty( 2\frac{v_{1}}{v_{2}}(-\kappa v_{1}) - \frac{v_{1}^{2}}{v_{2}^{2}} (-\kappa v_{2}) ) \right] \dd{x} \\
    		&= -\int_{\Omega_{1}} \kappa v_{2} \qty| \frac{u_{1}}{u_{2}} - \frac{v_{1}}{v_{2}}|^{2} \dd{x}, \\
    		\int_{\Omega_{2}} & \left[ 2\frac{w_{1}}{w_{2}}(-\rho  w_{1}) - \frac{w_{1}^{2}}{w_{2}^{2}} (-\rho  w_{2}) + 2\frac{z_{1}}{z_{2}}(\rho  w_{1}) - \frac{z_{1}^{2}}{z_{2}^{2}} (\rho  w_{2}) \right] \dd{x} \\
    		&= - \int_{\Omega_{2}} \rho  w_{2} \qty| \frac{w_{1}}{w_{2}} - \frac{z_{1}}{z_{2}} |^{2} \dd{x}.
    	\end{align*}
    	The boundary integrals can also be simplified in a similar manner:
    	\begin{align*}
    		\int_{\Gamma} & \qty( 2\frac{u_{1}}{u_{2}} \mu(w_{1} - u_{1}) - \frac{u_{1}^{2}}{u_{2}^{2}} \mu(w_{2} - u_{2}) -2\frac{w_{1}}{w_{2}} \mu(w_{1} - u_{1}) + \frac{w_{1}^{2}}{w_{2}^{2}} \mu(w_{2} - u_{2}) ) \dd{s} \\
    		&= - \int_{\Gamma} \mu(w_{2} + u_{2}) \qty| \frac{u_{1}}{u_{2}} - \frac{w_{1}}{w_{2}} |^{2} \dd{s}, \\
    		\int_{\Gamma} & \qty( 2\frac{v_{1}}{v_{2}} \nu(z_{1} - v_{1}) - \frac{v_{1}^{2}}{v_{2}^{2}} \nu(z_{2} - v_{2}) -2\frac{z_{1}}{z_{2}} \nu(z_{1} - v_{1}) + \frac{z_{1}^{2}}{z_{2}^{2}} \nu(z_{2} - v_{2}) ) \dd{s} \\
    		&= - \int_{\Gamma} \nu(z_{2} + v_{2}) \qty| \frac{v_{1}}{v_{2}} - \frac{z_{1}}{z_{2}} |^{2} \dd{s}.
    	\end{align*}
    	Combining all previous computations, we obtain the entropy-production functional
    	\begin{equation*}
    		\begin{split}
    			&\mathcal{D}(\mathbf{c}_{1}|\mathbf{c}_{2}) = - \dv{}{t} \mathcal{E}(\mathbf{c}_{1}|\mathbf{c}_{2}) \\
    			&= 2d_{3} \int_{\Omega_{1}} u_{2} \qty| \nabla\qty(\frac{u_{1}}{u_{2}}) |^{2} \dd{x} + 2d_{2} \int_{\Omega_{1}} v_{2} \qty| \nabla\qty(\frac{v_{1}}{v_{2}}) |^{2} \dd{x} +\kappa \int_{\Omega_{1}} v_{2} \qty| \frac{u_{1}}{u_{2}} - \frac{v_{1}}{v_{2}} |^{2} \dd{x} \\
    			& + 2d_{3} \int_{\Omega_{2}} w_{2} \qty| \nabla\qty(\frac{w_{1}}{w_{2}}) |^{2} \dd{x} + 2d_{2} \int_{\Omega_{2}} z_{2} \qty| \nabla\qty(\frac{z_{1}}{z_{2}}) |^{2} \dd{x} + \rho \int_{\Omega_{2}} w_{2} \qty| \frac{w_{1}}{w_{2}} - \frac{z_{1}}{z_{2}} |^{2} \dd{x} \\
    			& + \mu\int_{\Gamma} (w_{2} + u_{2}) \qty| \frac{u_{1}}{u_{2}} - \frac{w_{1}}{w_{2}} |^{2} \dd{s} + \nu\int_{\Gamma} (z_{2} + v_{2}) \qty| \frac{v_{1}}{v_{2}} - \frac{z_{1}}{z_{2}} |^{2} \dd{s},
    		\end{split}
    	\end{equation*}
    	as desired.
    \end{proof}

\end{appendices}

\bibliography{sn-bibliography}

\end{document}